\documentclass[a4paper]{amsart}
\usepackage{url}
\usepackage{amsmath}
\usepackage[makeroom]{cancel}
\usepackage{graphicx}
\usepackage[all]{xy}
\usepackage[mathscr]{eucal}
\usepackage{amsmath,amssymb,amsfonts}
\usepackage{MnSymbol}

\usepackage{mathrsfs,latexsym,amsthm,enumerate}
\usepackage{amscd}
\newtheorem{theorem}{Theorem}[section]
\newtheorem{lemma}[theorem]{Lemma}
\newtheorem{corollary}[theorem]{Corollary}
\newtheorem{example}[theorem]{Example}
\newtheorem{examples}[theorem]{Examples}
\newtheorem{proposition}[theorem]{Proposition}
\newtheorem{remark}[theorem]{Remark}

\usepackage{stackengine}
\usepackage{xcolor}
\usepackage[all]{xy}

\newcommand{\dom}{\mathbf{d}}
\newcommand{\ran}{\mathbf{r}}

\title{Non-commutative Stone duality}

\author{Mark V. Lawson}
\address{Mark V. Lawson, Department of Mathematics
and the
Maxwell Institute for Mathematical Sciences,
Heriot-Watt University,
Riccarton,
Edinburgh EH14 4AS,
UNITED KINGDOM}
\email{m.v.lawson@hw.ac.uk}

\begin{document}
\dedicatory{This paper is dedicated to the memory of my friend and colleague Iain Currie}

\begin{abstract}
We show explicitly that Boolean inverse semigroups are in duality with what we term Boolean groupoids.
This generalizes classical Stone duality, which we refer to as commutative Stone duality,
between generalized Boolean algebras and locally compact Hausdorff $0$-dimensional spaces.
\end{abstract}
\maketitle

\section{Introduction}

The theory of what we term {\em non-commutative Stone duality} grew out of the work of a number of authors 
\cite{K, Renault, Paterson},
\cite{Kell1, Kell2, Lenz} 
and 
\cite{LMS, Resende}.
In this paper, I shall concentrate on one aspect of that duality:
namely, how Boolean inverse semigroups are in duality with a class of \'etale groupoids 
called Boolean groupoids.
Specifically, we shall prove the following theorem and discuss some special cases (proved as Theorem~\ref{them:non-com-stone}):

\begin{theorem}[Non-commutative Stone duality] 
The category of Boolean inverse semigroups and callitic morphisms
is dually equivalent to the category of Boolean groupoids and coherent, continuous, covering functors.
\end{theorem}

This theorem generalizes what you will find in \cite{Lawson3} 
since we shall not assume that our topological groupoids are Hausdorff.
Although this theorem can be gleaned from our papers \cite{Lawson3, Lawson5, Lawson2016, Lawson2017, LL},
what I describe here has not been reported in one place before.

You might think that this duality is of merely parochial interest.
It isn't.
The work of Matui \cite{Matui12, Matui13} deals with \'etale groupoids of just the kind that figure in our duality theorem.
In addition, Matui refers constantly to `compact open $\mathcal{G}$-sets'.
These are precisely what we call `compact-open local bisections'
and are elements of the inverse semigroup associated with the \'etale groupoid.
Thus, even though Matui is not explicitly interested in Boolean inverse semigroups,
they are there implicitly.

Sections 3 to 7 are devoted to proving the above theorem
whereas in Section 2 classical Stone duality is decribed since this sets the scene for our generalization.
Thus the reader familiar with classical Stone duality can start reading at Section 3.
In Section 3, we describe the theory of Boolean inverse semigroups needed:
these should be regarded as the non-commutative generalized Boolean algebras.
In Section 4, we describe Boolean groupoids; these are a class of \'etale topological groupoids 
and play the role of `non-commutative topological spaces'.  
In Section 5, we show how to pass from Boolean groupoids to Boolean inverse semigroups,
where a key role is played by the compact-open local bisections of a Boolean groupoid.
This is a comparatively easy construction.
In Section 6, we show how to pass from Boolean inverse semigroups to Boolean groupoids by using ultrafilters.
This is quite technical.
In Section 7, we show, amongst other things, that the above two constructions are the inverse of each other.
We also bring on board morphisms and establish our main duality theorem which generalizes  the classical theory
described in Section 2.
Section 8 is dedicated to special cases; for example, those Boolean inverse semigroups
that have all binary meets correspond to Hausdorff Boolean groupoids.
There is one new application, in Section 9, which shows how our theory can be used to give an account 
of {\em unitization} first described in \cite[Definition 6.6.1]{W}.
In Section~10, I shall show how the above theory can be derived within a more general framework
using pseudogroups and arbitrary \'etale groupoids.

I shall assume you are conversant with the theory of inverse semigroups \cite{Lawson1998};
I will simply remind the reader of notation and terminology as we go along.
Observe that whenever I refer to an order on an inverse semigroup, 
I mean the {\em natural partial order} defined on every inverse semigroup.
I shall develop the theory of Boolean inverse semigroups from scratch
but this is no subsitute for a proper development of that theory as in \cite{W}.

We shall also need some terminology from the theory of posets throughout.
Let $P$ be a poset with a minimum (or bottom) element denoted by zero $0$.
In this context, singleton sets such as $\{x\}$ will be written simply as $x$.
If $X \subseteq P$, define
$$X^{\downarrow} = \{y \in P \colon y \leq x \text{ for some } x \in X \}
\text{ and } 
X^{\uparrow} = \{y \in P \colon x \leq y \text{ for some } x \in X\}.$$
If $X = X^{\downarrow}$ we say that $X$ is an {\em order-ideal}.
An order-ideal of the form $a^{\downarrow}$ is said to be {\em principal}. 
If for any $x,y \in X$ there exists $z \in X$ such that $z \leq x,y$,
we say that $X$ is {\em downwardly directed}.
If $X = X^{\uparrow}$ we say that $X$ is {\em upwardly closed}.

Let $X$ and $Y$ be posets.
A function $\theta \colon X \rightarrow Y$ is said to be {\em order-preserving}
if $x_{1} \leq x_{2}$ in $X$ implies that $\theta (x_{1}) \leq \theta (x_{2})$.
An {\em order-isomorphism} is a bijection which is order-preserving and whose inverse is 
order-preserving.
Let $F$ be a {\em non-empty}\footnote{Throughout this paper, filters will be assumed non-empty.} subset of $X$.
We say that it is a {\em filter} if it is downwardly directed and upwardly closed.
We say that it is {\em proper} if it does not contain the zero.
A maximal proper filter is called an {\em ultrafilter}.
Ultrafilters play an important role in this paper.

A {\em meet semilattice} is a poset in which each pair of elements has a greatest lower bound (or meet);
we write $x \wedge y$ for the  meet of $x$ and $y$.
The following was proved as \cite[Lemma 12.3]{Exel} and is very useful.

\begin{lemma}\label{lem:exel} Let $P$ be a meet semilattice with bottom element $0$,
and let $A$ be a proper filter in $P$.
Then $A$ is an ultrafilter if and only if the following holds:
if $x \wedge y \neq 0$ for all $x \in A$ then $y \in A$.
\end{lemma}

Let $Y$ be a meet semilattice with bottom element $0$.
Let $X \subseteq Y$ be any subset.
Define $X^{\wedge}$ to be the set of all finite meets of elements of $X$;
we say that $X$ has the {\em finite intersection property} if $0 \notin X^{\wedge}$.

Our references for topology are \cite{Simmons} and \cite{Willard}.
The following contains the main results we need.
For the proof of (1), see \cite[Section 23, Theorem A]{Simmons};
for (2), see \cite[Section 26, Theorem D]{Simmons};
for (3), see \cite[Section 21, Theorem A]{Simmons};
for (4), see \cite[Section 21, Theorem D]{Simmons};
for (5), see \cite[Theorem 18.2]{Willard}

\begin{lemma}\label{lem:topology-needed}\mbox{}
\begin{enumerate}

\item The product of any non-empty set of compact spaces is compact.

\item Every compact subspace of a Hausdorff space is closed.

\item Every closed subspace of a compact space is compact.

\item A topological space is compact if and only if every set of closed sets with the finite intersection property has a non-empty intersection

\item A Hausdorff space is locally compact if for every point $x$ there is a compact set $U$ such that $x \in U^{\circ}$, the union of all the open sets contained in $U$.

\end{enumerate}
\end{lemma}

Let $\mathscr{B}$ be any set of subsets of a set $X$.
We say that $\mathscr{B}$ is a {\em base} if the union of all elements in $\mathscr{B}$ is $X$
and if $x \in B \cap C$ where $B,C \in \mathscr{B}$ then there is a $D \in \mathscr{B}$ such that
$x \in D \subseteq B \cap C$.
Given a base $\mathscr{B}$, we may define as a topology all those sets which are unions of subsets of $\mathscr{B}$;
where the empty set is the union of the empty subset of $\mathscr{B}$.
See \cite[Theorem 5.3]{Willard}.
A space with a countable base is said to be {\em second countable}.
A topological space is {\em discrete} if every susbet is open.
A subset of a topological space is said to be {\em clopen}
if it is both open and closed.
A topological space is said to be {\em $0$-dimensional} if it has a base of clopen sets.
If $X$ is a topological space we denote its set of open subsets by $\Omega (X)$.\\

\noindent
{\bf Acknowledgements. }Some of the work for this paper was carried out at LaBRI, Universit\'e de Bordeaux during April 2018 whilst visiting David Janin.
I am also grateful to Phil Scott for alerting me to typos.
None of this work would have been possible without my collaboration with Daniel Lenz, 
and some very timely conversations with Pedro Resende.

\section{Classical Stone duality}

In this section, we shall describe classical Stone duality which relates generalized Boolean algebras
to locally compact Hausdorff $0$-dimensional spaces.
There are no new results in this section --- the theory is classical --- but it provides essential motivation for what we do
when we come to study Boolean inverse semigroups.
In Section 2.1, we shall recall the definition and first properties of Boolean algebras;
in Section 2.2, we shall describe the structure of finite Boolean algebras;
in Section 2.3, we describe classical Stone duality which deals with the relationship between Boolean algebras and compact Boolean spaces
---
this is the version of Stone duality that you will find well represented in the textbooks;
finally, in Section 2.4, we describe the extension of classical Stone duality to generalized Boolean algebras and locally compact Boolean spaces.

\subsection{Boolean algebras}
Boolean algebras may not rank highly in the pantheon of algebraic structures but they are, in fact, both mathematically interesting and remarkably useful; for example, 
they form the basis of measure theory.

Formally, a {\em Boolean algebra} is a $6$-tuple 
$(B,\vee,\wedge, \,' \, \, , 0,1)$
consisting of a set $B$, two binary operations $\vee$, called {\em join}, and $\wedge$, called {\em meet},
one unary operation $a \mapsto a'$, called {\em complementation},
and two constants $0$ and $1$
satisfying the following ten axioms:
$$
\begin{array}{ll}

{\rm (B1) }\, (x \vee y) \vee z = x \vee (y \vee z). &{\rm (B6) }\, x \wedge 1 = x.\\

{\rm (B2) }\, x \vee y = y \vee x. &{\rm (B7) }\, x \wedge (y \vee z) = (x \wedge y) \vee (x \wedge z).\\

{\rm (B3) }\, x \vee 0 = x. &{\rm (B8) }\, x \vee (y \wedge z) = (x \vee y) \wedge (x \vee z).\\

{\rm (B4) }\, (x \wedge y) \wedge z = x \wedge (y \wedge z). &{\rm (B9) }\, x \vee x' = 1.\\ 

{\rm (B5) }\, x \wedge y = y \wedge x. &{\rm (B10) }\, x \wedge x' = 0.\\

\end{array}
$$
The element $0$ is called the {\em bottom}  and $1$ is called the {\em top}.
A function $\theta \colon B \rightarrow C$ between two Boolean algebras
is called a {\em homomorphism} if it
preserves the two binary operations, the unary operation and maps constants to corresponding constants.
Boolean algebras and their homomorphisms form a category.
Observe that for a function to be a homomorphism of Boolean algebras it is enough that it preserves the constants
and maps meets and joins to meets and joins, respectively;
this is because the complement of $x$ is the unique element $y$ such that $1 = x \vee y$ and $0 = x \wedge y$.
The following lemma summarizes some important properties of Boolean algebras that readily follow from these axioms.

\begin{lemma}\label{lem:basic-properties} In a Boolean algebra $B$, the following hold for all $x,y \in B$:
$$\begin{array}{ll}
(1)\, x \vee x = x \text{ and } x \wedge x = x. &(5)\, x \vee y = x \vee (y \wedge x').\\
(2)\, x \wedge 0 = 0.                           &(6)\, x'' = x.\\
(3)\, 1 \vee x = 1.                             &(7)\, (x \vee y)' = x' \wedge y'.\\
(4)\, x = x \vee (x \wedge y).                  &(8)\, (x \wedge y)' = x' \vee y'. \\
\end{array}
$$
\end{lemma}

The theory of Boolean algebras is described in an elementary fashion in \cite{GH}
and from a more advanced standpoint in \cite{Koppelberg} and \cite{Sikorski}.
The first two chapters of \cite{Johnstone} approach the subject of Stone duality from the perspective of frame theory.

\begin{example} \mbox{}
{\em
\begin{enumerate}

\item The basic example of a Boolean algebra is the 
{\em power set Boolean algebra} which consists of the set of all subsets, $\mathsf{P}(X)$, of the non-empty set $X$
with the operations $\cup$, $\cap$ and $\overline{A} = X \setminus A$ and the two constants $\varnothing$, $X$.

\item We denote by $\mathbb{B}$ the unique two-element Boolean algebra.

\item Let $A$ be any finite non-empty set.
We call $A$ an {\em alphabet}.
Denote the free monoid on $A$ by $A^{\ast}$;
observe that $A^{\ast}$ is countably infinite.
Elements of the free monoid are called {\em strings}.
The {\em length} of a string $x$ is denoted by $|x|$.
The empty string is denoted by $\varepsilon$.
A subset of $A^{\ast}$ is called a {\em language} over $A$.
Recall that a language over an alphabet $A$ is said to be {\em recognizable} if there is a finite-state
automaton that accepts it. By Kleene's theorem \cite{Lawson2003}, the set of recognizable languages over $A$ is equal to the set of regular languages over $A$.
Denote the set of regular languages over $A$ by $\mbox{\rm Reg}(A)$.
This set is a Boolean algebra (with extra operations).
This Boolean structure can be exploited to provide a sophisticated way of studying families of regular languages \cite{Pipp, G}.

\end{enumerate}
}
\end{example}

Boolean algebras have their roots in the work of George Boole, though the definition of Boolean algebras seems to have been inspired by his work rather than originating there \cite{Burris, TH}.
Until the 1930s, research on Boolean algebras was essentially about axiomatics.
However, with Stone's paper \cite{Stone1936}, stability in the definition of Boolean algebras emerges because he showed that each Boolean algebra
could be regarded as a (unital) ring in which each element was idempotent;
rings such as this are called {\em Boolean rings}.
In the language of category theory, his result shows that the category of unital Boolean algebras is isomorphic to the category of unital Boolean rings
(where I have stressed the fact that the Boolean algebra has a top element and the ring has an identity).
The following result describes how the correspondence between Boolean algebras and Boolean rings works at the algebraic level.

\begin{theorem}\label{them:bool-ring} \mbox{}
\begin{enumerate}

\item Let $B$ be a Boolean algebra. Define $a + b = (a \wedge b') \vee (a' \wedge b)$ (the symmetric difference)
and $a \cdot b = a \wedge b$.
Then $(B, +, \cdot, 1)$ is a Boolean ring.

\item Let $(R, +, \cdot, 1)$ be a Boolean ring.
Define  $a \vee b = a + b + a \cdot b$, $a \wedge b = a \cdot b$ and $a' = 1 - a$.
Then $(R, \vee, \wedge, \, ' \, , 0,1)$ is a Boolean algebra

\item The constructions (1) and (2) above are mutually inverse. 

\end{enumerate}
\end{theorem}

The above result is satisfying since the definition of Boolean rings could hardly be simpler
but also raises the interesting question of why Marshall H. Stone (1903--1989), a functional analyst, 
should have been interested in Boolean algebras in the first place.
The reason is that Stone worked on the spectral theory of symmetric operators 
and this led to an interest in algebras of commuting projections.
Such algebras are naturally Boolean algebras.
The connection between Boolean inverse semigroups and $C^{\ast}$-algebras
continues this tradition.
The following theorem \cite{Foster} puts this connection in a slightly wider context.
If $R$ is a ring denote its set of idempotents by $\mathsf{E}(R)$.
 
\begin{theorem}\label{them:foster} Let $R$ be a unital commutative ring.
Then the set $\mathsf{E}(R)$ is a Boolean algebra when we define $e \vee f = e + f - ef$, $e \wedge f = e \cdot f$
and $e' = 1 - e$.
\end{theorem}

By Theorem~\ref{them:bool-ring} and Theorem~\ref{them:foster}, it is immediate that each Boolean algebra
arises as the Boolean algebra of idempotents of a commutative ring.

So far we have viewed Boolean algebras as purely algebraic objects,
but in fact they come equipped with a partial order that underpins this algebraic structure. 
Let $B$ be a Boolean algebra.
For $x,y \in B$, define $x \leq y$ if and only if $x = x \wedge y$;
we say that $y$ {\em lies above} $x$.
The proofs of the following are routine.

\begin{lemma}\label{lem:order} With the above definition, we have the following:
\begin{enumerate}
\item $\leq$ is a partial order on $B$.

\item $x \leq y$ if and only if $y = x \vee y$.

\item $a \wedge b = \mbox{\rm glb}\{a,b\}$  and $a \vee b = \mbox{\rm lub}\{a,b\}$.

\end{enumerate}
\end{lemma}

\subsection{Finite Boolean algebras}

In this section, we describe the structure of all finite Boolean algebras.
It is only included to provide motivation for the sections that follow.
The crucial idea is that of an atom.
A non-zero element $x \in B$ of a Boolean algebra is called an {\em atom} if $y \leq x$ implies that either $x = y$ or $y = 0$.
We denote the set of atoms of the Boolean algebra $B$ by $\mathsf{at}(B)$.  
The proof of the following is immediate from the definition.

\begin{lemma}\label{lem:jojo} 
Let $x$ and $y$ be atoms. Then either $x = y$ or $x \wedge y = 0$.
\end{lemma}

The proof of the following is immediate.

\begin{lemma}\label{lem:atoms} let $B$ be a finite Boolean algebra.
Then every non-zero element lies above an atom.
\end{lemma}

The following lemma will be useful.

\begin{lemma}\label{lem:useful-one} Let $B$ be a Boolean algebra.
If $a \nleq b$ then $a \wedge b' \neq 0$.
\end{lemma}
\begin{proof}
Suppose that $a \wedge b' = 0$.
Then $a = a \wedge 1 = a \wedge (b \vee b') = a \wedge b$ from which we get that $a \leq b$, which is a contradiction.
\end{proof}

Let $B$ be a finite Boolean algebra.
For each $a \in B$, denote by $\mathscr{U}_{a}$ the set of all atoms in $B$ below $a$.
By Lemma~\ref{lem:atoms}, 
it follows that $a \neq 0$ implies that $\mathscr{U}_{a} \neq \varnothing$.
The important properties of the sets $\mathscr{U}_{a}$ we shall need are listed below.

\begin{lemma}\label{lem:germany} Let $B$ be a finite Boolean algebra.
\begin{enumerate}

\item $\mathscr{U}_{0} = \varnothing$. 

\item $\mathscr{U}_{1} = \mathsf{at}(B)$.

\item $\mathscr{U}_{a} \cap \mathscr{U}_{b} = \mathscr{U}_{a \wedge b}$.

\item $\mathscr{U}_{a} \cup \mathscr{U}_{b} = \mathscr{U}_{a \vee b}$.

\item $\mathscr{U}_{a'} = \overline{\mathscr{U}_{a}}$.

\item $\mathscr{U}_{a} \subseteq \mathscr{U}_{b}$ if and only if $a \leq b$.

\item  $\mathscr{U}_{a} = \mathscr{U}_{b}$ if and only if $a = b$.

\end{enumerate}
\end{lemma}
\begin{proof} 
(1) Immediate.

(2) Immediate.

(3) Let $x \in \mathscr{U}_{a} \cap \mathscr{U}_{b}$.
Then $x \leq a$ and $x \leq b$.
It follows that $x \leq a \wedge b$ and so $x \in \mathscr{U}_{a \wedge b}$.
Now suppose that $x \in \mathscr{U}_{a \wedge b}$.
Then $x \leq a \wedge b$.
But $a \wedge b \leq a,b$.
It follows that $x \in \mathscr{U}_{a} \cap \mathscr{U}_{b}$.

(4) Let $x \in \mathscr{U}_{a} \cup \mathscr{U}_{b}$.
Then $x \leq a$ or $x \leq b$.
In either case, $x \leq a \vee b$.
It follows that $x \in \mathscr{U}_{a \vee b}$.
Conversely, suppose that $x \in \mathscr{U}_{a \vee b}$.
Then $x \leq a \vee b$.
It follows that $x = (a \wedge x) \vee (b \wedge x)$.
But $x$ is an atom.
So, either $a \wedge x = x$ or $b \wedge x = x$; 
that is either $x \leq a$ or $x \leq b$.
Hence $x \in \mathscr{U}_{a} \cup \mathscr{U}_{b}$.

(5) Suppose that $x \in \overline{\mathscr{U}_{a}}$.
Then $x \nleq a$. 
It follows that $x \wedge a' \neq 0$ by Lemma~\ref{lem:useful-one}.
But $x$ is an atom, and so $x \leq a'$.
The proof of the reverse inclusion follows by 
part (6) of Lemma~\ref{lem:basic-properties},

(6) Suppose that $\mathscr{U}_{a} \subseteq \mathscr{U}_{b}$. 
If $a \nleq b$ then $a \wedge b' \neq 0$
by Lemma~\ref{lem:useful-one}.
By Lemma~\ref{lem:atoms}, there is an atom $x \leq a \wedge b'$.
It follows that $x \leq a$ and $x \leq b'$ and so $x \nleq b$ by part (5) above.
This contradicts our assumption.
It follows that $a \leq b$.
The proof of the converse is immediate.

(7) This is immediate by part (6).
\end{proof}

\begin{proposition}\label{prop:finite} 
Every finite Boolean algebra is isomorphic to the Boolean algebra of subsets of a finite set.
\end{proposition}
\begin{proof} Let $B$ be a finite Boolean algebra.
As pour set, we take $\mathsf{at}(B))$, the set of atoms of $B$.
Define a function $B \rightarrow \mathsf{P}(\mathsf{at}(B))$, the set of all subsets of $\mathsf{at}(B)$, by $a \mapsto \mathscr{U}_{a}$.
By part (7) of  Lemma~\ref{lem:germany}, this is an injective morphism of Boolean algebras.
It remains only to prove that it is surjective.
Let $A = \{x_{1}, \ldots, x_{n}\}$ be any non-empty set of atoms.
Put $a = \bigvee_{i=1}^{n} x_{i}$.
We shall prove that $\mathscr{U}_{a} = A$.
Let $x$ be any atom such that $x \leq a$.
Then $x = (x \wedge x_{1}) \vee \ldots \vee (x \wedge x_{n})$.
By Lemma~\ref{lem:jojo}, it follows that $x = x_{i}$ for some $i$.
\end{proof}

We now place the above result in its proper categorical context.

\begin{theorem}[Stone duality for finite Boolean algebras]\label{them:finite-duality} 
The category of finite Boolean algebras and homomorphisms between them
is dually equivalent to the category of finite sets and functions between them. 
\end{theorem}
\begin{proof}

Let $\theta \colon B \rightarrow C$ be a homomorphism between finite Boolean algebras.
Define $\theta^{\sharp} \colon \mathsf{at}(C) \rightarrow \mathsf{at}(B)$ by
$\theta^{\sharp}(f) = e$ if $f \leq \theta (e)$ and $e$ is an atom.
We shall prove that $\theta^{\sharp}$ really is a function.
Suppose that $e$ and $e'$ are distinct atoms such that $\theta (e) \geq f$ and $\theta (e') \geq f$.
Then $\theta (e \wedge e') \geq f$ but $e \wedge e' = 0$, since $e$ and $e$ are distinct atoms, which implies that $f = 0$.
This is a contradiction.
Therefore, $\theta^{\sharp}$ is a partial function.
We have that $1 = \bigvee_{e \in \mathsf{at}(B)} \theta (e)$ inside $B$.
Thus $1 = \theta (1) =  \bigvee_{\theta (e) \in \mathsf{at}(B)} \theta (e)$.
But $1 \geq f$.
It follows that $f =  \bigvee_{\theta (e) \in \mathsf{at}(B)} \theta (e) \wedge f$.
But $f$ is an atom, so that either $\theta (e) \wedge f = 0$ or $\theta (e) \wedge f = f$.
In the latter case, $f \leq \theta (e)$.
We have therefore proved that $\theta^{\sharp}$ is a function.

Let $\alpha \colon \mathsf{at}(C) \rightarrow \mathsf{at}(B)$ be any function.
Define $\alpha^{\flat} \colon B \rightarrow C$ by
$$\alpha^{\flat}(e) = \bigvee_{\alpha (f) \leq e} f.$$
We prove that $\alpha^{\flat}$ is a homomorphism of Boolean algebras.
The join of the empty set is $0$ and so $\alpha^{\flat}(0) = 0$.
Similarly, $\alpha^{\flat}(1) = 1$.
The fact that $\alpha^{\flat}$ preserves joins follows from  Lemma~\ref{lem:germany}.
The fact that $\alpha^{\flat}$ preserves meets is straightforward.
It follows that $\alpha^{\flat}$ is a homomorphism of Boolean algebras.

Let $\theta \colon B \rightarrow C$ be a homomorphism of Boolean algebras.
Then $\theta^{\sharp} \colon \mathsf{at}(C) \rightarrow \mathsf{at}(B)$ is a well-defined function.
We calculate the effect of $(\theta^{\sharp})^{\flat}$ on atoms.
Let $e$ be an atom of $B$.
Then by definition
$$(\theta^{\sharp})^{\flat}(e) = \bigvee_{\theta^{\sharp} (f) \leq e} f.$$
But $e$ is an atom and so $\theta^{\sharp}(f) = e$.
It follows that $f \leq \theta (e)$.
We are therefore looking at the join of all atoms below $\theta (e)$, which is exactly $\theta (e)$.
We have therefore proved that $(\theta^{\sharp})^{\flat} = \theta$ on atoms.
It follows that $(\theta^{\sharp})^{\flat} = \theta$ as functions.

Let $\alpha \colon \mathsf{at}(C) \rightarrow \mathsf{at}(B)$ be any function.
Let $f$ be any atom of $C$.
Then $(\alpha^{\flat})^{\sharp}(f) = e$ if and only if $f \leq \alpha^{\flat} (e)$ and $e$ is an atom.
By definition $$\alpha^{\flat}(e) = \bigvee_{\alpha (i) \leq e} i = \bigvee_{\alpha (i) = e}i.$$
It follows that $f = i$ for some $i$.
Thus $\alpha (f) = e$.
It follows that $(\alpha^{\flat})^{\sharp} = \alpha$.

It is now routine to check that we have defined a dual equivalence of categories.
\end{proof}

\subsection{Arbitrary Boolean algebras}

In the light of Proposition~\ref{prop:finite},
it is tempting to conjecture that every Boolean algebra is isomorphic to a powerset Boolean algebra.
However, this turns out to be false;
powerset Boolean algebras always have atoms but there are Boolean algebras
that have no atoms at all 
(the {\em atomless} Boolean algebras. See \cite[page 118, Chapter 4]{GH}. The Lindenbaum-Tarski Boolean algebra constructed from classical propositional logic is another example).
To describe arbitrary Boolean algebras, we have to adopt a different approach, and this was just what Stone did \cite{Stone1937}.
You will find classical Stone duality described in the following references \cite{BS, GH, Johnstone, Koppelberg}.
The approach is symbolized below and can be viewed as a generalization of the finite case described in the previous section:

$$
\begin{array}{ccc}
\text{atom} & \stackrel{\text{\tiny replaced by}}{\xrightarrow{\hspace*{2cm}}} & \text{ultrafilter} \\
\text{powerset} & \stackrel{\text{\tiny replaced by}}{ \xrightarrow{\hspace*{2cm}}} & \text{topological space}\\
\text{atom }x \leq a        &\stackrel{\text{\tiny replaced by}}{\xrightarrow{\hspace*{2cm}}} & a \in F \text{ ultrafilter}\\
\end{array}
$$

Let $B$ be an arbitrary Boolean algebra.
There is no reason for $B$ to have atoms, so we have to find `atom substitutes' that do always exist.
This is the role of the ultrafilters.
A non-empty subset $F$ of a Boolean algebra $B$ is called a {\em filter} if it is closed under meets and upwardly closed.
The filter $F$ is said to be {\em proper} if $0 \notin F$.

\begin{example}\label{ex:needed}
{\em Let $a \in B$.
Then $a^{\uparrow}$ is a filter called the {\em principal filter generated by $a$}
}.
\end{example}



A proper filter $F$ is said to be {\em prime} if $a \vee b \in F$ implies that $a \in F$ or $b \in F$. 
A maximal proper filter is called an {\em ultrafilter}.

\begin{lemma}\label{lem:france-two} The following are equivalent for a proper filter $F$ in a Boolean algebra $B$.
\begin{enumerate}

\item $F$ is an ultrafilter.

\item For each non-zero $a \in B$ either $a \in F$ or $a' \in F$.

\item $F$ is a prime filter.

\end{enumerate}
\end{lemma}
\begin{proof} (1)$\Rightarrow$(2).
Let $F$ be an ultrafilter.
Suppose that $a \notin F$.
Then by Lemma~\ref{lem:exel}, there exists $b \in F$ such that $a \wedge b = 0$.
Now $1 = a \vee a'$.
Thus $b = (b \wedge a) \vee (b \wedge a')$.
So, $b = b \wedge a'$.
It follows that $b \leq a'$ giving $a' \in F$, as required.

(2)$\Rightarrow$(3).
We prove that $F$ is a prime filter.
Let $a \vee b \in F$.
Suppose that $a \notin F$ and $b \notin F$.
Then, by assumption, $a' \in F$ and $b' \in F$ so that $a' \wedge b' \in F$.
Thus $(a \vee b)' \in F$ which is a contradiction.

(3)$\Rightarrow$(1). Let $F$ be a prime filter.
We prove that $F$ is an ultrafilter.
We shall use Lemma~\ref{lem:exel}.
Assume that $a \in B$ is such that $a \wedge b \neq 0$ for all $b \in F$.
We shall prove that $a \in A$.
Suppose for the sake of argument, that $a \notin F$.
Now, $1 = a \vee a' \in F$ and so either $a \in F$ or $a' \in F$.
It follows that $a' \in F$ but $a \wedge a' = 0$
which contradicts our assumption about $a$.
Thus $a \in F$.
We have therefore proved that $F$ is an ultrafilter.
\end{proof}

\begin{remark}\label{rem:ba-prime-equal-uf}
{\em 
In the light of the above result, the terms `prime filter' and `ultrafilter' are interchangeable in a Boolean algebra.}
\end{remark}

Ultrafilters are connected with homomorphisms to the two-element Boolean algebra, $\mathbb{B}$.
The following is the ultrafilter version of \cite[Proposition 2.2]{Johnstone}.

\begin{lemma}\label{lem:hm} Let $B$ be a Boolean algebra.
\begin{enumerate} 
\item Let $\theta \colon B \rightarrow \mathbb{B}$ be a homomorphism of Boolean algebras.
Then $F = \theta^{-1}(1)$ is an ultrafilter.
\item Let $F$ be an ultrafilter. Then the characteristic function $\chi_{F} \colon B \rightarrow \mathbb{B}$
is a homomorphism of Boolean algebras.
\end{enumerate}
\end{lemma}
\begin{proof} 
These are both straightforward to prove using Lemma~\ref{lem:france-two}.
\end{proof}

The above lemma actually establishes a bijection between ultrafilters in $B$ and homomorphisms
from $B$ to the two-element Boolean algebra $\mathbb{B}$.
We can connect atoms with special kinds of ultrafilters;
this enables us to link what we are doing in this section with what we did previously.

\begin{lemma}\label{lem:green} Let $B$ be a Boolean algebra.
The principal filter $F = a^{\uparrow}$ is a prime filter if and only if $a$ is an atom.
\end{lemma}
\begin{proof} Let $a$ be an atom. From Example~\ref{ex:needed}, we know that $a^{\uparrow}$ is a filter.
We shall prove that it is a prime filter.
Suppose that $b \vee c \in F$.
Then $a \leq b \vee c$.
Thus $a = (a \wedge b) \vee (a \wedge c)$.
It cannot happen that both $a \wedge b = 0$ and $a \wedge c = 0$.
Also $a \wedge b \leq a$ and $a \wedge c \leq a$.
But $a$ is an atom.
If $a \wedge b = a$ then $a \leq b$ and $b \in F$;
if $a \wedge b = 0$ then $a \wedge c = a$ implying that $a \leq c$ and so $c \in F$.
This proves that $a^{\uparrow}$ is a prime filter.
We now prove the converse.
Suppose that $F$ is an ultrafilter.
We prove that $a$ is an atom.
Suppose not.
Then there is $0 < b < a$.
Then $b^{\uparrow}$ is a filter and $F \subset b^{\uparrow}$.
But this contradicts the assumption that $F$ is an ultrafilter.
It follows that $a$ must be an atom.
\end{proof}

The above lemma is only interesting in the light of the following result.
The routine proof uses Zorn's Lemma\footnote{The fairy godmother of mathematics.} or see \cite[Chapter 1, Proposition~2.16]{Koppelberg}.

\begin{lemma}[Boolean Prime Ideal Theorem]\label{lem:BIG}
A non-empty subset of a Boolean algebra is contained in an ultrafilter if and only if it has the finite intersection property.
\end{lemma}

The first corollary below is the analogue of the result for finite Boolean algebras that every non-zero element
is above an atom.

\begin{corollary}\label{cor:latvia} 
Every non-zero element of a Boolean algebra is contained in an ultrafilter.
\end{corollary}

The second corollary says that there are enough ultrafilters to separate points;
this is the analogue of the result that says in a finite Boolean algebra each element is a join of the atoms below it. 

\begin{corollary}\label{cor:germany-two} Let $a$ and $b$ be distinct non-zero elements of a Boolean algebra.
Then there is an ultrafilter that contains one of the elements and omits the other.
\end{corollary}
\begin{proof} Since $a \neq b$ then either $a \nleq b$ or $b \nleq a$.
Suppose that $a \nleq b$.
Then $a \wedge b' \neq 0$ by Lemma~\ref{lem:useful-one}. 
Thus by Corollary~\ref{cor:latvia} there is an ultrafilter $F$ that contains $a \wedge b'$.
It follows that $a \in F$ and $b \notin F$.
\end{proof}

Ultrafilters are the first step in generalizing the theory of finite Boolean algebras to arbitrary Boolean algebras.
The second is to introduce topological spaces to repace powersets.
A compact Hausdorff space which is $0$-dimensional is called a {\em Boolean space};
for emphasis, these will also be referred to in this paper as {\em compact Boolean spaces}.

\begin{lemma}\label{lem:poland} 
The clopen subsets of a Boolean space form a Boolean algebra.
\end{lemma}
\begin{proof} Let $X$ be a Boolean space and denote by $\mathsf{B}(X)$ the set of all clopen subsets of $X$.
Observe that $\varnothing, X \in \mathsf{B}(X)$.
If $A,B \in \mathsf{B}(X)$ then $A \cap B, A \cup B \in \mathsf{B}(X)$. 
Finally, if $A \in \mathsf{B}(X)$ then $\overline{A} \in \mathsf{B}(X)$. 
\end{proof}

Let $B$ be a Boolean algebra.
Define $\mathsf{X}(B)$ to be the set of ultrafilters on $B$.
If $a \in B$ denote by $\mathscr{U}_{a}$ the set of ultrafilters containing $a$.

\begin{remark}{\em By Lemma~\ref{lem:green}, the above notation is consistent
with that introduced in Section~2.2.}
\end{remark}

\begin{lemma}\label{lem:estonia} Let $B$ be a Boolean algebra.
\begin{enumerate}
\item $\mathscr{U}_{0} = \varnothing$.
\item $\mathscr{U}_{1} = \mathsf{X}(B)$.
\item $\mathscr{U}_{a} \cap \mathscr{U}_{b} = \mathscr{U}_{a \wedge b}$.
\item $\mathscr{U}_{a} \cup \mathscr{U}_{b} = \mathscr{U}_{a \vee b}$.
\item $\mathscr{U}_{a'} =  \overline{\mathscr{U}_{a}}$.
\end{enumerate}
\end{lemma}
\begin{proof} The proofs of (1) and (2) are straightforward.
The proof of (3) follows fom the fact that filters are closed under meets.
The proof of (4) follows from the fact that ultrafilters are prime filters.
The proof of (5) follows from the fact that $a \wedge a' = 0$
and the fact that ultrafilters are proper filters, 
and part (2) of Lemma~\ref{lem:france-two}.
\end{proof}

The above lemma tells us that the collection of sets  $\mathscr{U}_{a}$, where $a \in B$,
forms the base for a topology on $\mathsf{X}(B)$.
We shall first of all determine the salient properties of this topological space. 

\begin{lemma}\label{lem:Boolean-space} 
For each Boolean algebra $B$, the topological space $\mathsf{X}(B)$ is Boolean.
\end{lemma}
\begin{proof} The base of the topology consists of sets of the form $\mathscr{U}_{a}$.
These are open by fiat.
But by  part (5) of Lemma~\ref{lem:estonia} they are also closed.
It follows that $\mathsf{X}(B)$ is $0$-dimensional.
We prove that this space is Hausdorff.
Let $A$ and $B$ be distinct ultrafilters.
Then there exists $a \in A \setminus B$, and so $a \notin B$.
We now use Lemma~\ref{lem:france-two} to deduce that $a' \in B$. 
It follows that $A \in \mathscr{U}_{a}$ and $B \in \mathscr{U}_{a'}$.
By part (3) of Lemma~\ref{lem:estonia}, 
we have that $\mathscr{U}_{a} \cap \mathscr{U}_{a'} = \varnothing$.
Thus, the space $\mathsf{X}(B)$ is Hausdorff.
Finally, we prove that the space $\mathsf{X}(B)$ is compact.
Let $\mathscr{C} = \{\mathscr{U}_{a} \colon a \in I \}$ be a cover of $\mathsf{X}(B)$.
Suppose that no finite subset of $\mathscr{C}$ covers $\mathsf{X}(B)$.
Then for any $a_{1}, \ldots, a_{n} \in I$ we have that $\mathscr{U}_{a_{1}} \cup \ldots \cup \mathscr{U}_{a_{n}} \neq \mathsf{X}(B)$.
It follows that $a_{1} \vee \ldots \vee a_{n} \neq 1$ and so $a_{1}' \wedge \ldots \wedge a_{n}' \neq 0$.
Thus the set $I' = \{a' \colon a \in I\}$ has the finite intersection property.
By Lemma~\ref{lem:BIG}, there is an ultrafilter $F$ such that $I' \subseteq F$.
By assumption, $F \in \mathscr{U}_{a}$ for some $a \in I$ and so $a,a' \in F$, which is a contradiction.
\end{proof}

The topological space $\mathsf{X}(B)$ is called the {\em Stone space} of the Boolean algebra $B$.

We can now assemble Lemma~\ref{lem:poland} and Lemma~\ref{lem:Boolean-space} into the first main result.

\begin{proposition}\label{prop:big-theorem}\mbox{}
\begin{enumerate}

\item Let $B$ be a Boolean algebra.
Then $B \cong \mathsf{B} (\mathsf{X}(B))$, where here $\cong$ means an isomorphism of Boolean algebras.

\item Let $X$ be a Boolean space.
Then $X \cong \mathsf{X} (\mathsf{B}(X))$, where here $\cong$ means a homeomorphism of topological spaces.

\end{enumerate}
\end{proposition}
\begin{proof} (1) Define $\alpha \colon B \rightarrow \mathsf{B} \mathsf{X}(B)$ by $a \mapsto \mathscr{U}_{a}$.
By Lemma~\ref{lem:estonia} this is a homomorphism of Boolean algebras.
It is injective by Corollary~\ref{cor:germany-two}.
We prove that it is surjective.
An element of $\mathsf{B} \mathsf{X}(B)$ is a clopen subset of $\mathsf{X}(B)$.
Since it is open, it is a union of open sets of the form $\mathscr{U}_{a}$,
but closed subsets of compact spaces are compact by part (3) of Lemma~\ref{lem:topology-needed}.
It follows that it is a union of a finite number of sets of the form $\mathscr{U}_{a}$
and so must itself be of that form.  

(2) Let $x \in X$. Define $O_{x}$ to be the set of all clopen sets that contain $x$.
It is easy to check that this is a prime filter in $\mathsf{B}(X)$ and so $O_{x} \in \mathsf{X} (\mathsf{B}(X))$.
Define $\beta \colon X \rightarrow \mathsf{X} (\mathsf{B}(X))$ by $x \mapsto O_{x}$.
Since both domain and codomain spaces are compact and Hausdorff, to prove that $\beta$ is a homeomorphism it
is enough to prove that it is bijective and continuous.
Suppose that $O_{x} = O_{y}$.
If $x \neq y$ then by the fact that $X$ is Hausdorff we can find disjoint open sets $U$ and $V$ such that
$x \in U$ and $y \in V$.
But $X$ is $0$-dimensional and so we can assume, without loss of generality, that $U$ and $V$ are clopen
from which we deduce that $O_{x} \neq O_{y}$.
It follows that $\beta$ is injective.
Next, let $F$ be any ultrafilter in $\mathsf{B}(X)$.
Then this is an ultrafilter consisting of clopen subsets of a compact space.
Since $F$ is a filter, it has the finite intersection property.
By part (4) of Lemma~\ref{lem:topology-needed}, it follows that there is an element $x$ in the intersection of all the elements of $F$.
Thus $F \subseteq O_{x}$.
But $F$ is an ultrafilter and so $F = O_{x}$.
We have therefore proved our function is a bijection.
Finally, we prove continuity.
Let $U$ be an open subset of $\mathsf{X} (\mathsf{B}(X))$.
Then $U$ is a union of the basic open sets which are clopen.
These have the form $\mathscr{U}_{A}$ where $A$ is a clopen subset of $X$.
Thus it is enough to calculate $\beta^{-1}(\mathscr{U}_{A})$.
But $O_{x} \in \mathscr{U}_{A}$ if and only if $x \in A$.
Thus $\beta^{-1}(\mathscr{U}_{A}) = A$.
\end{proof}

We can extend the above result to maps to obtain the following:

\begin{theorem}[Classical Stone duality I]\label{them:classical-stone-duality} 
The category of Boolean algebras and their homomorphisms is dually equivalent
to the category of Boolean spaces and continuous functions between them.
\end{theorem}
\begin{proof} In Lemma~\ref{lem:hm}, we proved that there is
a bijective map between the ultrafilters in $B$ and the Boolean algebra homomorphisms from $B$ to $\mathbb{B}$,
the $2$-element Boolean algebra.
This bijection associates with the ultrafilter $F$ its characteristic function $\chi_{F}$.
Let $\theta \colon B_{1} \rightarrow B_{2}$ be a homomorphism between Boolean algebras.
Let $F \in \mathsf{X}(B_{2})$ be an ultrafilter.
Then $\chi_{F} \theta$ is the characteristic function of an ultrafilter in $B_{1}$.
In this way, we can map homomomorphisms $B_{1} \rightarrow B_{2}$ to continuous functions $\mathsf{X}(B_{1}) \leftarrow \mathsf{X}(B_{2})$
with a consequent reversal of arrows.
In the other direction, let $\phi \colon X_{1} \rightarrow X_{2}$ be a continuous function.
Then $\phi^{-1}$ maps clopen sets to clopen sets.
In this way, we can map continuous functions $X_{1} \rightarrow X_{2}$ to homomorphisms $\mathsf{B}(X_{1}) \leftarrow \mathsf{B}(X_{2})$.
The result now follows from Proposition~\ref{prop:big-theorem}.
\end{proof}

\begin{example}\label{ex:some-stone-spaces} 
{\em Here are some examples of classical Stone duality.

\begin{enumerate}

\item Let $B$ be a finite Boolean algebra.
By Lemma~\ref{lem:green}, the ultrafilters of $B$ are in bijective correspondence with the atoms of $B$.
We may therefore identify the elements of $\mathsf{X}(B)$ with the set of atoms of $B$.
Let $a \in B$. We describe the set $\mathscr{U}_{a}$ in terms of atoms.
The ultrafilter $b^{\uparrow} \in \mathscr{U}_{a}$ if and only is $b \leq a$.
So, the set $\mathscr{U}_{a}$ is in bijective correspondence with the set of atoms below $a$.
It follows that the Boolean space $\mathsf{X}(B)$ is homeomorphic with the the discrete space of atoms of $B$.
Let $\theta \colon B \rightarrow C$ be a homomorphism of finite Boolean algebras.
If $c \in C$ is an atom then $c^{\uparrow}$ is an ultrafilter in $C$ and so $\theta^{-1}(c^{\uparrow})$ is an ultrafilter in $B$.
It follows that $\theta^{-1}(c^{\uparrow}) = b^{\uparrow}$, where $b$ is an atom in $B$.
Thus $x \geq b$ if and only if $\theta (x) \geq c$.
We therefore have that $\theta (b) \geq c$.
But $b$ is the only atom of $B$ which has this property.
In this way, the classical theory of finite Boolean algebras can be derived from Stone duality;
that is, Theorem~\ref{them:finite-duality} is a special case of Theorem~\ref{them:classical-stone-duality}.

\item Tarski proved that any two atomless, countably infinite Boolean algebras are isomorphic \cite[Chapter 16, Theorem 10]{GH}.
It makes sense, therefore, to define the {\em Tarski algebra}\footnote{Not an established term.} to be an atomless, countably infinite Boolean algebra.
We describe the Stone space of the Tarski algebra.
An element $x$ of a topological space is said to be {\em isolated} if $\{x\}$ is open.
Suppose that $a$ is an atom of the Boolean algebra $B$.
By definition $\mathscr{U}_{a}$ is the set of all ultrafilters that contain $a$.
But $a^{\uparrow}$ is an ultrafilter containing $a$ by Lemma~\ref{lem:green} and, evidently, the only one.
Thus $\mathscr{U}_{a}$ is an open set containing one point and so the point $a^{\uparrow}$ is isolated.
Suppose that $\mathscr{U}_{a}$ contains exactly one point $F$.
Then $F$ is the only ultrafilter containing $a$.
Suppose that $a$ were not an atom.
Then we could find $0 \neq b < a$.
Thus $a = b \vee (a \wedge b')$.
By Corollary~\ref{cor:latvia}, 
there is an ultrafiler $F_{1}$ containing $b$,
and there is an ultrafilter $F_{2}$ containing $a \wedge b'$.
Then $F_{1} \neq F_{2}$ but both contain $a$.
This is a contradiction.
It follows that $a$ is an atom.
We deduce that the atoms of the Boolean algebra determine
the isolated points of the associated Stone space.
It follows that a Boolean algebra has the property that every element is above an atom (that is, it is {\em atomic})
if and only if the isolated points in its Stone space form a dense subset.
We deduce that the Stone space associated with an atomless Boolean algebra has no isolated points.
If $B$ is countable then its Stone space is second-countable since
$B$ is isomorphic to the set of all clopen subsets of the Stone space of $B$.
The Stone space of the Tarski algebra is therefore  
a $0$-dimensional, second countable, compact, Hausdorff space with no isolated points.
Observe by \cite[Theorem 9.5.10]{Vickers} that such a space is metrizable.
It follows by Brouwer's theorem, \cite[Theorem 30.3]{Willard}, that the Stone space of the Tarski algebra
is the {\em Cantor space}.

\item The Cantor space described in Example~2 above often appears in disguise.
Let $A$ be any finite set with at least two elements.
Denote by $X = A^{\omega}$ the set of all right-infinite strings of elements over $A$.
We can regard this set as the product space $A^{\mathbb{N}}$ which is compact since $A$ is finite
by part (1) of Lemma~\ref{lem:topology-needed}.
For each finite string $x \in A^{\ast}$ denote by $xX$ the subset of $X$ that consists of  
all elements of $X$ that begin with the finite string $x$.
This is an open set of $X$.
If $a \in A$ denote by $\hat{a} = A \setminus \{a\}$. 
Let $x = x_{1} \ldots x_{n}$ have length $n \geq 1$.
Then $\overline{xX} = \hat{x}_{1}X \cup x_{1}\hat{x}_{2}X \cup \ldots x_{1}\ldots x_{n-1}\hat{x}_{n}X$.
It follows that if $xX$ is open then $\overline{xX}$ is open.
Thus the sets $xX$ are clopen.
The set $A^{\ast}$ is countably infinite and so the number of clopen subsets is countably infinite.
If $x,y \in A^{\ast}$ then there are a few possibilities.
If neither $x$ nor $y$ is the prefix of the other then $xX \cap yX = \varnothing$.
Now, suppose that $x = yu$.
Then $xX = yuX \subseteq yX$ from which it follows that $xX \cap yX = xX$.
An open subset $U$ of $X$ has the form $U = X_{1} \ldots X_{n}X$,
where the $X_{i}$ are subsets of $A$.
This is a union of sets of the form $xX$ where $x \in X_{1} \ldots X_{n}$.
It follows that the sets $\varnothing$ and $xX$, where $x \in A^{\ast}$, form a clopen base for the topology on $X$.
If $w$ and $w'$ are distinct elements of $X$ then they will differ in the $n$th position and so belong to disjoint sets of the form $xX$.
It follows that $X$ is a second-countable Boolean space.
This space cannot have any isolated points:
if $\{w\}$ is an open subset then it must be a union of sets of the form $xX$
but this is impossible,
It follows that $X$ is the Cantor space.
We refer the reader to \cite{Lawson2007, Lawson2007b} and \cite[Section 5]{LS} for more on this topological space.

\item We construct the Stone spaces of the powerset Boolean algebras $\mathsf{P}(X)$.
The isolated points of the Stone space of $\mathsf{P}(X)$ form a dense subset of the Stone space which is homeomorphic to
the discrete space $X$. Thus the Stone space of $\mathsf{P}(X)$ is a compact Hausdorff space that contains a copy of the discrete space $X$.
In fact, the Stone-\v{C}ech compactification of the discrete space $X$ is precisely the Stone space of $\mathsf{P}(X)$.
See \cite[Section 30, Theorem A]{Simmons} and \cite[Section 75]{Simmons}.

\item Let $A$ be any finite alphabet.
We shall use notation from the theory of regular expressions so that $L + M$ means $L \cup M$ and $x$ can mean $\{x \}$ (but also the string $x$ in a different context).
A language $L$ over $A$ is said to be {\em definite}\footnote{Strictly speaking, `reverse definite'.}
if $L = X + YA^{\ast}$ where $X,Y \subseteq A^{\ast}$ are finite languages. 
Denote the set of definite languages by $\mathscr{D}$.
This forms a Boolean algebra.
The Stone space of the Boolean algebra $\mathscr{D}$ can be described as follows.
Put $\mathscr{X} = A^{\ast} + A^{\omega}$.
If $x$ and $y$ are distinct elements of $\mathscr{X}$, 
define $x \wedge y$ to be the largest common prefix of $x$ and $y$.
Define $d(x,y) = 0$ if $x = y$ else $d(x,y) = \exp(-|x \wedge y|)$.
Then $d$ is an ultrametric on $\mathscr{X}$ and $\mathscr{X}$ is a complete metric space with respect to this ultrametric.
The open balls are of the form $\{x\}$ or $xA^{\ast} + xA^{\omega}$ where $x \in A^{\ast}$ and form a basis for the topology.
Thus $\mathscr{X}$ is a Boolean space.
It can be proved that the Stone space of $\mathscr{D}$ is this ultrametric space $\mathscr{X}$.
See \cite{Pin} for more on this example.

\end{enumerate}
}
\end{example}

\subsection{Generalized Boolean algebras}

There is a generalization of classical Stone duality,
Theorem~\ref{them:classical-stone-duality}, 
that relates what are termed generalized Boolean algebras to locally compact Boolean spaces.
At the level of objects, this was described in \cite[Theorem 4]{Stone1937} and at the level of homomorphisms in \cite{Doctor}.

In elementary work \cite{Johnstone}, Boolean algebras are usually defined with a top element and globally defined complements.
However, this is too restrictive for the applications we have in mind; it corresponds in topological language
to only looking at compact spaces even though many mathematically interesting spaces are locally compact.
In this section, we shall study what are termed `generalized Boolean algebras' or, to adapt terminology current in ring theory,
{\em non-unital Boolean algebras}.
Similarly, in this section, a {\em distributive lattice} will always have a bottom but not necessarily a top.
Let $D$ be a distributive lattice equipped with a binary operation $\setminus$ such that for all $x,y \in D$ we have that
$0 = y \wedge (x \setminus y)$ and $x = (x \wedge y) \vee (x \setminus y)$.
We say that such a distributive lattice is a {\em generalized Boolean algebra}.
Let $B$ be a generalized Boolean algebra.
A subset $C \subseteq B$ is said to be a {\em subalgebra} if it contains the bottom element of $B$,
is closed under meets and joins and is closed under the operation $\setminus$.
Such a $C$ is a generalized Boolean algebra in its own right.

If $b \leq a$ in a lattice then $[b,a]$ denotes the set of all elements $x$ of the lattice such that $b \leq x \leq a$.
We call the set $[b,a]$ an {\em interval}.
If $c \in [b,a]$ then a {\em complement} of $c$ is an element $d \in [b,a]$ such
that $c \wedge d = b$ and $c \vee d = a$.
We say that $[b,a]$ is {\em complemented} if every element has a complement.
Let $D$ be a distributive lattice.
We say it is {\em relatively complemented} if for every pair $b \leq a$,
the interval $[b,a]$ is complemented.

\begin{lemma}\label{lem:gen-ba} 
Let $D$ be a distributive lattice.
Then the following are equivalent:
\begin{enumerate}
\item $D$ is a generalized Boolean algebra.
\item Each non-zero principal order-ideal of $D$ is a unital Boolean algebra.
\item $D$ is relatively complemented.
\end{enumerate}
\end{lemma} 
\begin{proof} 
(1)$\Rightarrow$(2). Let $a \in B$ be non-zero. 
Then $a^{\downarrow}$ is a distribitive lattice with bottom element $0$ and top element $a$.
Let $b \leq a$.
Then $b \wedge (a \setminus b) = 0$ and $b \vee (a \setminus b) = a$.
It follows that within $a^{\downarrow}$ we should define $b' = a \setminus b$.
Thus each non-zero principal order ideal is a unital Boolean algebra. 

(2)$\Rightarrow$(1).
Let $x,y \in D$
where $x \neq 0$.
Then $x \wedge y \leq x$.
Define $x \setminus y = (x \wedge y)'$ where $(x \wedge y)'$ is the complement of $x \wedge y$ in the Boolean algebra $x^{\downarrow}$.
By definition $(x \wedge y) \vee (x \setminus y) = x$
and $y \wedge (x \setminus y) = y \wedge ((x \setminus y) \wedge x = (x \wedge y) \wedge (x \setminus y) = 0$.
If $x = 0$ then define $0 = (0 \setminus 0)$.

(1)$\Rightarrow$(3).
Let $b \leq a$.
Let $x \in [b,a]$.
Put $y = (a \setminus x) \vee b$.
Then $x \wedge y = b$ and $x \vee y = a$.
We have proved that in each interval, every element has a complement.

(3)$\Rightarrow$(1). Immediate.
\end{proof}

In the light of the above result, we shall regard generalized Boolean algebras as
distributive lattices with zero in which each non-zero principal order-ideal
is a Boolean algebra.

\begin{example}{\em Let $B$ be the set of all finite subsets of $\mathbb{N}$.
Then $B$ is a generalized Boolean algebra but not a (unital) Boolean algebra.}
\end{example}

We may define ultrafilters and prime filters in generalized Boolean algebras just as we defined them in unital Boolean algebras.
Let $B$ be a generalized Boolean algebra.
Define $\mathsf{X}(B)$ to be the set of ultrafilters on $B$.
If $a \in B$ denote by $\mathscr{U}_{a}$ the set of ultrafilters containing $a$.

\begin{lemma}\label{lem:bijection} Let $B$ be a generalized Boolean algebra and let $a$ be any non-zero element.
Then there is an order-isomorphism between the filters in the Boolean algebra $a^{\downarrow}$ and 
the filters in $B$ that contain $a$.
Under this order-isomorphism, proper filters correspond to proper filters,
and ultrafilters to ultrafilters.
\end{lemma}
\begin{proof} Let $F$ be a filter of $B$ that contains $a$.
Put $F_{\downarrow} = F \cap a^{\downarrow}$.
Then $F_{\downarrow}$ is non-empty and it is straightforward to show that it is a filter.
Observe that if $F_{1} \subseteq F_{2}$ are filters of $B$ that contain $a$
then $(F_{1})_{\downarrow} \subseteq (F_{2})_{\downarrow}$.

Let $G$ be a filter of $a^{\downarrow}$.
The proof that $G^{\uparrow}$, taken in $B$, is a filter of $B$ that contains $a$ is straightforward.
Observe that if $G_{1} \subseteq G_{2}$ are both filters of $a^{\downarrow}$
then $G_{1}^{\uparrow} \subseteq G_{2}^{\uparrow}$.

It is now routine to check that
$F = (F_{\downarrow})^{\uparrow}$ and $G = (G^{\uparrow})_{\downarrow}$.
We have therefore established our order-isomorphism.
Since $B$ and $a^{\downarrow}$ have the same bottom element
it is routine to check that proper filters in $a^{\downarrow}$ are mapped
to proper filters in $B$,
and that ultrafilters in $a^{\downarrow}$ are mapped to ultrafilters in $B$.
\end{proof}

Part (1) below was proved as \cite[Theorem 3]{Stone1937b}
and 
part (2) below was proved as \cite[Proposition 1.6]{LL} and uses Lemma~\ref{lem:bijection}.

\begin{lemma}\label{lem:ll} \mbox{}
\begin{enumerate}
\item In a distributive lattice every ultrafilter is a prime filter.
\item A distributive lattice is a generalized Boolean algebra if and only if every prime filter is an ultrafilter.
\end{enumerate}
\end{lemma}

\begin{remark}\label{rem:prime-equal-uf}
{\em 
It follows that in a generalized Boolean algebra, prime filters and ultrafilters are the same.
}
\end{remark}

Let $X$ be a Hausdorff space.
Then $X$ is {\em locally compact} if each point of $X$ is contained in the interior of a compact subset \cite[Theorem 18.2]{Willard}.

\begin{lemma}\label{lem:boolean-space} Let $X$ be a Hausdorff space.
Then the following are equivalent.
\begin{enumerate}
\item $X$ is locally compact and $0$-dimensional.
\item $X$ has a base of compact-open sets.
\end{enumerate}
\end{lemma}
\begin{proof} (1)$\Rightarrow$(2).
Let $U$ be a clopen set (since the space is $0$-dimensional).
Let $x \in U$.
Since $X$ is locally compact, there exists a compact set $V$ such that $x \in V^{\circ}$.
Now $x \in U \cap V^{\circ}$ is open and $X$ has a basis of clopen sets.
In particular, we can find a clopen set $W$ such that $x \in W \subseteq U \cap V^{\circ} \subseteq V$.
By part (3) of Lemma~\ref{lem:topology-needed},
$W$ is a closed subset of the compact set $V$ and so $W$ is compact.
It follows that $x \in W \subseteq U$ where $W$ is compact-open.
Thus $X$ has a base of compact-open sets.

(2)$\Rightarrow$(1).
By part (2) of Lemma~\ref{lem:topology-needed},
every compact subset of a Hausdorff space is closed.
It follows that $X$ has a basis of clopen subset.
It is immediate that the space is locally compact.
\end{proof}

We define a {\em locally compact Boolean space} to be a $0$-dimensional, locally compact Hausdorff space.
Let $X$ be a locally compact Boolean space.
Denote by $\mathsf{B}(X)$ the set of all compact-open subsets of $X$.
The proof of the following is straightforward using Lemma~\ref{lem:topology-needed}.

\begin{lemma}\label{lem:compact-open} Let $X$ be a locally compact Boolean space.
Then under the usual operations of union and intersection, 
the poset $\mathsf{B}(X)$ is a generalized Boolean algebra.
\end{lemma}

The proof of the following lemma is routine, once you recall that ultrafilters and prime filters
are the same thing in generalized Boolean algebras.

\begin{lemma}\label{lem:new-estonia} Let $B$ be a generalized Boolean algebra.
\begin{enumerate}
\item $\mathscr{U}_{0} = \varnothing$.
\item $\mathscr{U}_{a} \cap \mathscr{U}_{b} = \mathscr{U}_{a \wedge b}$.
\item $\mathscr{U}_{a} \cup \mathscr{U}_{b} = \mathscr{U}_{a \vee b}$.
\end{enumerate}
\end{lemma}

The above lemma tells us that the collection of all sets of the form $\mathscr{U}_{a}$, where $a \in B$,
is the base for a topology on $\mathsf{X}(B)$.
We shall first of all determine the salient properties of the topological space $\mathsf{X}(B)$.

\begin{lemma}\label{lem:lc-Boolean-space} 
For each generalized Boolean algebra $B$, the topological space $\mathsf{X}(B)$ is a locally compact Boolean space.
\end{lemma}
\begin{proof} Let $A$ and $B$ be distinct ultrafilters.
Let $a \in A \setminus B$;
such an element exists since we cannot have that $A$ is a proper subset of $B$ since both are ultrafilters.
By Lemma~\ref{lem:exel}, there exists $b \in B$ such that $a \wedge b = 0$.
Observe that $\mathscr{U}_{a} \cap \mathscr{U}_{b} = \varnothing$ and $A \in \mathscr{U}_{a}$ and $B \in \mathscr{U}_{b}$.
We have proved that $\mathsf{X}(B)$ is Hausdorff.
It only remains to prove that each of the sets $\mathscr{U}_{a}$ is compact.
Suppose that $\mathscr{U}_{a} \subseteq \bigcup_{i \in I} \mathscr{U}_{b_{i}}$.
Observe that $\mathscr{U}_{a} = \bigcup_{i \in I} \mathscr{U}_{a \wedge b_{i}}$.
So, without loss of generality, we can assume that $b_{i} \leq a$.
Thus we are given that $\mathscr{U}_{a} = \bigcup_{i \in I} \mathscr{U}_{b_{i}}$
where $b_{i} \leq a$.
By assumption, $a^{\downarrow}$ is a unital Boolean algebra.
The result now follows by Lemma~\ref{lem:bijection} and 
Lemma~\ref{lem:Boolean-space}. 
\end{proof}

Just as before, the topological space $\mathsf{X}(B)$ is called the {\em Stone space} of the generalized Boolean algebra $B$.
We can now assemble Lemma~\ref{lem:compact-open}  and Lemma~\ref{lem:lc-Boolean-space} into the following result:

\begin{proposition}\label{prop:second-big-theorem}\mbox{}
\begin{enumerate}

\item Let $B$ be a generalized Boolean algebra.
Then $B \cong \mathsf{B} (\mathsf{X}(B))$, where here $\cong$ means an isomorphism of generalized Boolean algebras.

\item Let $X$ be a locally compact Boolean space.
Then $X \cong \mathsf{X} (\mathsf{B}(X))$, where here $\cong$ means a homeomorphism of topological spaces.

\end{enumerate}
\end{proposition}
\begin{proof} (1) Define $\alpha \colon B \rightarrow \mathsf{B} (\mathsf{X}(B))$ by $a \mapsto \mathscr{U}_{a}$.
Let $a$ and $b$ be elements of $B$.
Suppose that $\mathscr{U}_{a} = \mathscr{U}_{b}$.
Then both $a$ and $b$ are in the order ideal $(a \vee b)^{\downarrow}$.
It follows by Lemma~\ref{lem:bijection} and Proposition~\ref{prop:big-theorem} that $a = b$.
By Lemma~\ref{lem:new-estonia}, the bottom element is mapped to the bottom element
and binary meets and binary joins are preserved.
It remains to show that it is surjective.
Let $U$ be a compact-open set of $\mathsf{B} (\mathsf{X}(B))$.
Since it is open it is a union of basic open sets and since it is compact it is 
a union of a finite number of basic open sets.
But this implies that $U$ is a basic open set and so $U = \mathscr{U}_{a}$ for some $a \in B$.

(2) For each $x \in X$ denote by $O_{x}$ the set of all compact-open sets that contain $x$.
It is easy to check that $O_{x}$ is a prime filter in $\mathsf{B}(X)$
and so we have defined a map from $X$ to $\mathsf{X} (\mathsf{B}(X))$.
This map is injective because locally compact Boolean spaces are Hausdorff.
Let $F$ be an arbitrary ultrafilter in $\mathsf{B}(X)$.
This is therefore an ultrafilter whose elements are compact-open.
Let $V \in F$.
We now use part (2) of Lemma~\ref{lem:topology-needed}:
compact subsets of Hausdorff spaces are closed.
Since $F$ is an ultrafilter, 
all intersections of elements of $F$ with $V$ are non-empty and the set of sets so formed has the finite 
intersection property.
It follows that there is a point $x$ that belongs to them all.
Thus $x$ belongs to every element of $F$.
Thus $F \subseteq O_{x}$ from which we get equality since we are dealing with ultrafilters.
We have therefore established that we have a bijection.
In particular, every ultrafilter in $\mathsf{B}(X)$ is of the form $O_{x}$.
We prove that this function and its inverse are continuous.
Let $V$ be a compact-open set in $X$.
Then the image of $V$ under our map is the set
$S = \{O_{x} \colon x \in V\}$.
But $V$ is an element of $\mathsf{B}(X)$.
The set $S$ is just $\mathscr{U}_{V}$.
We now show that the inverse function is continuous.
An element of a base for the topology on $\mathsf{B}(X)$ is of the form $\mathscr{U}_{a}$
where $a \in \mathsf{B}(X)$.
Let $a = V$ a compact-open subset of $X$.
A typical element of $\mathscr{U}_{V}$ is $O_{x}$ where $x \in V$.
It follows that the inverse image of $\mathscr{U}_{V}$ is $V$.
\end{proof}

Let $\theta \colon B \rightarrow C$ be a homomorphism of generalized Boolean algebras.
We say it is {\em proper} if $C = \mbox{\rm im}(\theta)^{\downarrow}$;
in other words, each element of $C$ is below an element of the image.
A continuous map between topological spaces is said to be {\em proper}
if the inverse images of compact sets are compact.

\begin{theorem}[Commutative Stone duality II]\label{them:classical-stone-dualityII} 
The category of generalized Bool\-ean algebras and proper homomorphisms
is dually equivalent to the category of locally compact Boolean spaces and proper continuous homomorphisms.
\end{theorem}
\begin{proof} Let $\theta \colon B_{1} \rightarrow B_{2}$ be a proper homomorphism between generalized Boolean algebras
and let $F$ be an ultrafilter in $B_{2}$.
Then $\theta^{-1}(F)$ is non-empty because the homomorphism is proper and it is an ultrafilter in $B_{1}$.
We therefore have a map $\theta^{-1} \colon \mathsf{X}(B_{2}) \rightarrow \mathsf{X}(B_{1})$.
Let $\phi \colon X_{1} \rightarrow X_{2}$ be a proper continuous map between locally compact Boolean spaces
and  let $U$ be a compact-open subset of $X_{2}$. 
Then $\phi^{-1}(U)$ is also compact-open.
We therefore have a map $\phi^{-1} \colon \mathsf{B}(X_{2}) \rightarrow \mathsf{B}(X_{1})$.
It is now routine to check using Proposition~\ref{prop:second-big-theorem},
that we have a duality between categories.
\end{proof}

The above theorem generalizes Theorem~\ref{them:classical-stone-duality}
since homomorphisms between unital Boolean algebras are automatically proper,
and in a Hausdorff space compact sets are closed and closed subsets of compact spaces are themselves compact by Lemma~\ref{lem:topology-needed}
and so continuous maps between Boolean spaces are automatically proper.
We shall now generealize the above theorem in the sections that follow.

Every locally compact Hausdoff space admits a one-point compactification \cite[Section 37]{Simmons}
in which the resulting space is compact Hausdorff.
If the original space if $0$-dimensional so, too, is its one-point compactification;
see \cite[Exercise 43.19]{GH} and \cite[page 387]{Stone1937}.

\begin{lemma}\label{lem:one-point}
The one-point compactification of a locally compact Boolean space is a compact Boolean space.
\end{lemma}

We shall return to this lemma later in Section~9.

\section{Boolean inverse semigroups}

Let $S$ be an inverse semigroup.
We denote its semilattice of idempotents by $\mathsf{E}(S)$.
If $X \subseteq S$, define $\mathsf{E}(X) = X \cap \mathsf{E}(S)$.
If $a \in S$, write $\mathbf{d}(a) = a^{-1}a$ and $\mathbf{r}(a) = aa^{-1}$.
We say that $a,b \in S$ are {\em compatible}, written $a \sim b$, if $a^{-1}b$ and $ab^{-1}$ are both idempotents.
A pair of elements being compatible is a necessary condition for them to have an upper bound.
A subset is said to be {\em compatible} if every pair of elements in that subset is compatible.
The following was proved in \cite[Lemma 1.4.11, Lemma 1.4.12]{Lawson1998}.

\begin{lemma}\label{lem:buffs1}\mbox{}
\begin{enumerate} 

\item $s \sim t$ if and only if $s \wedge t$ exists and 
$\dom (s \wedge t) = \dom (s) \wedge \dom (t)$
and
$\ran (s \wedge t) = \ran (s) \wedge \ran (t)$.

\item If $a \sim b$ then
$$a \wedge b = ab^{-1}b = bb^{-1}a = aa^{-1}b = ba^{-1}a.$$

\end{enumerate}
\end{lemma}

We now suppose our inverse semigroup contains a zero.
If $e$ and $f$ are idempotents then we say they are {\em orthogonal}, written $e \perp f$, if $ef = 0$.
If $a$ and $b$ are elements of an inverse semigroup with zero we say that they are {\em orthogonal},
written $a \perp b$, if $\mathbf{d}(a) \perp \mathbf{d}(b)$ and $\mathbf{r}(a) \perp \mathbf{r}(b)$.
If $a \perp b$ then $a \sim b$;
in this case, if $a \vee b$ exists we often write $a \oplus b$ and talk about {\em orthogonal joins}.
This terminology can be extended to any finite set.

An inverse semigroup with zero is said to be {\em distributive} if it has all binary compatible joins
and multiplication distributes overs such joins.
The semilattice of idempotents of a distributive inverse semigroup
is a distributive lattice.

An inverse semigroup is a {\em meet-semigroup} if it has all binary meets.
Let $S$ be an arbitrary inverse semigroup.
A function $\phi \colon S \rightarrow \mathsf{E}(S)$ is called a 
{\em fixed-point operator} if for each $a \in S$ the element $\phi (a)$ is the largest
idempotent less than or equal to $a$.
The proofs of the following can be found in \cite{Leech} or follows from the definition.

\begin{lemma}\label{lem:meets}
Let $S$ be an inverse semigroup.
\begin{enumerate}
\item $S$ is a meet-semigroup if and only if it has a fixed-point operator.
\item If $S$ is a meet-semigroup, then we may define $\phi$ by $\phi (a) = a \wedge \mathbf{d}(a) (= a \wedge \mathbf{r}(a)$).
\item If $\phi$ is a fixed-point operator then $\phi (ae) = \phi (a)e$ and $\phi (ea) = e \phi (a)$ for all $e \in \mathsf{E}(S)$.
\item If $\phi$ is a fixed-point operator then $a \wedge b = \phi (ab^{-1})b$.
\end{enumerate}
\end{lemma}

It is important to be able to manipulate meets and joins in a distributive inverse semigroup.
The following result tells us how.
Part (1) was proved as \cite[Proposition~1.4.17]{Lawson1998},
part (2) was proved as  \cite[Proposition 1.4.9]{Lawson1998},
and parts (3), (4) and (5) were proved as \cite[Lemma 2.5]{Lawson3}.

\begin{lemma}\label{lem:meets-joins} In a distributive inverse semigroup, the following hold.
\begin{enumerate}

\item In a distributive inverse monoid, if $a \vee b$ exists then
$$\mathbf{d}(a \vee b) = \mathbf{d}(a) \vee \mathbf{d}(b)
\mbox{ and }
\mathbf{r}(a \vee b) = \mathbf{r}(a) \vee \mathbf{r}(b).$$

\item If $a \wedge b$ exists then $ac \wedge bc$ exists and $(a \wedge b)c = ac \wedge bc$, and dually.

\item Suppose that $\bigvee_{i=1}^{m} a_{i}$ and $c \wedge \left( \bigvee_{i=1}^{m} a_{i} \right)$ exist.
Then all the meets $c \wedge a_{i }$ exist as does the join $\bigvee_{i=1}^{m} c \wedge a_{i}$
and we have that
$$
c \wedge \left( \bigvee_{i=1}^{m} a_{i} \right)
= 
\bigvee_{i=1}^{m} c \wedge a_{i}.
$$

\item Suppose that $a$ and $b = \bigvee_{j=1}^{n} b_{j}$ are such that  all the meets $a \wedge b_{j}$ exist.
Then $a \wedge b$ exists and is equal to $\bigvee_{j} a \wedge b_{j}$.

\item Let $a = \bigvee_{i=1}^{m} a_{i}$ and $b = \bigvee_{j=1}^{n} b_{j}$ and suppose that all meets $a_{i} \wedge b_{j}$ exist.
Then $\bigvee_{i,j} a_{i} \wedge b_{j}$ exists as does $a \wedge b$ and we have that $a \wedge b = \bigvee_{i,j} a_{i} \wedge b_{j}$.

\end{enumerate}
\end{lemma}




A distributive inverse semigroup is said to be {\em Boolean} if its semilattice of idempotents is a generalized Boolean algebra.

\begin{examples}
{\em Here are some examples of Boolean inverse semigroups.
\begin{enumerate}

\item Let $X$ be an infinite set. Denote by $\mathcal{I}^{\rm fin}(X)$ the set of all partial bijections of the set $X$
with finite domains. Then this is a generalized Boolean inverse semigroup that is not a Boolean inverse monoid.

\item Symmetric inverse monoids $\mathcal{I}(X)$ are Boolean inverse monoids. Its Boolean algebra of idempotents 
is isomorphic to the power set Boolean algebra $\mathsf{P}(X)$. If $X$ is finite with $n$ elements then we denote the symmetric inverse monoid
on an $n$-element set by $\mathcal{I}_{n}$. 

\item Groups with zero adjoined, denoted by $G^{0}$. These may look like chimeras but they are honest-to-goodness Boolean inverse monoids
whose Boolean algebra of idempotents is isomorphic to the 2-element Boolean algebra.

\item $R_{n}$ the set of all $n \times n$ rook matrices \cite{Solomon}. These are all $n \times n$ matrices over the numbers $0$ and $1$ such that
each row and each column contains at most one non-zero entry.  In fact, $R_{n}$ is isomorphic to $\mathcal{I}_{n}$.

\item $R_{n}(G^{0})$ the set of all  $n \times n$ rook matrices over a group with zero \cite{Malandro}. These are all $n \times n$ matrices over the group with zero $G^{0}$
in which each row and each column contains at most one non-zero entry.

\item Let $S$ be a Boolean inverse semigroup.
Then the set $M_{\omega}(S)$ of all $\omega \times \omega$ generalized rook matrices over $S$ is a Boolean inverse semigroup.
See \cite{KLLR}.

\end{enumerate}
}
\end{examples}



\begin{lemma}\label{lem:complement} Let $S$ be a Boolean inverse semigroup.
Let $b \leq a$.
Then there is a unique element, denoted by $(a \setminus b)$, such that the following properties hold:
$(a \setminus b) \leq a$, 
$b \perp (a \setminus b)$,  
and $a = b \vee (a \setminus b)$.
\end{lemma}
\begin{proof} Observe that there is an order-isomorphism from $a^{\downarrow}$ to $\mathbf{d}(a)^{\downarrow}$ under the map $x \mapsto \mathbf{d}(x)$.
If $b \leq a$ then define $(a \setminus b)$ to be the unique element below $a$ that corresponds to $\mathbf{d}(a) \setminus \mathbf{d}(b)$ under the above order-isomorphism. 
\end{proof}

Let $S$ be a Boolean inverse semigroup.
If $X \subseteq S$ is any non-empty subset, then $X^{\vee}$ denotes the set 
of all binary joins of compatible pairs of elements of $X$.
Clearly, $X \subseteq X^{\vee}$.
A subset $I$ of $S$ is said to be an {\em additive ideal}
if $I$ is a {\em semigroup ideal} of $S$ (that is, $SI \subseteq I$ and $IS \subseteq I$)
and $I$ is closed under binary compatible joins.
If $\theta \colon S \rightarrow T$ is a morphism of Boolean inverse semigroups
then the set $K = \{s \in S \colon \theta (s) = 0\}$ is called the {\em kernel} of $\theta$
and is clearly an additive ideal of $S$.
We say that $S$ is {\em $0$-simplifying} if $S \neq \{0\}$
and the only additive ideals are $\{0\}$ and $S$ itself.
Let $e$ and $f$ be any idempotents.
We say that a non-empty finite set $X = \{x_{1},\ldots, x_{n}\}$ is a {\em pencil from $e$ to $f$}
if $e = \bigvee_{i=1}^{n} \mathbf{d}(x_{i})$ and $\mathbf{r}(x_{i}) \leq f$.
Suppose that $e = 0$.
Then all of the $x_{i} = 0$.
It follows that there is always a pencil from $0$ to any idempotent $f$.
On the other hand if $f = 0$ then $\mathbf{r}(x_{i}) = 0$ and so $x_{i} = 0$ and it follows that $e = 0$.
It is easy to check that if $I$ is an additive ideal and $f \in I$, where $f$ is an idempotent,
and there is a pencil from $e$ to $f$, where $e$ is an idempotent, then $e \in I$.

\begin{lemma}\label{lem:zero-simplifying} 
Let $S$ be a Boolean inverse semigroup not equal to zero.
Then $S$ is $0$-simplifying if and only if whenever $e$ and $f$ are non-zero idempotents
there is a pencil from $e$ to $f$.
\end{lemma}
\begin{proof} Suppose that $S$ is $0$-simplifying, and let $e$ and $f$ be nonzero idempotents.
Then $(SfS)^{\vee}$ is a nonzero additive ideal of $S$.
By assumption $S = (SfS)^{\vee}$.
It follows that $e \in (SfS)^{\vee}$.
It is now routine to check that there is a pencil from $e$ to $f$.
We now prove the converse.
Let $I$ be a nonzero additive ideal of $S$ and let $s \in S$ be arbitrary and non-zero.
Let $f \in I$ be any non-zero idempotent.
Then there is a pencil from $s^{-1}s$ to $f$.
It follows that $s^{-1}s \in I$.
But $I$ is a semigroup ideal and so $s \in I$.
We have proved that $I = S$.
\end{proof}

\begin{remark}
{\em Let $X$ be a pencil from $e$ to $f$ in a Boolean inverse semigroup.
We can always assume, for any distinct $x,y \in X$, 
that $\mathbf{d}(x) \perp \mathbf{d}(y)$.}
\end{remark}

Let $T$ be a Boolean inverse semigroup.
An inverse subsemigroup $S$ of $T$ is said to be a {\em subalgebra}
if $S$ is closed under binary compatible joins taken in $T$
and if $e,f \in \mathsf{E}(S)$ then $e \setminus f \in \mathsf{E}(S)$.
Observe that $S$ is a Boolean inverse semigroup for the induced operations;
observe that $\mathsf{E}(S)$ is a subalgebra of the generalized Boolean algebra $\mathsf{E}(T)$.

\begin{remark}{\em We should note that Wehrung \cite[Definition 3.1.17]{W} uses the term
{\em additive inverse subsemigroup} for what we have termed a subalgebra.}
\end{remark}

\begin{quote}
{\em Our perspective is that Boolean inverse semigroups are non-commutative generalizations of generalized Boolean algebras.}
\end{quote}

\section{Boolean groupoids}

We shall assume that the reader is familiar with the basic ideas and definitions of category theory
as described in the first few chapters of Mac~Lane \cite{Maclane}.
Our goal is just to present the perspective on categories needed in this paper.

A category is usually regarded as a `category of structures' of some kind, such as the category of sets or the category of groups.
A (small) category can, however, also be regarded as an algebraic structure; that is, as a set equipped with a partially defined binary operation satisfying certain axioms.
We shall need both perspectives in this paper, but the latter perspective will be foremost.
This algebraic approach to categories was an important ingredient in Ehresmann's work \cite{E} and applied by Philip Higgins to prove some basic results in group theory \cite{Higgins}.

To define the algebraic notion of a category, we begin with a set $C$ equipped with a partial binary operation which we denote by concatenation.
We write $\exists ab$ to mean that the product $ab$ is defined.
An {\em identity} in such a structure is an element $e$ such that if $\exists ae$ then $ae = a$ and if $\exists ea$ then $ea = a$.
A {\em category} is a set equipped with a partial binary operation satisfying the following axioms:
\begin{description}
\item[{\rm (C1)}] $\exists a(bc)$ if and only if $\exists (ab)c$ and when one is defined so is the other and they are equal.
\item[{\rm (C2)}] $\exists abc$ if and only if $\exists ab$ and $\exists bc$.
\item[{\rm (C3)}] For each $a$ there is an identity $e$, perforce unique, such that $\exists ae$,
and there exists an identity $f$, perforce unique, such that $\exists fa$.
I shall write $\mathbf{d}(a) = e$ and $\mathbf{r}(a) = f$ and draw the picture
$$f \stackrel{a}{\longleftarrow} e.$$
The set of all elements from $e$ to $f$ is called a {\em hom-set} and denoted $\mbox{hom}(e,f)$.
\end{description}
You can check that $\exists ab$ if and only if $\mathbf{d}(a) = \mathbf{r}(b)$.

\begin{example}{\em A category with one identity is a monoid.
Thus, viewed in this light, categories are `monoids with many identities'.}
\end{example}

The morphisms of categories are called {\em functors}; we shall think of functors as generalizations of monoid homomorphisms.

Let $C$ be a category.
If $A,B \subseteq C$ then we may define $AB$ to be that subset of $C$
which consists of all products $ab$ where $a \in A$, $b \in B$ and $ab$ is defined in the category.
We call this {\em subset multiplication}. 

We now define groupoids.
An element $a$ of a category is said to be {\em invertible} if there exists an element $b$ such that $ab$ and $ba$ are identities.
If such an element $b$ exists it is unique and is called the {\em inverse} of $a$;
we denote the inverse of $a$ when it exists by $a^{-1}$.
A category in which every element is invertible is called a {\em groupoid}.

\begin{example}{\em \mbox{}
\begin{enumerate}
\item A groupoid with one identity is a group.
Thus groupoids are `groups with many identities'.

\item A set can be regarded as a groupoid in which every element is an identity.

\item Equivalence relations can be regarded as groupoids.
They correspond to {\em principal groupoids}; that is, those groupoids in which given any identities $e$ and $f$ there is at most one element $g$ of the groupoid
such that $f \stackrel{g}{\longleftarrow} e$.
A special case of such groupoids are the {\em pair groupoids}, $X \times X$, which correspond
to equivalence relations having exactly one equivalence class.


\end{enumerate}
}
\end{example}

If $G$ is a groupoid and $A \subseteq G$ then $A^{-1}$ is the set
of all inverses of elements of $A$. 

We shall need the following notation for the maps involved in defining a groupoid (not entirely standard).
Define $\mathbf{d}(g) = g^{-1}g$ and $\mathbf{r}(g) = gg^{-1}$.
If $g \in G$ define $\mathbf{i}(g) = g^{-1}$.
Put 
$$G \ast G = \{ (g,h) \in G \times G \colon \mathbf{d}(g) = \mathbf{r}(h) \},$$
and define $\mathbf{m} \colon G \ast G \rightarrow G$ by $(g,h) \mapsto gh$.
If $U,V \subseteq G$, define $U \ast V = (U \times V) \cap (G \ast G)$.
The set of identities of $G$ is denoted by $G_{o}$.
If $e$ is an identity in $G$ then $G_{e}$ is the set of all elements $a$ such that
$a^{-1}a = e = aa^{-1}$.
We call this the {\em local group at $e$}.
Put $\mbox{Iso}(G) = \bigcup_{e \in G_{o}} G_{e}$.
This is called the {\em isotropy groupoid} of $G$.

We now show how to construct all groupoids.
Let $G$ be a groupoid.
We say that elements $g,h \in G$ are {\em connected}, denoted $g \equiv h$, if there is an element $x \in G$ such that
$\mathbf{d}(x) = \mathbf{d}(h)$ and $\mathbf{r}(x) = \mathbf{d}(g)$.
The $\equiv$-equivalence classes are called the {\em connected components of the groupoid}.
If $\exists gh$ then necessarily $g \equiv h$.
It follows that $G = \coprod_{i \in I} G_{i}$ where the $G_{i}$ are the connected components of $G$.
Each $G_{i}$ is a connected groupoid.
So, it remains to describe the structure of all connected groupoids.
Let $X$ be a non-empty set and let $H$ be a group.
The set of triples $X \times H \times X$ becomes a groupoid when we define
$(x,h,x')(x',h',x'') = (x,hh',x'')$ and $(x,h,y)^{-1} = (y,h^{-1},x)$.
It is easy to check that $X \times H \times X$ is a connected groupoid.
Now let $G$ be an arbitrary groupoid.
Choose, and fix, an identity $e$ in $G$.
Denote the local group at $e$ by $H$.
For each identity $f$ in $G$ choose an element $x_{f}$ such that $\mathbf{d}(x_{f}) = e$ and $\mathbf{r}(x_{f}) = f$.
Put $X = \{x_{f} \colon f \in G_{o} \}$.
We prove that $G$ is isomorphic to $X \times H \times X$.
Let $g \in G$.
Then $x_{\mathbf{r}(g)}^{-1}gx_{\mathbf{d}(g)} \in H$.
Define a map from $G$ to $X \times H \times X$ by $g \mapsto (x_{\mathbf{r}(g)}, x_{\mathbf{r}(g)}^{-1}gx_{\mathbf{d}(g)}, x_{\mathbf{d}(g)})$.
It is easy to show that this is a bijective functor.

We shall need some special kinds of functors in our duality theory called
covering functors.
These we define now.
Let $G$ be any groupoid and $e$ any identity.
The {\em star of $e$}, denoted by $\mbox{St}_{e}$, 
consists of all elements $g \in G$ such that $\mathbf{d}(g) = e$.
Let $\theta \colon G \rightarrow H$ be a functor between groupoids.
Then for each identity $e \in G$, the functor $\theta$ induces a function
$\theta_{e}$ mapping $\mbox{St}_{e}$ to $\mbox{St}_{\theta (e)}$.
If all these functions are injective (respectively, surjective)
then we say that $\theta$ is {\em star-injective} (respectively, {\em star-surjective}).
A {\em covering functor} is a functor which is star-bijective.
The following was proved as \cite[Lemma 2.26]{Lawson3}.

\begin{lemma}\label{lem:former} Let $\theta \colon G \rightarrow H$ be a covering functor
between groupoids. Suppose that the product $ab$ is defined in $H$ and that $\theta (x) = ab$.
Then there the exist $u,v \in G$ such that $x = uv$ and $\theta (u) = a$ and $\theta (v) = b$.
\end{lemma}  

The key definition needed to relate groupoids and inverse semigroups in our non-commutative generalization of Stone duality is the following.
A subset $A \subseteq G$ is called a {\em local bisection} if $A^{-1}A,AA^{-1} \subseteq G_{o}$.

\begin{lemma}\label{lem:water} 
A subset $A \subseteq G$ is a local bisection if and only if $a,b \in A$ and $\mathbf{d}(a) = \mathbf{d}(b)$ implies that $a = b$,
and $\mathbf{r}(a) = \mathbf{r}(b)$ implies that $a = b$.
\end{lemma}
\begin{proof} Suppose that $A$ is a local bisection.
Let $a,b \in A$ such that $\mathbf{d}(a) = \mathbf{d}(b)$.
Then the product $ab^{-1}$ exists and, by assumption, is an identity.
It follows that $a = b$.
A similar argument shows that if $a,b \in A$ are such that $\mathbf{r}(a) = \mathbf{r}(b)$
then $a = b$.
We now prove the converse.
We prove that $A^{-1}A \subseteq G_{o}$.
Let $a,b \in A$ and suppose that $a^{-1}b$ is exists.
Then $\mathbf{r}(a) = \mathbf{r}(b)$.
By assumption, $a = b$ and so $a^{-1}b$ is an identity.
The proof that $AA^{-1} \subseteq G_{o}$ is similar.
\end{proof}

What follows is based on \cite{Resende, Sims}.

Just as we can study topological groups, so we can study topological groupoids.
A {\em topological groupoid} is a groupoid $G$ equipped with a topology, and $G_{o}$ is equipped with the subspace topology, 
such that the maps $\mathbf{d}, \mathbf{r}, \mathbf{m}, \mathbf{i}$ are all continuous functions
where $\mathbf{d}, \mathbf{r} \colon G \rightarrow G$ and $\mathbf{m} \colon G \ast G \rightarrow G$.
Clearly, it is just enough to require that $\mathbf{m}$ and $\mathbf{i}$ are continuous.

A topological groupoid is said to be {\em open} if the map $\mathbf{d}$ is an open map;
it is said to be {\em \'etale} if the map $\mathbf{d}$ is a local homeomorphism.

\begin{remark}{\em It is worth observing (see \cite[page~12]{Sims}) that
the definition of \'etale is based on the function $\mathbf{d} \colon G \rightarrow G$.}
\end{remark}





In this paper, we shall focus on \'etale topological groupoids.
The obvious question is why should \'etale groupoids be regarded as a nice class of topological groupoids?
The following result due to Pedro Resende provides us with one reason.
For the following see \cite[Exercises I.1.8]{Resende}.

\begin{lemma}\label{lem:rally} Let $G$ be a topological groupoid.
Then $G$ is \'etale if and only if $\Omega (G)$, the set of all open subsets of $G$, is a monoid under subset multiplication with $G_{o}$ as the identity.
\end{lemma}

We may paraphrase the above theorem by saying that \'etale groupoids are those topological groupoids that have an algebraic alter ego.

We say that an \'etale topological groupoid is {\em Boolean} if its space of identities is a locally compact Boolean space.
\begin{quote}
{\em Our perspective is that Boolean groupoids are `non-commutative' generalizations of locally compact Boolean spaces.}
\end{quote}
Thinking of topological groupoids as non-commutative spaces in this way goes back to \cite{K, Renault}.

\section{From Boolean groupoids to Boolean inverse semigroups}

The goal of this section is to show how to construct Boolean inverse semigroups from Boolean groupoids.

\begin{lemma}\label{lem:tea} 
The set of all local bisections of a discrete groupoid forms a Boolean inverse monoid under subset multiplication.
\end{lemma}
\begin{proof} Let $A$ and $B$ be local bisections.
We prove that $AB$ is a local bisection.
We calculate $(AB)^{-1}AB$.
This is equal to $B^{-1}A^{-1}AB$.
Now $A^{-1}A$ is a set of identities.
Thus $B^{-1}A^{-1}AB \subseteq B^{-1}B$.
But $B^{-1}B$ is a set of identities.
It follows that $(AB)^{-1}AB$ is a set of identities.
By a similar argument we deduce that $AB(AB)^{-1}$ is a set of identities.
We have therefore proved that the product of two local bisections is a local bisection.
The proof of associativity is straightforward.
Since $G_{o}$ is a local bisection, we have proved that the set of local bisections is a monoid.
Observe that if $A$ is a local bisection, then  $A = AA^{-1}A$ and $A = A^{-1}AA^{-1}$.
Thus the semigroup is regular.
Suppose that $A^{2} = A$, where $A$ is a local bisection.
Then $a = bc$ where $b,c \in A$.
But $\mathbf{d}(a) = \mathbf{d}(c)$, and so $a = c$,
and $\mathbf{r}(a) = \mathbf{r}(b)$, and so $a = b$.
It follows that $a = a^{2}$.
But the only idempotents in groupoids are identities and so $a$ is an identity.
We have shown that if $A^{2} = A$ then $A \subseteq G_{o}$.
It is clear that if $A \subseteq G_{o}$ then $A^{2} = A$.
We have therefore proved that the idempotent local bisections are precisely the subsets of
the set of identities.
The product of any two such idempotents is simply their intersection
and so idempotents commute with each other.
It follows that our monoid is inverse.
It is easy to check that $A \leq B$ in this inverse semigroup precisely when
$A \subseteq B$.
Now, the idempotents are the subsets of the set of identities
and the natural partial order is subset inclusion.
It follows that the set of identities is a Boolean algebra,
since it is isomorphic to the Boolean algebra of all subsets of $G_{o}$.
Suppose that $A$ and $B$ are local bisections such that $A \sim B$.
Then it is easy to check that $A \cup B$ is a local bisection.
Clearly, subset multiplication distributes over such unions.
We have therefore proved that the monoid is a Boolean inverse monoid.
\end{proof}

A subset $A \subseteq G$ of a groupoid is called a {\em bisection} if
$$A^{-1}A, AA^{-1} = G_{o}.$$
The following is immediate by Proposition~\ref{lem:tea} and tells us that we may also construct groups from groupoids.

\begin{corollary} The set of bisections of a discrete groupoid is just the group of units of the inverse monoid
of all local bisections of the discrete groupoid.
\end{corollary}

\noindent
{\bf Definition. } Let $G$ be a Boolean groupoid.
Denote by $\mathsf{KB}(G)$ the set of all compact-open local bisections of $G$.\\

\begin{proposition}\label{prop:bis-from-boolean-groupoids}
Let $G$ be a Boolean groupoid.
Then $\mathsf{KB}(G)$ is a Boolean inverse semigroup.
\end{proposition}
\begin{proof} Let $U$ and $V$ be two compact-open local bisections.
Since the groupoid $G$ is \'etale, the product $UV$ is open by Lemma~\ref{lem:rally}.
The product of local bisections is a local bisection by the proof of Lemma~\ref{lem:tea}.
It remains to show that $UV$ is compact.
Let $UV \subseteq \bigcup_{i \in I} A_{i}$ where the $A_{i}$ are open local bisections
--- since the open local bisections of an \'etale groupoids form a base for the topology.
Then $U^{-1}U \cap VV^{-1} = U^{-1}UVV^{-1} \subseteq  \bigcup_{i \in I} U^{-1}A_{i}V^{-1}$.
The sets $U^{-1}A_{i}V^{-1}$ are open local bisections.
By assumption $U^{-1}U \cap VV^{-1}$ is compact-open;
here we use the fact that the identity space of a Boolean groupoid is a locally compact Boolean space.
Thus  $U^{-1}U \cap VV^{-1} = U^{-1}UVV^{-1} \subseteq  \bigcup_{i = 1}^{n} U^{-1}A_{i}V^{-1}$,
relabelling if necessary.
It follows that 
$UV \subseteq  \bigcup_{i = 1}^{n} UU^{-1}A_{i}V^{-1}V \subseteq       \bigcup_{i = 1}^{n} A_{i}$
and so $UV$ is compact.
It follows that $\mathsf{KB}(G)$ is a semigroup.
The proof that it is a Boolean inverse semigroup now follows easily from what we have done above 
and Lemma~\ref{lem:tea}.
\end{proof}

\section{From Boolean inverse semigroups to Boolean groupoids}

The goal of this section is to show how to construct Boolean groupoids from Boolean inverse semigroups.

Let $S$ be an inverse semigroup.
A non-empty subset $A \subseteq S$ is said to be a (respectively, {\em proper}) {\em filter} 
(respectively, if $0 \notin A$),
if the set $A$ is downwardly directed and upwardly closed.
A maximal proper filter is called an {\em ultrafilter}.
Let $A$ be a proper filter in a distributive inverse semigroup $S$.
We say it is {\em prime} if $a \vee b \in A$ implies that $a \in A$ or $b \in A$.
The proof of the following is straightforward.

\begin{lemma}\label{lem:minuses} Let $S$ be a Boolean inverse semigroup.
Then $A$ is an ultrafilter if and only if  $A^{-1}$ is an ultrafilter.
\end{lemma}

A subset $A$ of an inverse semigroup $S$ is said to be a {\em coset} if $a,b,c \in A$ implies that $ab^{-1}c \in A$.
The following extends \cite[Lemma~2.6]{Lawson3}.

\begin{lemma}\label{lem:coset} 
Every filter is a coset.
\end{lemma}
\begin{proof} Let $A$ be a filter and let $a,b,c \in A$.
Let $d \in A$ where $d \leq a,b,c$.
Then $d = dd^{-1}d \leq ab^{-1}c$.
It follows that $ab^{-1}c \in A$.
\end{proof}

We define an {\em idempotent filter} to be a filter that contains an idempotent.
The following was proved for cosets as \cite[Proposition 1.5]{Lawson1996} and so follows by Lemma~\ref{lem:coset}.

\begin{lemma}\label{lem:idempotent-filter-is} 
A filter is idempotent if and only if it is an inverse subsemigroup.
\end{lemma}

We now relate idempotent filters in $S$ to filters in $\mathsf{E}(S)$.

\begin{lemma}\label{lem:idempotent-filter} Let $S$ be an inverse semigroup.
There is an order-isomorphism between the idempotent filters in $S$ and the filters in $\mathsf{E}(S)$,
in which proper filters correspond to proper filters.
\end{lemma}
\begin{proof}  Let $A$ be an idempotent filter in $S$.
We prove first that $\mathsf{E}(A)$ is a filter in $\mathsf{E}(S)$. 
Let $e,f \in \mathsf{E}(A)$.
By assumption, $e,f \in A$.
Thus there is an element $i \in A$ such that $i \leq e,f$.
But the set of idempotents of an inverse semigroup is an order-ideal.
It follows that $i$ is an idempotent.
But $i \leq e,f$ and so $i \leq ef$.
However, $A$ is a filter and so since $i \in A$ we have that $ef \in A$ and so $ef \in \mathsf{E}(A)$.
Let $e \in \mathsf{E}(A)$ and let $e \leq f$ where $f$ is an idempotent.
Then $f \in A$ and so $f \in \mathsf{E}(A)$.
We have therefore proved that $\mathsf{E}(A)$ is a filter in $\mathsf{E}(S)$.
It is clear that if $A \subseteq B$ then $\mathsf{E}(A) \subseteq \mathsf{E}(B)$.
Now, let $F$ be a filter in $\mathsf{E}(S)$ is is easy to check that $F^{\uparrow}$ is 
an idempotent filter in $S$.
It is clear that if $F \subseteq G$ then $F^{\uparrow} \subseteq G^{\uparrow}$.
Let $A$ be an idempotent filter in $S$.
We prove that $A = \mathsf{E}(A)^{\uparrow}$.  
It is clear that $\mathsf{E}(A)^{\uparrow} \subseteq A$.
Let $a \in A$.
By assumption $e \in A$ for some idempotent $e$.
But $a,e \in A$.
There exists $f \leq a,e$ which has to be an idempotent since the set of idempotents of an inverse
semigroup is an order-ideal.
It follows, in particular, that $f \leq a$,
which proves the claim.
Let $F$ be a filter in $\mathsf{E}(S)$.
It is now routine to check that $F = \mathsf{E}(F^{\uparrow})$.
We have therefore proved our order-isomorphism and it is clear that 
proper filters map to proper filters.
\end{proof}

\begin{lemma}\label{lem:relating} Let $A$ be an idempotent filter in a distributive inverse semigroup $S$.
\begin{enumerate}
\item $A$ is a prime filter in $S$ if and only if $\mathsf{E}(A)$ is a prime filter in $\mathsf{E}(S)$. 
\item $A$ is an ultrafilter in $S$ if and only if $\mathsf{E}(A)$ is an ultrafilter in $\mathsf{E}(S)$. 
\end{enumerate}
\end{lemma}
\begin{proof} (1) Suppose first that $A$ is a prime filter in $S$.
We saw in Lemma~\ref{lem:idempotent-filter}, that $\mathsf{E}(A)$ is a proper filter in $\mathsf{E}(S)$.
Suppose that $e \vee f \in \mathsf{E}(A)$.
Then $e \vee f \in A$.
It follows that $e \in A$ or $f \in A$.
It is now immediate that $\mathsf{E}(A)$ is a prime filter in $\mathsf{E}(S)$.
We now prove the converse.
Suppose that  $\mathsf{E}(A)$ is a prime filter in $\mathsf{E}(S)$.
We prove that $A$ is a prime filter in $S$.
Let $a \vee b \in A$.
Then $\mathbf{d} (a \vee b) = \mathbf{d}(a) \vee \mathbf{d}(b) \in \mathsf{E}(A)$.
It follows that $\mathbf{d}(a) \in \mathsf{E}(A)$ or $\mathbf{d}(b) \in \mathsf{E}(A)$.
Without loss of generality, we suppose the former.
So, $a \vee b \in A$ and $\mathbf{d}(a) \in A$.
But $A$ is an inverse subsemigroup and so $a = (a \vee b)\mathbf{d}(a) \in A$.
This proves that $A$ is a prime filter in $S$.

(2) Suppose that $A$ is an ultrafilter in $S$.
We prove that $\mathsf{E}(A)$ is an ultrafilter in $\mathsf{E}(S)$.
Suppose that $F$ is a proper filter in  $\mathsf{E}(S)$ such that $\mathsf{E}(A) \subseteq F$.
We claim that $F^{\uparrow}$ is a proper filter in $A$.
Let $a, b \in F^{\uparrow}$.
The  $e \leq a$ and $f \leq b$ where $e,f \in F$.
But then $ef \leq a,b$ and $ef \in F$.
The set $F^{\uparrow}$ is clearly upwardly closed and evidently proper.
But $A \subseteq F^{\uparrow}$ and $A$ is a maximal proper filter.
We deduce that $A = F^{\uparrow}$.
It is now immediate that $\mathsf{E}(A) = F$.
We have therefore proved that $\mathsf{E}(A)$ is an ultrafilter in $\mathsf{E}(S)$.
We now prove the converse.
Suppose that $\mathsf{E}(A)$ is an ultrafilter in $\mathsf{E}(S)$.
We prove that $A$ is an ultrafilter in $S$.
Suppose that $A \subseteq B$ where $B$ is a proper filter in $S$.
Since $A$ is an idempotent filter so too is $B$.
We therefore have that $A = \mathsf{E}(A)^{\uparrow} \subseteq \mathsf{E}(B)^{\uparrow} = B$.
But $\mathsf{E}(A) \subseteq \mathsf{E}(B)$ and, by assumption, $\mathsf{E}(A)$ is an ultrafilter in $\mathsf{E}(S)$.
It follows that $\mathsf{E}(A) = \mathsf{E}(B)$ from which we deduce that $A = B$. 
\end{proof}

\begin{lemma}\label{lem:d-for-a-filter} Let $A$ be a (respectively, proper) filter.
Then $(A^{-1}A)^{\uparrow}$ is a (respectively, proper) idempotent filter.
Likewise, $(AA^{-1})^{\uparrow}$ is a filter.
\end{lemma}
\begin{proof} We prove that $(A^{-1}A)^{\uparrow}$ is a filter.
Let $x,y \in (A^{-1}A)^{\uparrow}$.
Then $a^{-1}b \leq x$ and $c^{-1}d \leq y$ where $a,b,c,d \in A$.
Let $z \leq a,b,c,d$ where $z \in A$.
Then $z^{-1}z \leq x,y$.
It is clear that $(A^{-1}A)^{\uparrow}$ is upwardly closed.
It is also clear that if $A$ is a proper filter so too is $(A^{-1}A)^{\uparrow}$.
The fact that $(A^{-1}A)^{\uparrow}$ is an idempotent filter is immediate.
\end{proof}

Let $A$ be a (respectively, proper) filter.
Then $A^{-1}$ is a (respectively, proper) filter.
If $A$ is a (proper) filter define 
$\mathbf{d}(A) = (A^{-1}A)^{\uparrow}$ 
and
$\mathbf{r}(A) = (AA^{-1})^{\uparrow}$.
By Lemma~\ref{lem:d-for-a-filter}, these are both (proper) idempotent filters.
In fact, we have the following which is easy to prove using Lemma~\ref{lem:idempotent-filter-is}. 

\begin{lemma}\label{lem:identities-in-bis} Let $S$ be a Boolean inverse semigroup.
If $A$ is an idempotent ultrafilter then $\mathbf{d}(A) = A$ and $\mathbf{r}(A) = A$. 
\end{lemma}

We can now explain why we have used the term `coset'.

\begin{lemma}\label{lem:cosets} Let $A$ be a filter in an inverse semigroup.
Then $A = (a \mathbf{d}(A))^{\uparrow}$ and $A = (\mathbf{r}(A)a)^{\uparrow}$,
where $a \in A$. 
\end{lemma}
\begin{proof} We prove that $A = (a \mathbf{d}(A))^{\uparrow}$.
Let $x \in A$.
Then, since $a \in A$, there exists $b \in A$ such that $b \leq x,a$.
Observe that $ab^{-1}b \leq xx^{-1}x = x$ and $ab^{-1}b \in a\mathbf{d}(A)$.
Thus $A \subseteq (a \mathbf{d}(A))^{\uparrow}$.
To prove the reverse inclusion let $x \in (a \mathbf{d}(A))^{\uparrow}$.
Then $ab^{-1}c \leq x$ where $a,b,c \in A$.
But we proved above that $A$ was a coset and so $ab^{-1}c \in A$.
It follows that $x \in A$.
\end{proof}

\begin{lemma}\label{lem:d-is-extremal} Let $A$ be a filter in a distributive inverse semigroup.
\begin{enumerate}
\item $A$ is a prime filter if and only if $\mathbf{d}(A)$ is a prime filter, and dually.
\item $A$ is an ultrafilter if and only if $\mathbf{d}(A)$ is an ultrafilter, and dually.
 \end{enumerate}
\end{lemma}
\begin{proof} (1) Suppose that $A$ is a prime filter. We prove that $\mathbf{d}(A)$ is a prime filter.
Let $x \vee y \in \mathbf{d}(A)$.
Then $a^{-1}b \leq x \vee y$.
It follows that $aa^{-1}b \leq a(x \vee y)$.
But $aa^{-1}b \in A$ since $A$ is a coset.
Thus $ax \vee ay \in A$.
Without loss of generality, suppose that $ax \in A$.
Then $a^{-1}(ax)  \in A^{-1}A$ and so $x \in \mathbf{d}(A)$.
We have therefore proved that $\mathbf{d}(A)$ is a prime filter.
Suppose now that $\mathbf{d}(A)$ is a prime filter.
We prove that $A$ is a prime filter.
Let $x \vee y \in A$.
Then $\mathbf{d}(x \vee y) \in \mathbf{d}(A)$.
Without loss of generality, suppose that $\mathbf{d}(x) \in \mathbf{d}(A)$.
It follows that $a^{-1}b \leq \mathbf{d}(x)$ where $a,b \in A$.
Thus $(x \vee y)a^{-1}b \leq (x \vee y)\mathbf{d}(x) = x$ where $(x \vee y)a^{-1}b \in A$ since $A$ is a coset.
We have proved that $x \in A$.

(2) Suppose that $A$ is an ultrafilter.
We prove that $\mathbf{d}(A)$ is an ultrafilter.
Suppose that $\mathbf{d}(A) \subseteq B$ where $B$ is a proper filter in $S$.
Observe that $B$ is an idempotent filter.
Let $a \in A$.
Then $A = (a\mathbf{d}(A))^{\uparrow}$.
It follows that $A \subseteq (aB)^{\uparrow}$.
By assumption $A = (aB)^{\uparrow}$.
Thus $\mathbf{d}(A) = B$, as required.
Suppose now that  $\mathbf{d}(A)$ is an ultrafilter.
We prove that $A$ is an ultrafilter.
Suppose that $A \subseteq B$ where $B$ is a proper filter.
Then $\mathbf{d}(A) \subseteq \mathbf{d}(B)$.
By assumption $\mathbf{d}(A) = \mathbf{d}(B)$
from which it follows that $A = B$.
\end{proof}

The following result is important.
In particular, prime filters are easier to work with than ultrafilters.

\begin{lemma}\label{lem:prime-equals-uf} In a Boolean inverse semigroup,
prime filters are the same as ultrafilters.
\end{lemma}
\begin{proof} The proper filter $A$ is prime if and only if $\mathbf{d}(A)$ is prime by Lemma~\ref{lem:d-is-extremal}.
By Lemma~\ref{lem:relating}, the proper filter $\mathbf{d}(A)$ is prime if and only if $\mathsf{E}(\mathbf{d}(A))$ is a prime 
filter in the generalized Boolean algebra $\mathsf{E}(S)$.
But in a generalized Boolean algebra, byLemma~\ref{lem:ll}, prime filters are the same as ultrafilters.
We can now work our way backwards to etablish the claim.
\end{proof}

\noindent
{\bf Definition. }Let $A$ and $B$ be ultrafilters.
Define $A \cdot B$ precisely when $\mathbf{d}(A) = \mathbf{r}(B)$.
In which case, $A \cdot B = (AB)^{\uparrow}$.
Observe that $\mathbf{d}(A^{-1}) = \mathbf{r}(A)$.
It follows that $A^{-1} \cdot A = \mathbf{d}(A)$
and $A \cdot A^{-1} = \mathbf{r}(A)$.\\

\begin{lemma}\label{lem:product} Let $S$ be a Boolean inverse semigroup. 
\begin{enumerate}
\item If $A$ is an ultrafilter then both $\mathbf{d}(A)$ and $\mathbf{r}(A)$ are ultrafilters.
\item If $A$ and $B$ are ultrafilters of $S$ such that $A \cdot B$ is defined, 
then $A \cdot B$ is an ultrafilter of $S$ such that 
$\mathbf{d}(A \cdot B) = \mathbf{d}(B)$
and
$\mathbf{r}(A \cdot B) = \mathbf{r}(A)$.
\item If $A$ is an ultrafilter in $S$ then $A \cdot \mathbf{d}(A)$ is defined
and $A \cdot \mathbf{d}(A) = A$.
Similarly, $\mathbf{r}(A) \cdot A = A$.
\item $(A \cdot B) \cdot C = A \cdot (B \cdot C)$
when the product are defined.
\end{enumerate}
\end{lemma}
\begin{proof} 
(1) This follows by Lemma~\ref{lem:d-is-extremal}.

(2) We prove first that if $A$ and $B$ are proper filters then $A \cdot B$ is a proper filter.
Let $x, y \in A \cdot B$.
Then $ab \leq x$ and $cd \leq y$ where $a \in A$, $b \in B$, $c \in A$ and $d \in B$.
Let $u \leq a,c$ and $u \in A$, and let $v \leq b,d$ and $v \in B$.
Then $uv \leq ab \leq x$ and $uv \leq cd \leq y$.
Observe that $uv \in AB$ and so we have proved that $(AB)^{\uparrow}$ is downwardly directed.
It is clearly upwardly directed.
Suppose that $0 \in (AB)^{\uparrow}$.
Then $ab = 0$ for some $a \in A$ and $b \in B$.
But $(A^{-1}A)^{\uparrow} = (BB^{-1})^{\uparrow}$.
It follows that $a^{-1}a \in (BB^{-1})^{\uparrow}$.
Thus $bb^{-1} \leq a^{-1}a$ where $b \in B$.
It follows that $0 = ab =a(aa^{-1})b \geq  bb^{-1}b = b$.
It follows that $b = 0$ which contradicts the assumption that $B$ is a proper filter.
It is routine to prove that $\mathbf{d}(A \cdot B) = \mathbf{d}(B)$.
The result now follows by  Lemma~\ref{lem:d-is-extremal}.

(3) This is straightforward on the basis of Lemma~\ref{lem:coset}.

(4) This follows by (2) above and Lemma~\ref{lem:cosets}.
\end{proof}

\noindent
{\bf Definition. }
Let $S$ be a Boolean inverse semigroup. 
Denote by $\mathsf{G}(S)$ the set of prime filters on $S$.\\

Using the identification between prime filters and ultrafilters proved in
Lemma~\ref{lem:prime-equals-uf} together with Lemma~\ref{lem:product}, Lemma~\ref{lem:minuses} and Lemma~\ref{lem:identities-in-bis},
we have proved the following.

\begin{lemma}\label{lem:groupoid-bis} Let $S$ be a Boolean inverse semigroup.
Then $\mathsf{G}(S)$ is a groupoid with respect to the partially defined operation $\cdot$,
where the identities of this groupoid are precisely the idempotent ultrafilters.
\end{lemma}

Let $S$ be a Boolean inverse semigroup.
Let $a \in S$.
Denote by $\mathscr{U}_{a}$ the set of all primefilters in $S$ that contain $a$.
Observe that if $a = 0$ then $\mathscr{U}_{0} = \varnothing$.

\begin{lemma}\label{lem:primefilters-form-a-base}
Let $S$ be a Boolean inverse semigroup.
\begin{enumerate}
\item Each non-zero element of $S$ is contained in a prime filter.
Thus $\mathscr{U}_{a} \neq \varnothing$ if and only if $a \neq 0$.
\item $\left( \mathscr{U}_{a} \right)^{-1} = \mathscr{U}_{a^{-1}}$.
\item Let $A \in \mathscr{U}_{a} \cap \mathscr{U}_{b}$.
Then there exists $c \leq a,b$ such that $A \in \mathscr{U}_{c} \subseteq \mathscr{U}_{a} \cap \mathscr{U}_{b}$.
\item If $a \sim b$ then $\mathscr{U}_{a \vee b} = \mathscr{U}_{a} \cup \mathscr{U}_{b}$.
\end{enumerate}
\end{lemma}
\begin{proof}
(1) Let $a \neq 0$.
Then $\mathbf{d}(a) \neq 0$.
Let $F$ be any prime filter in $\mathsf{E}(S)$ that contains $\mathbf{d}(a)$.
Then $F^{\uparrow}$ is an idempotent prime filter in $S$.
Thus $A = (aF^{\uparrow})^{\uparrow}$ is a prime filter in $S$ that contains $a$.

(2) Straightforward.

(3) Let $A \in \mathscr{U}_{a} \cap \mathscr{U}_{b}$.  
It follows that $a,b \in A$.
But $A$ is a filter and so there is $c \in A$ such that $c \leq a,b$.
It follows that $A \in \mathscr{U}_{c}$ and that $\mathscr{U}_{c} \subseteq \mathscr{U}_{a} \cap \mathscr{U}_{b}$.

(4) We suppose that $a \sim b$ and so $a \vee b $ exists.
The inclusion $\mathscr{U}_{a} \cup \mathscr{U}_{b} \subseteq \mathscr{U}_{a \vee b}$ is immediate.
The reverse inclusion follows from the fact that ultrafilters are the same as prime filters in
a Boolean inverse semigroup by Lemma~\ref{lem:prime-equals-uf} 
\end{proof}

By Lemma~\ref{lem:primefilters-form-a-base}, the collection of all sets $\mathscr{U}_{a}$, where $a \in S$,
forms a base for a topology on $\mathsf{G}(S)$.
Denote by $\mathsf{F}(S)$ the set of all idempotent prime filters.
This can be topologized by giving it the subspace topology.
Thus the sets of the form $\mathscr{U}_{a} \cap \mathsf{F}(S)$ form a base for the topology on $\mathsf{F}(S)$.
Observe that 
$$\mathscr{U}_{a} \cap \mathsf{F}(S) = \bigcup_{f \leq a, f^{2} = f} \mathscr{U}_{f}.$$
It follows by the above observation and Lemma~\ref{lem:idempotent-filter-is}
that the collection of sets $\mathscr{U}_{e}$, where $e$ is an idempotent, forms a base for the subspace
topology on $\mathsf{F}(S)$.
By Lemma~\ref{lem:idempotent-filter} and Lemma~\ref{lem:prime-equals-uf},
we see that there is a bijection between the set $\mathsf{F}(S)$ and the set of prime filters of $\mathsf{E}(S)$.
If $e \in \mathsf{E}(S)$ then we denote the set of ultrafilters of $\mathsf{E}(S)$
that contain $e$ by $\mathscr{V}_{e}$.
The bijection above restricts to a bijection between $\mathscr{U}_{e}$ and $\mathscr{V}_{e}$
for each idempotent $e$.
We have therefore proved the following.

\begin{lemma}\label{lem:topology-idempotent} Let $S$ be a Boolean inverse semigroup.
Then the topological space of idempotent prime filters is homeomorphic to the Stone space of $\mathsf{E}(S)$.
\end{lemma}

\begin{lemma}\label{lem:local-bisection} Let $A$ and $B$ be filters such that $A \cap B \neq \varnothing$
and $\mathbf{d}(A) = \mathbf{d}(B)$ (respectively, $\mathbf{r}(A) = \mathbf{r}(B)$). Then $A = B$.
\end{lemma}
\begin{proof} Let $a \in A \cap B$.
Put $\mathbf{d}(A) = C$.
Then $A = (aC)^{\uparrow} = B$.
The proof of the other case is similar.
\end{proof}

\begin{proposition}\label{prop:from-bis-to-boolean-groupoid} Let $S$ be a Boolean inverse semigroup.
Then $\mathsf{G}(S)$ is a Boolean groupoid.
\end{proposition} 
\begin{proof} First we show that $\mathsf{G}(S)$ is a topological groupoid.
By part (2) of Lemma~\ref{lem:primefilters-form-a-base}, the inversion map is continuous.
We observe that
$$\mathbf{m}^{-1}(\mathscr{U}_{s}) = \left( \bigcup_{0 \neq ab \leq s} \mathscr{U}_{a} \times \mathscr{U}_{b} \right) \cap (\mathsf{G} (S) \ast \mathsf{G} (S))$$
for all $s \in S$.
The proof is straightforward and the same as Step~3 of the proof of \cite[Proposition~2.22]{Lawson3}
and shows that $\mathbf{m}$ is a continuous function.

We show that $\mathsf{G}(S)$ is \'etale.
It is enough to show that the map from $\mathscr{U}_{a}$ to $\mathscr{U}_{\mathbf{d}(a)}$
given by $A \mapsto \mathbf{d}(A)$ is a homeomorphism.
The proof of this is the same as the proof of Step~4 of the proof of \cite[Proposition~2.22]{Lawson3}.

The fact that the identity space of $\mathsf{G}(S)$ is homeomorphic to
the Stone space of $\mathsf{E}(S)$ follows by Lemma~\ref{lem:topology-idempotent}.
This tells us that our \'etale topological groupoid is a Boolean groupoid.
\end{proof}

\noindent
{\bf Definition. }If $S$ is a Boolean inverse semigroup, then we refer to $\mathsf{G}(S)$ as the {\em Stone groupoid} of $S$.\\

\section{Non-commutative Stone duality}

In this section, we shall generalize Theorem~\ref{them:classical-stone-dualityII}
by replacing generalized Boolean algebras by Boolean inverse semigroups,
and locally compact Boolean spaces by Boolean groupoids.

\subsection{Properties of prime filters}

Our first goal now is to prove that we have {\em enough} ultrafilters in a Boolean inverse semigroup.
We adapt to our setting the proofs to be found in \cite[Chapter I, Section 2]{Johnstone}.
Any proofs that are omitted can be found in \cite{LL}.
Let $S$ be a distibutive inverse semigroup.
An order-ideal of $S$ closed under binary joins is called an {\em additive order-ideal}.
If $A$ is an order-ideal then $A^{\vee}$ denotes the set of all binary joins of compatible pairs of elements of $A$.
If $A$ is an order-ideal then $A^{\vee}$ is an additive order-ideal containing $A$.
We say that an additive order-ideal $A$ is {\em prime} if $a^{\downarrow} \cap b^{\downarrow} \subseteq A$ 
implies that $a \in A$ or $b \in A$.
The proof of the following is straightforward or can be found as
\cite[Lemma 3.10]{LL}.

\begin{lemma}\label{lem:prime-filter} Let $S$ be a distributive inverse semigroup.
Then $A$ is a prime filter if and only if
$S \setminus A$ is a prime additive order-ideal.
\end{lemma}

The following can be proved using Zorn's Lemma.

\begin{lemma}\label{lem:ZL} Let $S$ be a distributive inverse semigroup.
Let $I$ be an additive order-ideal of $S$ and let $F$ be a filter disjoint from $I$.
Then there is an additive order-ideal $J$ maximal with respect to two properties:
$I \subseteq J$ and $J \cap F = \varnothing$.
\end{lemma}

The proof of the following is straightforward.

\begin{lemma} Let $S$ be a distributive inverse semigroup.
Let $I$ be an additive order-ideal and let $a$ be an arbitrary element of $S$.
Then $I \cup a^{\downarrow}$ is an order-ideal
and
$$\left( I \cup a^{\downarrow} \right)^{\vee}
=
\{x \vee b \colon x \in I, b \leq a, x \sim b \}.$$ 
\end{lemma}

\begin{lemma}\label{lem:prime} Let $S$ be a distributive inverse semigroup.
Let $F$ be a filter in $S$ and let $J$ be an additive order-ideal
maximal amongst all additive order-ideals disjoint from $F$.
Then $J$ is a prime additive order-deal.
\end{lemma}

The following lemma is crucial to our program.

\begin{lemma}\label{lem:crucial} Let $S$ be a Boolean inverse semigroup.
Let $a,b \in S$ such that $b \nleq a$.
Then there is an ultrafilter that contains $b$ but omits $a$.
\end{lemma} 
\begin{proof} By assumption, $b^{\uparrow} \cap a^{\downarrow} = \varnothing$.
Now $b^{\uparrow}$ is a filter and $a^{\downarrow}$ is an additive order-ideal.
By Lemma~\ref{lem:ZL}, there is an additive order-ideal $J$ maximal with respect to two properties:
$a^{\downarrow} \subseteq J$ and $J \cap b^{\uparrow} = \varnothing$.
By Lemma~\ref{lem:prime}, we have that $J$ is prime.
By Lemma~\ref{lem:prime-filter}, the complement of $J$ in $S$ is a prime filter 
containing $b$ but omitting $a$.
\end{proof}

\begin{lemma}\label{lem:needed-for-later}
Let $S$ be a Boolean inverse semigroup.
\begin{enumerate}
\item $\mathscr{U}_{a} \subseteq \mathscr{U}_{b}$ if and only if $a \leq b$.
\item $\mathscr{U}_{a} = \mathscr{U}_{b}$ if and only if $a = b$.
\item $\mathscr{U}_{a} \cdot \mathscr{U}_{b} = \mathscr{U}_{ab}$.
\item $\mathscr{U}_{a}$ contains only idempotent ultrafiters if and only if $a$ is an idempotent.
\item $\mathscr{U}_{a}$ is compact.
\item $\mathscr{U}_{a}$ is an idempotent if and only if $a$ is an idempotent.
\item $\mathscr{U}_{a} \sim \mathscr{U}_{b}$ if and only if $a \sim b$.
\end{enumerate}
\end{lemma}
\begin{proof} 
(1) Suppose that  $\mathscr{U}_{a} \subseteq \mathscr{U}_{b}$.
If $a \nleq b$ then by Lemma~\ref{lem:crucial} there exists an ultrafilter that contains $a$ and omits $b$,
which contradicts our assumption.
Thus $a \leq b$ as required.

(2) Immediate by (1) above.

(3) Let $A \in \mathscr{U}_{a}$ and $B \in \mathscr{U}_{b}$ such that $A \cdot B$ is defined.
Then $A \cdot B \in \mathscr{U}_{ab}$. The proof of the reverse inclusion is the same as the proof of part (4) of \cite{Lawson3}.

(4) Only one direction needs proving.
Observe that $\mathscr{U}_{a} \subseteq \mathscr{U}_{a^{2}}$ since
any ultrafilter that contains $a$ is an inverse subsemigroup and so must also contain $a^{2}$.
We therefore have that $a \leq a^{2}$ by (1) above from which it follows that $a = a^{2}$.

(5) By Lemma~\ref{lem:cosets}, there is a bijection between the set $\mathscr{U}_{a}$ and the set $\mathscr{U}_{\mathbf{d}(a)}$
given by $A \mapsto \mathbf{d}(A)$.
A base for the open sets of $\mathscr{U}_{a}$ is the collection $\mathscr{U}_{c}$ where $c \leq b$.
By the above bijection, the set $\mathscr{U}_{c}$ is mapped to the set $\mathscr{U}_{\mathbf{d}(c)}$.
It follows that there is a homeomorphism between $\mathscr{U}_{a}$ and the set $\mathscr{U}_{\mathbf{d}(a)}$.
But by Lemma~\ref{lem:topology-idempotent}, the space $\mathscr{U}_{\mathbf{d}(a)}$ is homeomorphic with the space
$\mathscr{V}_{\mathbf{d}(a)}$ of the prime filters in $\mathsf{E}(S)$ which contain $\mathbf{d}(a)$.
But the sets $\mathscr{V}_{\mathbf{d}(a)}$ are compact by the proof of Lemma~\ref{lem:lc-Boolean-space}.
It follows that $\mathscr{U}_{a}$ is compact.

(6) This follows by part (3) above and part (2).

(7) Only one direction needs proving. Suppose that $\mathscr{U}_{a} \sim \mathscr{U}_{b}$. 
Then by  part (2) of Lemma~\ref{lem:primefilters-form-a-base} and part (3) above,
both $\mathscr{U}_{a^{-1}b}$ and $\mathscr{U}_{ab^{-1}}$ are idempotents.
It follows by part (6) above, that $a \sim b$.
\end{proof}

\subsection{Properties of compact-open local bisections}

Let $G$ be a Boolean groupoid and let $g \in G$.
Define $\mathscr{F}_{g}$ to be the set of compact-open local bisections of $G$ that contain $g$.

\begin{lemma}\label{lem:compact-open-prime} Let $G$ be a Boolean groupoid.
\begin{enumerate}
\item $\mathscr{F}_{g}$ is a prime filter in $\mathsf{KB}(G)$.
\item Every prime filter in $\mathsf{KB}(G)$ is of the form $\mathscr{F}_{g}$ for some $g \in G$.
\item $\mathscr{F}_{g} \cdot \mathscr{F}_{g^{-1}} = \mathscr{F}_{\mathbf{r}(g)}$
and
$\mathscr{F}_{g^{-1}} \cdot \mathscr{F}_{g} = \mathscr{F}_{\mathbf{d}(g)}$.
\item Suppose that $gh$ is defined in $G$. 
Then
$\mathbf{d}(\mathscr{F}_{g}) = \mathbf{r}(\mathscr{F}_{h})$
and 
$\mathscr{F}_{g} \cdot \mathscr{F}_{h} = \mathscr{F}_{gh}$.
\item If $\mathscr{F}_{g} = \mathscr{F}_{h}$. Then $g = h$.
\end{enumerate}
\end{lemma}
\begin{proof}
(1) Let $U, V \in \mathscr{F}_{g}$.
Then $U \cap V$ is an open set containing $g$.
We now use the fact that the compact-open local bisections of $G$ form a base.
There is therefore a compact-open local bisection $W$ containing $g$ such that $W \subseteq U \cap V$.
It follows that $\mathscr{F}_{g}$ is downwardly directed.
It is clearly closed upwards and doesn't contain the empty set.
It is clearly a prime filter.

(2) The proof is the same as the proof of \cite[part (5) of Lemma 2.19]{Lawson3}.

(3) We shall prove 
$\mathscr{F}_{g} \cdot \mathscr{F}_{g^{-1}} = \mathscr{F}_{\mathbf{r}(g)}$
since the proof of the other case is similar.
We use the fact that in a Boolean groupoid the product of compact-open local
bisections is a compact-open local bisection by Proposition~\ref{prop:bis-from-boolean-groupoids}.
Thus 
$\mathscr{F}_{g} \cdot \mathscr{F}_{g^{-1}} \subseteq \mathscr{F}_{\mathbf{r}(g)}$.
But both left-hand side and right-hand side are prime filters
in a Boolean inverse semigroup.
It follows that both are ultrafilters and so must be equal.

(4) The proof is similar to the proof of (3).

(5) Let $g,h \in U \in \mathscr{F}_{g} = \mathscr{F}_{h}$.
We have that 
$\mathbf{d}(\mathscr{F}_{g}) = \mathbf{d}(\mathscr{F}_{h})$.
By definition, $\mathbf{d}(\mathscr{F}_{g}) = \mathscr{F}_{g}^{-1} \cdot \mathscr{F}_{g}$
and this is equal to $\mathscr{F}_{\mathbf{d}(g)}$ by (3) above.
Thus 
$\mathscr{F}_{\mathbf{d}(g)} = \mathscr{F}_{\mathbf{d}(h)}$.
It follows that the compact-open subsets of $G_{o}$ that contain $\mathbf{d}(g)$
are the same as the compact-open subsets of $G_{o}$ that contain $\mathbf{d}(h)$.
But the groupoid $G$ is Boolean and so $G_{o}$ is a locally compact Boolean space.
It follows that $\mathbf{d}(g) = \mathbf{d}(h)$.
But $g$ and $h$ certainly both belong to the same compact-open local bisection
and so $g = h$. 
\end{proof}

\subsection{Proof of the first part of the main theorem}

Given a Boolean inverse semigroup $S$, then by Proposition~\ref{prop:from-bis-to-boolean-groupoid} we have shown how to construct a Boolean groupoid $\mathsf{G}(S)$,
and given a Boolean groupoid $G$, then by Propsition~\ref{prop:bis-from-boolean-groupoids} we have shown how to construct a Boolean inverse semigoup $\mathsf{KB}(G)$.
The following result tells us what happens when we iterate these two constructions.

\begin{proposition}\label{prop:isomorphism-of-structures} \mbox{}
\begin{enumerate}
\item Let $S$ be a Boolean inverse semigroup. Define a function $\alpha \colon S \rightarrow \mathsf{KB} (\mathsf{G}(S))$ by $\alpha (a) = \mathscr{U}_{a}$.
Then this is an isomorphism of semigroups.
\item Let $G$ be a Boolean groupoid. Define a function $\beta \colon G \rightarrow \mathsf{G} (\mathsf{KB}(G))$  by $\beta (g) = \mathscr{F}_{g}$.
Then $\beta$ is an isomorphism of groupioids and a homeomorphism.
\end{enumerate}
\end{proposition}
\begin{proof} (1) By Lemma~\ref{lem:local-bisection}, the set $\mathscr{U}_{a}$ is a local bisection.
It is open by definition of the topology.
It is compact by part (5) of Lemma~\ref{lem:needed-for-later}.
Thus $\mathscr{U}_{a}$ is a compact-open local bisection.
It follows that the function is well-defined.
It is a semigroup homomorphism by part (3) of Lemma~\ref{lem:needed-for-later}, since $\mathscr{U}_{a}\mathscr{U}_{b} = \mathscr{U}_{ab}$.
It is injective by part (2) of Lemma~\ref{lem:needed-for-later}.
It remains to show that it is surjective.
Let $U$ be any compact-open local bisection of $\mathsf{G}(S)$.
Since it is open it is a union of sets of the form $\mathscr{U}_{a}$ and since it is compact
it is a union of a finite numbers of sets of this form.
It follows that $U = \bigcup_{i=1}^{n} \mathscr{U}_{a_{i}}$.
But $\mathscr{U}_{a_{i}} \subseteq U$ which is the natural partial order in the inverse semigroup  $\mathsf{KB} (\mathsf{G}(S))$.
It follows that the set of elements of the form $\mathscr{U}_{a_{i}}$ is compatible.
By part (7) of Lemma~\ref{lem:needed-for-later},
it follows that set $\{a_{1}, \ldots, a_{m}\}$ is compatible.
Put $a = \bigvee_{i=1}^{n}a_{i}$.
Then $\alpha (a) = U$.
We have therefore proved that $\alpha$ is an isomorphism of semigroups.
  
(2) By Lemma~\ref{lem:compact-open-prime}, 
$\beta$ is a bijective functor.
It remains to show that it is a homeomorphism.
Since $G$ is a Boolean groupoid, a base for the topology on $G$ is provided by the compact-open local bisections.
Let $U$ be a comapct-open local bisection of $G$.
Then $U \in \mathsf{KB}(G)$.
We may therefore form the set $\mathscr{U}_{U}$ which is a typical element of the base for the topology 
on $\mathsf{G}(\mathsf{KB}(G))$.
It is now easy to check (or see the proof of \cite[Part (1) of Proposition 2.23]{Lawson3}),
that the bijection $\beta$ restricts to a bijection between $U$ and $\mathscr{U}_{U}$. 
\end{proof}

\subsection{Proof of the main theorem}

Our goal now is to take account of appropriate morphisms in our constructions.
We refer the reader to \cite{KL} for information about more general kinds of morphisms.

The statement and proof of \cite[part (3), Lemma 3.11]{LL} is incorrect.
We now give the correct statement and proof.

\begin{lemma}\label{lem:weakly-meet} Let $\theta \colon S \rightarrow T$ be a morphism of distributive inverse semigroups.
Then for each prime filter $P$ we have that $\theta^{-1}(P)$ is non-empty if and only if each $t \in T$
can be written $t = \bigvee_{i=1}^{n} t_{i}$ where each $t_{i} \leq \theta (s_{i})$ for some $s_{i} \in S$.
\end{lemma}
\begin{proof} We prove the easy direction first.
Suppose that for each $t \in T$ we can write $t = \bigvee_{i=1}^{n} t_{i}$ where each $t_{i} \leq \theta (s_{i})$ for some $s_{i} \in S$.
Let $P$ be any prime filter.
By assumption it is non-empty.
Let $t \in P$.
By assumption, we can write $t = \bigvee_{i=1}^{n} t_{i}$ where each $t_{i} \leq \theta (s_{i})$ for some $s_{i} \in S$.
But $P$ is a prime filter.
Thus $t_{i} \in P$ for some $i$.
It follows that $\theta (s_{i}) \in P$ for some $s_{i} \in S$.
It follows that  $\theta^{-1}(P)$ is non-empty. 
We now prove the converse.
Suppose that $t \in S$ which cannot be written in the stated form.
Then $t \notin \left( \mbox{im}(\theta)^{\downarrow} \right)^{\vee}$.
Put $I = \left( \mbox{im}(\theta)^{\downarrow} \right)^{\vee}$.
Then $t^{\uparrow} \cap I = \emptyset$.
We now use Section~7.1 to deduce that there is a prime filter $P$ that contains $t$ 
and is disjoint from $J$.
But this implies that $\theta^{-1}(P)$ is empty which is a contradiction.
It follows that no such element $t$ exists.
\end{proof}

A morphism $\theta \colon S \rightarrow T$ of Boolean inverse semigroups is said to be 
{\em weakly-meet-preserving} if $t \leq \theta (a), \theta (b)$
there exists $c \leq a,b$ such that $t \leq \theta (c)$.
The following is \cite[Proposition 3-4.6]{W}.

\begin{lemma}\label{lem:noise} Let $S$ be a Boolean inverse semigroup.
Let $I$ be an additive ideal of $S$.
\begin{enumerate}
\item Define $(a,b) \in \varepsilon_{I}$ if and only if there exists $c \leq a,b$ such that $a \setminus c, b \setminus c \in I$.
Then $\varepsilon_{I}$ is an additive congruence with kernel $I$.
\item If $\sigma$ is any additive congruence with kernel $I$ then $ \varepsilon_{I} \subseteq \sigma$.
\end{enumerate}
\end{lemma}

An additive congruence is {\em ideal-induced} if it equals $\varepsilon_{I}$ for some additive ideal $I$.
The following result is due to Ganna Kudryavtseva (private communication).

\begin{proposition}\label{prop:anja} A morphism of Boolean inverse semigroups is weakly-meet-pre\-serving if
and only if its associated congruence is ideal-induced.
\end{proposition}
\begin{proof} Let $I$ be an additive ideal of $S$ and let $\varepsilon_{I}$ be its associated additive congruence on $S$.
Denote by $\nu \colon S \rightarrow S/\varepsilon_{I}$ its associated natural morphism.
We prove that $\nu$ is weakly-meet-preserving.
Denote the $\varepsilon_{I}$-class containing $s$ by $[s]$.
Let $[t] \leq [a], [b]$.
Then $[t] = [at^{-1}t]$ and $[t] = [bt^{-1}t]$.
By definition there exist $u,v \in S$ such that
$u \leq t,at^{-1}t$ and $v \leq t,bt^{-1}t$ such that
$t \setminus u, at^{-1}t \setminus u, t \setminus v, bt^{-1}t \setminus v \in I$.
Now $[t] = [u] = [at^{-1}t]$ and $[t] = [v] = [bt^{-1}t]$.
Since $u,v \leq t$ it follows that $u \sim v$ and so $u \wedge v$ exists.
Clearly, $u \wedge v \leq a,b$.
In addition $[t] = [u \wedge v]$.
We have proved that $\nu$ is weakly-meet-preserving.

Conversely, let $\theta \colon S \rightarrow T$ be weakly-meet-preserving.
We prove that it is determined by its kernel $I$.
By part (2) of Lemma~\ref{lem:noise}, it is enough to prove that if $\theta (a) = \theta (b)$ then
we can find $c \leq a,b$ such that $a \setminus c, b \setminus c \in I$.
Put $t = \theta (a) = \theta (b)$.
Then there exists $c \leq a,b$ such that $t \leq \theta (c)$.
It is easy to check that $\theta (a \setminus c) = 0 = \theta (b \setminus c)$.
We have therefore proved that $a \setminus c, b \setminus c \in I$ and so $(a,b) \in \varepsilon_{I}$.
\end{proof}

We now combine the above two properties.
A morphism $\theta \colon S \rightarrow T$ of Boolean inverse semigroups is said to be {\em callitic}\footnote{I made this word up. It comes from the Greek word `kallos' meaning beauty.
I simply wanted to indicate that these maps were sufficiently `nice'.} 
if it satisfies two conditions:
\begin{enumerate}
\item We require that $\theta$ be {\em proper}. This means that for each $t \in T$ we can write $t = \bigvee_{i=1}^{n} t_{i}$
where each $t_{i} \leq \theta (s_{i})$ for some $s_{i} \in S$.
\item We require that $\theta$ be {\em weakily-meet-preserving}.
\end{enumerate}
A continuous function between topological spaces is said to be {\em coherent} if the inverse images
of compact-open subsets are compact-open.
You can easily check that the collection of Boolean inverse semigroups and callitic morphisms forms a category, 
as does the collection of Boolean groupoids and coherent, continuous covering functors.
We can now state and prove the main theorem of this paper.

\begin{theorem}[Non-commutative Stone duality]\label{them:non-com-stone} The category of Boolean inverse semigroups and callitic morphisms
is dually equivalent to the category of Boolean groupoids and coherent, continuous covering functors.
\end{theorem}
\begin{proof} Let $\theta \colon S \rightarrow T$ be a callitic morphism between Boolean inverse semigroups.
Let $B$ be a prime filter in $T$.
We prove that $\theta^{-1}(B)$ is a prime filter in $S$.
By Lemma~\ref{lem:weakly-meet} this set is non-empty.
Let $x,y \in \theta^{-1}(B)$.
Then $\theta (x), \theta (y) \in B$.
But $B$ is a filter and so there is an element $b \in B$ such that
$b \leq \theta (x), \theta (y)$.
Since $\theta$ is weakly-meet-perserving, 
there is $s \in S$ such that $s \leq x,y$ and $b \leq \theta (s)$.
But $b \in B$ and so $\theta (s) \in B$ and so $s \in \theta^{-1}(B)$.
Let $x \in \theta^{-1}(B)$ and $x \leq y$.
Then $\theta (x) \leq \theta (y)$ and $\theta (x) \in B$.
It follows that $\theta (y) \in B$ and so $y \in \theta^{-1}(B)$.
Let $x \vee y \in \theta^{-1}(B)$.
Then $\theta (x \vee y) \in B$.
Now we use the fact that $\theta$ is also a morphism to get that $\theta (x) \vee \theta (y) \in B$.
But $B$ is a prime filter.
Without loss of generality, suppose that $\theta (x) \in B$ and so $x \in \theta^{-1}(B)$.
Put $\theta^{\star} = \theta^{-1}$.
We have therefore defined a function
$\theta^{\star} \colon \mathsf{G}(T) \rightarrow \mathsf{G}(S)$.
It remains to show that $\theta^{\star}$ is a coherent, continuous covering functor.

The bulk of the proof is taken up with showing that $\theta^{-1}$ is a functor.
Let $F$ be an idempotent prime filter in $T$.
Thus by Lemma~\ref{lem:idempotent-filter-is},
this is an inverse subsemigroup of $T$.
Then $\theta^{-1}(F)$ is a prime filter in $S$
and the inverse image of an inverse subsemigroup is an inverse subsemigroup.
It follows that $\theta^{-1}(F)$ is an idempotent prime filter.
We have therefore shown that $\theta^{-1}$ maps identities to identities.

We next prove that if $F$ and $G$ are prime filters such that $F^{-1} \cdot F = G \cdot G^{-1}$ then
$$( \theta^{-1} (F) \theta^{-1} (G) )^{\uparrow} = \theta^{-1} ( (FG)^{\uparrow}  ).$$
We prove first that
$$\theta^{-1} (F) \theta^{-1} (G) \subseteq \theta^{-1} (FG).$$
Let $s \in \theta^{-1} (F) \theta^{-1} (G)$.
Then $s = ab$ where $a \in \theta^{-1} (F)$ and $b \in \theta^{-1} (G)$.
Thus $\theta (s) = \theta (a) \theta (b) \in FG$.
It follows that $s \in \theta^{-1} (FG)$.
Observe that $\theta^{-1} (X)^{\uparrow} \subseteq \theta^{-1} (X^{\uparrow})$.
It follows that
$$( \theta^{-1} (F) \theta^{-1} (G) )^{\uparrow} \subseteq \theta^{-1} ( (FG)^{\uparrow}  ).$$
We now prove the reverse inclusion.
Let $s \in \theta^{-1} ( (FG)^{\uparrow} ).$
Then $\theta (s) \in F \cdot G$ and so
$fg \leq \theta(s)$ for some $f \in F$ and $g \in G$.
The map $\theta$ is assumed proper and so we may quickly deduce that there exists $v \in S$ such that $\theta (v) \in G$.
Consider the product $\theta (s) \theta (v)^{-1}$.
Since $\theta (s) \in F \cdot G$ and $\theta (v)^{-1} \in G^{-1}$ we have that
$\theta (s)\theta (v)^{-1} \in F \cdot G \cdot G^{-1} = F \cdot F^{-1} \cdot F = F$.
Thus $\theta (sv^{-1}) \in F$,  and we were given $\theta (v) \in G$, and clearly $(sv^{-1})v \leq s$.
Put $a = sv^{-1}$ and $b = v$.
Then $ab \leq s$ where $\theta (a) \in F$ and $\theta (b) \in G$.
It follows that $s \in (\theta^{-1} (F) \theta^{-1} (G) )^{\uparrow}$.

We may now show that $\theta^{-1}$ is a functor.
Let $F$ be a prime filter.
Observe that $\theta^{-1}(F)^{-1} = \theta^{-1} (F^{-1})$.
We have that
$$(\theta^{-1} (F^{-1}) \theta^{-1} (F))^{\uparrow}
=
(\theta^{-1} (F)^{-1} \theta^{-1} (F))^{\uparrow}
=
\mathbf{d} ( \theta^{-1} (F))$$
and
$$\theta^{-1} ( (F^{-1}F)^{\uparrow} ) = \theta^{-1} (\mathbf{d} (F)).$$
Hence by our result above
$$\theta^{-1} (\mathbf{d} (F))
=
\mathbf{d} ( \theta^{-1} (F)).$$
A dual result also holds and so $\theta^{-1}$ preserves the domain and codomain operations.
Suppose that $\mathbf{d}(F) = \mathbf{r} (G)$ so that $F \cdot G$ is defined.
By our calculation above $\mathbf{d} (\theta^{-1}(F)) = \mathbf{r} (\theta^{-1} (G))$
and so the product $\theta^{-1} (F) \cdot \theta^{-1} (G)$ is defined.
By our main result above we have that
$$\theta^{-1} (F \cdot G) = \theta^{-1} (F) \cdot \theta^{-1} (G),$$
as required.
We have therefore shown that $\theta^{-1}$ is a functor.
The proof that $\theta^{-1}$ is a covering functor follows the same lines as the proof of \cite[Proposition~2.15]{Lawson3}:
the proof of star injectivity uses Lemma~\ref{lem:local-bisection},
and the proof of star surjectivity uses this same lemma and Lemma~\ref{lem:d-is-extremal}.
To show that $\theta^{-1}$ is continuous, observe that 
a basic open set of $\mathsf{G}(S)$ has the form $\mathscr{U}_{s}$ for some $s \in S$.
It is simple to check that this is pulled back under the map $(\phi^{-1})^{-1}$ to the set $\mathscr{U}_{\theta (s)}$.
We now prove coherence.
Let $X$ be a compact-open subset of $\mathsf{G}(S)$.
Then $X$ may be written as a union of compact-open local bisections.
Thus, by compactness, we can write $X$ as a union of a finite number of compact-open local bisections.
We now use the previous result to deduce that the inverse image of $X$ is a finite union of compact-open local bisections
and so is itself a compact-open set.

Let $\phi \colon G \rightarrow H$ be a  coherent, continuous covering functor.
Let $U$ be a compact-open local bisection of $H$.
Then $\phi^{-1}(U)$ is compact-open since $\phi$ is coherent.
Without loss of generality, we can assume that it is non-empty.
We prove that it is a local bisection.
Let $g,h \in \phi^{-1}(U)$ be such that $\mathbf{d}(g) = \mathbf{d}(h)$.
Then $\theta (g), \theta (h) \in U$ and $\mathbf{d}(\theta (g)) = \mathbf{d}(\theta (h))$.
By assumption, $U$ is a local bisection and so $\theta (g) = \theta (h)$.
We now use the fact that $\phi$ is star-injective to deduce that $g = h$.
A dial result proves that $U$ is a local bisection.
Put $\phi^{-1} = \phi_{\star}$.
We have therefore defined a function
$\phi_{\star} \colon \mathsf{KB}(H) \rightarrow \mathsf{KB}(G)$.
It remains to show that $\phi_{\star}$ is a callitic morphism.

The proof that it is a semigroup homomorphism follows the same lines as the proof
in \cite[Proposition 2.17]{Lawson3} where we use Lemma~\ref{lem:former}.
We show that this map is proper and weakly-meet-preserving.
We prove first that $\phi^{-1}$ is proper. 
Let $B \in \mathsf{KB}(G)$.
Then $B$ is a non-empty compact-open local bisection in $G$.
Let $g \in B$.
Then $\phi (g) \in H$.
Clearly, $H$ is an open set containing $\phi (g)$.
Since $H$ is \'etale, it follows that $H$ is a union of compact-open local bisections 
and so $\phi (g) \in C_{g}$ for some an compact-open local bisection $C_{g}$ in $H$.
Since $\phi$ is continuous and coherent $g \in \phi^{-1}(C_{g})$ is compact-open 
and because $\phi$ is a covering functor $\phi^{-1} (C_{g})$ is a local bisection.
It follows that $B \subseteq \bigcup_{g \in B} \phi^{-1}(C_{g})$.
Since $B$ is compact, we may in fact write
$B \subseteq \bigcup_{i=1}^{m} \phi^{-1}(C_{g_{i}})$
for some finite set of elements $g_{1}, \ldots, g_{m} \in B$.
Put $B_{i} = B \cap \phi^{-1}(C_{g_{i}})$.
This is clearly an open local bisection and $B = \bigcup_{i=1}^{m} B_{i}$ 
and each $B_{i} \subseteq \phi^{-1}(C_{i})$ where $C_{i} = C_{g_{i}}$.
We prove that we may find compact-open local bisections $D_{i}$ such that
$B$ is the union of the $D_{i}$ and $D_{i} \subseteq \phi^{-1}(C_{i})$.
Since $B_{i}$ is an open local bisection it is a union of compact-open local bisections.
Amalgamating these unions we have that $B$ is a union of compact-open local bisections
each of which is a subset of one of the $\phi^{-1}(C_{i})$.
It follows that $B$ is a union of a finite number of such compact-open local bisections.
Define $D_{i}$ to be the union of those which are contained in $\phi^{-1}(C_{i})$ and the result follows.
We now prove that $\phi^{-1}$ is weakly-meet-preserving.
Let $A,B \in \mathsf{KB}(H)$.
Then $A$ and $B$ are compact-open local bisections of $H$.
Let $Y$ be any compact-open local bisection of $G$ such that $Y \subseteq \phi^{-1}(A), \phi^{-1}(B)$.
Clearly, $Y \subseteq \phi^{-1}(A \cap B)$.
We can at least say that $A \cap B$ is an open local bisection
and 
$\theta (Y) \subseteq A \cap B$.
Since $\phi$ is continuous, we know that $\phi (Y)$ is compact.
Now $H$ has a base of compact-open local bisections.
It follows that $A \cap B$ is a union of compact-open local bisections.
But $\theta (Y)$ is compact and so $\theta (Y)$ is contained in a finite union of compact-open local
bisections that is also contained in $A \cap B$.
Thus $\theta (Y) \subseteq V = \bigcup_{i=1}^{m} V_{i} \subseteq A \cap B$.
Now $A \cap B$ a local bisection implies that $V$ is a local bisection.
It is evident that $V$ is a compact-open local bisection itself.
We therefore have $\theta (Y) \subseteq V \subseteq A \cap B$.
Hence $Y \subseteq \theta^{-1}(\theta (Y)) \subseteq \theta^{-1}(V) \subseteq \theta^{-1}(A \cap B)$. 

Let $\theta \colon S \rightarrow T$ be a callitic morphism between two Boolean inverse semigroups.
Define $\mathsf{G}(\theta) = \theta^{\star}$.
Let $\phi \colon G \rightarrow H$ be a coherent continuous covering functor between Boolean groupoids.
Define $\mathsf{KB}(\mathsf{\phi}) = \phi_{\star}$.
It is now routine to check that we have defined functors:
where $\mathsf{G}$ takes us from Boolean inverse semigroups and callitic morphisms
to the dual of the category of Boolean groupoids and coherent continuous covering functors;
and where $\mathsf{KB}$ takes us from the category of Boolean groupoids and coherent continuous covering functors
to the dual of the category of Boolean inverse semigroups and callictic morphisms.

Let $\theta \colon S \rightarrow T$ be a callitic morphism between two Boolean inverse semigroups.
We shall compare this to the callitic morphism 
$(\theta^{\star})_{\star} \colon \mathsf{KB}(\mathsf{G}(S)) \rightarrow \mathsf{KB}(\mathsf{G}(T))$
using Proposition~\ref{prop:isomorphism-of-structures}.
We have to compute $(\theta^{\star})_{\star}(\mathscr{U}_{s})$.
It is routine to check that this is $\mathscr{U}_{\theta (s)}$. 
Let $\phi \colon G \rightarrow H$ be a coherent continuous covering functor between Boolean groupoids.
We shall compare this to the coherent continuous covering functor
$(\phi_{\star})^{\star} \colon \mathsf{G}(\mathsf{KB}(G)) \rightarrow \mathsf{G}(\mathsf{KB}(H))$
using Proposition~\ref{prop:isomorphism-of-structures}.
We have to compute $(\phi_{\star})^{\star}(\mathscr{F}_{g})$.
It is routine to check that this is $\mathscr{F}_{\phi (g)}$.

The functor $\mathsf{KB} \circ \mathsf{G}$ is now clearly naturally isomorphic with the identity functor on 
the category of Boolean inverse semigroups and callitic morphisms,
whereas
the functor $\mathsf{G} \circ \mathsf{KB}$ is now clearly naturally isomorphic with the identity functor on
the category of Boolean groupoids and coherent continuous covering functors.
\end{proof}

\section{Special cases}

In this section, we shall describe some special cases of Theorem~\ref{them:non-com-stone}. 
None of these is new, but they have not appeared altogether in this way before.

By the results of Section~2.3 and 2.4,
a generalized Boolean algebra with an identity is nothing other than a Boolean algebra.
It follows that the Boolean inverse semigroups with the property 
that the space of identities of their Stone groupoids are compact are precisely the Boolean inverse monoids.

The group of units of a Boolean inverse monoid is just the set of all compact-open bisections of the associated Stone groupoid;
the compact-open bisections form what is known as the {\em topological full group} of the Boolean groupoid.

Suppose that $S$ is a countable Boolean inverse semigroup.
Then the Stone groupoid of $S$ must have a countable base of compact-open local bisections.
It follows that its Stone groupoid is second-countable.

The Stone space of the Tarski algebra is the Cantor space by Example (2) of Example~\ref{ex:some-stone-spaces}.
Define a {\em Tarski monoid} to be a countably infinite Boolean inverse monoid
whose semilattice of idempotents forms a Tarski algebra.
Define a {\em Tarski groupoid} to be a second-countable Boolean groupoid whose space of identities is the Cantor space.

Following Wehrung \cite{W}, 
define a {\em semisimple} Boolean inverse semigroup to be one 
in which for each element $a$ the principal order ideal $a^{\downarrow}$
is finite. 

\begin{proposition}\label{prop:semisimple} Let $S$ be a Boolean inverse semigroup.
Then $S$ is semisimple if and only if $\mathsf{G}(S)$ carries the discrete topology.
\end{proposition}
\begin{proof} Let $a$ be any element of such a semigroup.
Then $\mathscr{U}_{a} = \bigcup_{i=1}^{n} \mathscr{U}_{a_{i}}$, where $a_{1}, \ldots, a_{n}$ are all the atoms below $a$;
this is proved by observinng that in a semisimple Boolean inverse semigroup,
every non-zero element lies above an atom and so every non-zero element is a join of atoms.
If $a$ is an atom then $\mathscr{U}_{a}$ contains exactly one element.
It follows that the sets $\mathscr{U}_{a}$, where $a$ is an atom, form a base for the topology on $\mathsf{G}(S)$,
and so this set is equipped with the discrete topology.
Conversely, suppose that $\mathsf{G}(S)$ is equipped with the discrete topology.
Then, for each $a \in S$, we have that $\mathscr{U}_{a}$ is finite since it is compact-open.
But $b \leq a$ if and only if $\mathscr{U}_{b} \subseteq \mathscr{U}_{a}$.
This proves that the Boolean inverse semigroup $S$ is semisimple.
\end{proof}

If $S$ is a semisimple Boolean inverse semigroup,
then the Stone groupoid $\mathsf{G}(S)$ is isomorphic to the set of atoms of $S$ equipped with the restricted product;
see \cite{Lawson2021} for the structure of semisimple Boolean inverse semigroups.
The significance of semisimple Boolean inverse semigroups is explained by the following result.
We say that a Boolean inverse semigroup is {\em atomless} if it has no atoms.

\begin{proposition}[Dichotomy theorem] Let $S$ be a $0$-simplifying Boolean inverse semigroup.
Then either $S$ is semisimple or $S$ is atomless.
\end{proposition}
\begin{proof} Suppose that $a$ is an atom.
Then $\mathbf{d}(a)$ is an atom.
We can therefore assume, without loss of generality, that there is at least one idempotent atom $e$.
Let $x$ be any non-zero element of $S$.
We prove that $x^{\downarrow}$ is finite.
It is enough to prove that $\mathbf{d}(x)^{\downarrow}$ is finite.
Let $f \leq \mathbf{d}(x)$.
Since the semigroup is assumed to be $0$-simplifying, there is a pencil from $f$ to $e$.
There is therefore a finite set of elements $\{x_{1},\ldots, x_{n} \}$ such that $f = \bigvee_{i=1}^{n} \mathbf{d}(x_{i})$
and $\mathbf{r}(x_{i}) \leq e$.
But $e$ is an atom.
Without loss of generality, we assume that all the $x_{i}$ are non-zero.
Thus $\mathbf{r}(x_{i}) = e$.
It follows that each $\mathbf{d}(x_{i})$ is an atom
and we have proved that $f$ is a join of a finite number of atoms.
It follows that $\mathbf{d}(x)^{\downarrow}$ is a finite Boolean algebra.
\end{proof}

Our next result tells us when the Stone groupoid is Hausdorff.

\begin{proposition}\label{prop:ttwo} Let $S$ be a Boolean inverse semigroup.
Then $S$ is a meet-semigroup if and only if $\mathsf{G}(S)$ is Hausdorff.
\end{proposition}
\begin{proof} Let $S$ be a Boolean inverse semigroup and suppose that it is a meet-semigroup.
Let $F$ and $G$ be distinct ultrafilters in $\mathsf{G}(S)$,
since neither ultrafilter can be a subset of the other,
we can find elements $a \in F \setminus G$ and $b \in G \setminus F$.
The element $a \wedge b$ exists by assumption and also by assumption $a \neq a \wedge b$ and $b \neq a \wedge b$.
Observe that $F \in \mathscr{U}_{a \setminus (a \wedge b)}$ 
and
$G \in \mathscr{U}_{b \setminus (a \wedge b)}$.
Let $x \leq a \setminus (a \wedge b), b \setminus (a \wedge b)$.
Then $x \leq a \wedge b$.
It follows that $x = 0$.
We have proved that $\mathscr{U}_{a \setminus (a \wedge b)} \cap \mathscr{U}_{b \setminus (a \wedge b)} = \varnothing$
and so we have proved that $\mathscr{G}(S)$ is Hausdorff.
Conversely, suppose that $\mathscr{G}(S)$ is Hausdorff.
Let $A$ and $B$ be two compact-open local bisections.
By part (2) of Lemma~\ref{lem:topology-needed}, both $A$ and $B$ are clopen.
Thus $A \cap B$ is clopen.
But by part (3) of Lemma~\ref{lem:topology-needed},
$A \cap B$ is also compact, and so it is a compact-open local bisection.
The result, that $S$ is a meet-semigroup, follows by Proposition~\ref{prop:isomorphism-of-structures}.
\end{proof}

\begin{remark}{\em The above result was essentially the basis of my paper \cite{Lawson3}. 
This paper won the Mahony-Neumann-Room Prize of the Australian Mathematical Society in 2017.}
\end{remark}

Recall that an inverse semigroup is {\em fundamental} if the only elements
that commute with all idempotents are themselves idempotents.
If we denote the {\em centralizer of the idempotents} by $\mathsf{Z}(\mathsf{E}(S))$
then an inverse semigroup is fundmanetal when  $\mathsf{E}(S) = \mathsf{Z}(\mathsf{E}(S))$
An \'etale topological groupoid $G$ is said to be {\em effective} if the 
interior of the isometry groupoid consists simple of the space of identities.

\begin{proposition}\label{prop:effective} Let $S$ be a Boolean inverse semigroup.
Then $S$ is fundamental if and only if $\mathsf{G}(S)$ is effective.
\end{proposition}
\begin{proof} We prove first that $a \in \mathsf{Z}(\mathsf{E}(S))$ if and only if 
$\mathscr{U}_{a} \subseteq \mbox{\rm Iso}(\mathsf{G}(S))$.
Suppose that $a$ centralizes all the idempotents of $S$.
Then, in particular, $a$ entralizes $aa^{-1}$.
Thus $aaa^{-1} = aa^{-1}a = a$.
It follows that $a^{-1}a aa^{-1} = aa^{-1}$.
This proves that $aa^{-1} \leq a^{-1}a$.
By symmetry, $a^{-1}a \leq aa^{-1}$ and so $a^{-1}a = aa^{-1}$.
Let $A \in \mathscr{U}_{a}$.
We need to prove that $\mathbf{d}(A) = \mathbf{r}(A)$.
Let $x \in \mathbf{d}(A)$.
Then there is an idempotent $e \in \mathbf{d}(A)$ such that $e \leq x$.
By assumption, $a \in A$ and so $a^{-1}a \in \mathbf{d}(A)$.
It follows that $ea^{-1}a \in A$ and clearly $ea^{-1}a \leq x$.
But $ea^{-1}a = eaa^{-1} = aea^{-1}$, where we have again used the fact that $a$ commutes with all idempotents.
It follows that $aea^{-1} \leq x$.
Now $ae \in A \cdot \mathbf{d}(A) = A$.
It follows that $x \in \mathsf{r}(A)$.
We have therefore prove that $\mathsf{d}(A) \subseteq \mathbf{r}(A)$.
The proof of the reverse inclusion follows by symmetry.
We have therefore proved that $\mathbf{d}(A) = \mathbf{r}(A)$.
It follows that 
$a \in \mathsf{Z}(\mathsf{E}(S))$ 
implies that $\mathscr{U}_{a} \subseteq \mbox{\rm Iso}(\mathsf{G}(S))$.
Now, suppose that $\mathscr{U}_{a} \subseteq \mbox{\rm Iso}(\mathsf{G}(S))$.
Let $e$ be any idempotent.
We shall prove that $\mathscr{U}_{ae} = \mathscr{U}_{ea}$ and the result will follow by part (2) of
Lemma~\ref{lem:needed-for-later}.
Suppose that $A \in \mathscr{U}_{ae}$.
Then $ae \in A$ and so $a \in A$.
By assumption, $\mathbf{d}(A) = \mathbf{r}(A)$.
Now, $ae \in A$ and so $(ae)^{-1}ae =  ea^{-1}a$.
It follows that $ea^{-1}a \in \mathbf{r}(A)$.
Hence $a^{-1}aea \in A$.
It follows that $ea \in A$.
We have therefore proved that $A \in \mathscr{U}_{ea}$.
By symmetry, $\mathscr{U}_{ae} = \mathscr{U}_{ea}$.

We now prove the claim.
Suppose that $S$ is fundamental.
Let $\mathscr{U}_{a} \subseteq  \mbox{\rm Iso}(\mathsf{G}(S))$.
Then $a$ must commute with all idempotents and so is itself an idempotent.
It follows that the only open sets in $\mbox{\rm Iso}(\mathsf{G}(S))$
are open sets of identities.
It follows that $\mathsf{G}(S)$ is effective.
Conversely, suppose that $\mathsf{G}(S)$ is effective.
Let $a$ commute with all idempotents.
Then $\mathscr{U}_{a} \subseteq \mbox{\rm Iso}(\mathsf{G}(S))$.
It follows that every element of $\mathscr{U}_{a}$ is an idempotent ultrafilter.
Thus by part (4) of Lemma~\ref{lem:needed-for-later}, we deduce that $a$ is an idempotent
and so $S$ is fundamental.
\end{proof}

Fundamental inverse semigroups have the important property that
there are no non-trivial {\em idempotent-separating} homomophsims.

Recall that an {\em infinitesimal} in an inverse semigroup $S$ with zero
is a non-zero element $a$ such that $a^{2} = 0$.
We say that a Boolean inverse semigroup is {\em basic} if every element is 
a finite join of infinitesimals and an idempotent.
You can check that basic inverse semigroups always have meets; see \cite[Lemma 4.30]{Lawson3}.

\begin{lemma}\label{lem:infinitesimals-in-ultrafilters} Let $S$ be a Boolean inverse semigroup.
Then every ultrafilter $A$ such that $\mathbf{d}(A) \neq \mathbf{r}(A)$
contains an infinitesimal. 
\end{lemma}
\begin{proof} Let $A$ be an ultrafilter such that $\mathbf{d}(A) \neq \mathbf{r}(A)$.
Then $\mathbf{d}(A) = E^{\uparrow}$ and $\mathbf{r}(A) = F^{\uparrow}$ where $E$ and $F$ are ultrafilters in
the generalized Boolean algebra $\mathsf{E}(S)$.
By classical Stone duality extended to generalized Boolean algebras, 
we know that the structure space of $\mathsf{E}(S)$ is Hausdorff.
Let $e,f \in \mathsf{E}(S)$ be such that $E \in \widetilde{U}_{e}$, $F \in \widetilde{U}_{f}$ and $\widetilde{U}_{e} \cap \widetilde{U}_{f} = \varnothing$.
In particular, $ef = 0$.
Let $a \in A$.
Then $eaf \in A$.
But $eaf$ is an infinitesimal. 
\end{proof}

The following was stated in \cite[Proposition 4.31]{Lawson2017}
but the condition that the Boolean inverse semigroup be a meet-monoid was omitted.

\begin{proposition}\label{prop:basic-principal} Let $S$ be a Boolean inverse meet-semigroup.
Then $S$ is basic if and only if $\mathsf{G}(S)$ is a principal groupoid.
\end{proposition}
\begin{proof} The proof that basic implies principal does not require the assumption that
the Boolean inverse semigroup be a meet-monoid.
Since ultrafilters are prime in Boolean inverse semigroups,
it follows that if $A$ is an element of a local group of the groupoid
then it cannot contain infinitesimals and so must contain an idempotent from which it follows that $A$ is an identity
ultrafilter.
We now prove the converse and use a different approach from the one adopted in \cite{Lawson2016}.
We are given that $\mathsf{G}(S)$ is principal and Hausdorff and we shall prove that $S$ is basic.
Let $a$ be any element of $S$.
Then, by assumption, $\phi (a) \leq a$ is the largest idempotent less than or equal to $a$.
Observe that $a = \phi (a) \vee (a \setminus \phi (a))$ is an orthogonal join
and that $\phi (a \setminus \phi (a)) = 0$.
With this in mind,
let $b \in S$ be an element such that $\phi (b) = 0$.
We shall be done if we prove that $b$ is a finite join of infinitesimals.
Consider the set $\mathscr{U}_{b}$ of all prime filters that contain $b$.
Let $A \in \mathscr{U}_{b}$.
The groupoid $\mathsf{G}(S)$ is principal.
If $\mathbf{d}(A) = \mathbf{r}(A)$ then $A$ must be an identity prime filter and so
contains non-zero idempotents.
It follows that $b$ must be above a non-zero idempotent, which is a contradiction.
Thus $\mathbf{d}(A) \neq \mathbf{r}(A)$.
We now use Lemma~\ref{lem:infinitesimals-in-ultrafilters} to deduce that
when the groupoid $\mathsf{G}(S)$ is principal each non-identity prime filter $A$ must contain infinitesimals.
But any element below an infinitesimal is either zero or an infinitesimal.
Thus $b$ must lie above an infinitesimal.
It follows that $\mathscr{U}_{b} = \bigcup_{x \leq b, x^{2} = 0, x \neq 0} \mathscr{U}_{x}$.
We now use compactness of $U_{b}$ to deduce that 
$\mathscr{U}_{b} = \bigcup_{i=1}^{m} \mathscr{U}_{x_{i}}$ 
where each $x_{i}$ is an infinitesimal $x_{i} \leq b$.
It follows by part (2) of Lemma~\ref{lem:needed-for-later} 
that $b$ is a join of a finite number of infinitesimals.
We have therefore proved that $S$ is basic.
\end{proof}

A subset of a groupoid is said to be {\em invariant} if it is a union of connected components.
The following combines results to be found in \cite{Lenz} and \cite{Lawson2016}. 

\begin{lemma} Let $S$ be a Boolean inverse semigroup.
Then there is an order-isomorphism between the set of additive ideals of $S$
and the set of open invariant subsets of $\mathsf{G}(S)$.
\end{lemma}

An \'etale groupoid is said to be {\em minimal} if it contains exactly two open invariant subsets.
The following was essentially proved as \cite[Corollary 4.8]{Lawson3},
but we do not need the assumption that the semigroup is a meet-semigroup.

\begin{proposition}\label{prop:zero-simplifying} A Boolean inverse semigroup $S$ is
$0$-simplifying if and only if $\mathsf{G}(S)$ is minimal.
\end{proposition}

A Boolean inverse semigroup which is fundamental and $0$-simplifying is said to be {\em simple}.
The terminology is explained by the following result.

\begin{lemma}Let $\theta\colon S \rightarrow T$ be a morphism of Boolean inverse semigroups
where $S$ is simple. Then $\theta$ is injective. 
\end{lemma}
\begin{proof} We prove first that $\theta$ is injective on idempotents;
in other words, that $\theta$ is idempotent-separating.
Let $e$ and $f$ be idempotents such that $\theta (e) = \theta (f)$.
Then $e \wedge f \leq e$ and $\theta (e \setminus (e \wedge f)) = 0$.
By assumption, the semigroup $S$ is $0$-simplifying and so $e = e \wedge f$.
By symmetry $f = e \wedge f$.
It follows that $e = f$.
We have proved that $\theta$ is idempotent-separating.
However, we are assuming also that $S$ is fundamental.
This means that $\theta$ is actually injective.
\end{proof}

An inverse semigroup $S \neq \{0\}$ 
is said to be {\em $0$-simple} if its only semigroup ideals
are $\{0\}$ and $S$ itself.
Observe that being $0$-simple is a stronger condition than being $0$-simplifying.
The following was proved as \cite[Proposition 3.2.10]{Lawson1998}.

\begin{lemma}\label{lem:zero-simple} Let $S$ be an inverse semigroup with zero.
Then it is $0$-simple if and only if for any two idempotents $e$ and $f$ there exists an idempotent $i$ such that $e \,\mathscr{D}\, i \leq f$.
\end{lemma}

A non-zero idempotent $e$ is said to be {\em properly infinite} if we may find orthogonal idempotents
$i$ and $j$ such that $e \, \mathscr{D} \, i$ and $f \, \mathscr{D} \, j$ and $i,j \leq e$.
An inverse semigroup with zero is said to be {\em purely infinite} if every non-zero idempotent
is properly infinite.
The proof of the following is \cite[Lemma 4.11]{Lawson2016}.

\begin{lemma}\label{lem:Tarski-zero-simple} Let $S$ be a $0$-simple Tarski monoid.
Then $S$ is purely infinite.
\end{lemma}

The proof of the following is \cite[Theorem 4.16]{Lawson2016}.

\begin{proposition} Let $S$ be a Tarski monoid.
Then $S$ is $0$-simple if and only if $S$ is $0$-simplifying and purely infinite.
\end{proposition}
\begin{proof} One direction is proved by Lemma~\ref{lem:Tarski-zero-simple}.
The other direction is proved in \cite{Lawson2016},
but since the proof there is slightly garbled we give the complete proof here.
It is just a translation of \cite[Proposition~4.11]{Matui13}.
Let $e$ and $f$ be any non-zero idempotents.
Under the assumption that $S$ is $0$-simplifying, 
we may find elements $w_{1}, \ldots, w_{n}$ such that
$e = \bigoplus_{i=1}^{n} \mathbf{d}(w_{i})$ and $\mathbf{r}(w_{i}) \leq f$.
Using the fact that the semigroup is purely infinite,
we may find elements $a$ and $b$ such that $\mathbf{d}(a) = f = \mathbf{d}(b)$,
$\mathbf{r}(a) \perp \mathbf{r}(b)$ and $\mathbf{r}(a), \mathbf{r}(b) \leq f$.
Define elements $v_{1}, \ldots, v_{n}$ as follows:
$v_{1} = a$, $v_{2} = ba$, $v_{3} = b^{2}a$, \ldots, $v_{n} = b^{n-1}a$.
Observe that $\mathbf{d}(v_{1}) = \mathbf{d}(v_{2}) = \ldots = \mathbf{d}(v_{n}) = f$,
and that $\mathbf{r}(v_{i}) \leq f$ for $1 \leq i \leq n$.
The elements $\mathbf{r}(v_{i})$ are pairwise orthogonal.
Consider now the elements $v_{1}w_{1}, \ldots, v_{n}w_{n}$.
Observe that $\mathbf{d}(v_{i}) \geq \mathbf{r}(w_{i})$.
It follows that the domains of these elements are pairwise orthogonal as indeed are their ranges.
These elements are compatible and so we may form their (orthogonal) join:
$w = \bigoplus_{i=1}^{n} v_{i}w_{i}$.
Observe that $\mathbf{d}(w) = e$ and that the ranges, being orthogonal and each less than or equal to $f$ must have a join which is less than or equal to $f$.
\end{proof}

A Boolean inverse semigroup that is fundamental and $0$-simple is {\em congruence-free};
what we mean by this terminology is that there are no non-trivial congruences on $S$ 
of any description. 
See \cite{Lawson1998}.

If $S$ is a Boolean inverse semigroup then we may always write $S = \bigcup_{e \in \mathsf{E}(S)} eSe$.
A Boolean inverse semigroup is said to be {\em $\sigma$-unital}\footnote{A term taken 
from ring theory.} if there is a non-decreasing sequence of idempotents $e_{1} \leq e_{2} \leq \ldots$ 
such that $S = \bigcup_{i=1}^{\infty} e_{i}Se_{i}$. 
We call $\{e_{i} \colon i \in \mathbb{N} \setminus \{0\}\}$ the {\em $\sigma$-unit}.

\begin{example}{\em Let $S$ be a Boolean inverse monoid.
Then the Boolean semigroup $M_{\omega}(S)$ is $\sigma$-unital.}
\end{example}

A topological space is said to be {\em $\sigma$-compact} 
if it is a union of countably many compact spaces.

\begin{proposition} Let $S$ be a Boolean inverse semigroup.
Then $S$ is $\sigma$-unital if and only if the identity space of $\mathsf{G}(S)$ is
$\sigma$-compact.
\end{proposition}
\begin{proof} Suppose first that $S$ is $\sigma$-compact.
If $A$ is an ultrafilter that is also idempotent then it contains an idempotent.
It is immediate that  $\mathsf{G}(S)_{o} = \bigcup_{i=1}^{\infty} \mathscr{U}_{e_{i}}$.
Thus $\mathsf{G}(S)_{o}$ is $\sigma$-compact.
Conversely, suppose that  $\mathsf{G}(S)_{o} = \bigcup_{i=1}^{\infty} U_{i}$,
where each $U_{i}$ is compact.
Observe that  $\mathsf{G}(S)_{o} = \bigcup_{f \in \mathsf{E}(S)} \mathscr{U}_{f}$.
Then we can choose $f_{i}$ such that $\mathsf{G}(S)_{o} = \bigcup_{i=1}^{\infty} \mathscr{U}_{f_{i}}$.
Define $e_{1} = f_{1}$, $e_{2} = f_{1} \vee f_{2}$, $e_{3} = f_{1} \vee f_{2} \vee f_{3}$, \ldots.
Let $s \in S$.
Then we can find an idempotent such that $s = isi$;
for example, $i = \mathbf{d}(s) \vee \mathbf{r}(s)$ will work.
All the ultrafilters in $\mathscr{U}_{i}$ are idempotent.
It follows that $\mathscr{U}_{i} \subseteq \mathscr{U}_{e_{p}}$ for some $p$.
Thus by part (1) of Lemma~\ref{lem:needed-for-later},
we have that $i \leq e_{p}$.
We have therefore shown that $s \in e_{p}Se_{p}$.
\end{proof}

The following table summarizes the different aspects of non-commutative Stone duality we have proved:

\vspace{0.5cm}
\begin{center}
\begin{tabular}{|c||c|}\hline
{\bf Boolean inverse semigroup} & {\bf Boolean groupoid}  \\ \hline \hline
{\small Boolean inverse monoid} & {\small Boolean groupoid with a compact identity space} \\ \hline
{\small Group of units of monoid} & {\small Topological full group} \\ \hline
{\small Countable} & {\small Second-countable}  \\ \hline 
{\small Tarski algebra of idempotents} & {\small Cantor space of identities} \\ \hline
{\small Tarski monoid} & {\small Tarski groupoid} \\ \hline
{\small Semisimple} & {\small Discrete} \\ \hline
{\small Meet-semigroup} & {\small Hausdorff} \\ \hline
{\small Fundamental} & {\small Effective}  \\ \hline
{\small Basic} & {\small Principal and Hausdorff} \\ \hline
{\small $0$-simplifying} & {\small Minimal} \\ \hline
{\small Simple} & {\small Minimal and effective} \\ \hline
{\small $0$-simple Tarski monoid} & {\small Purely infinite and minimal Tarski groupoid} \\ \hline
{\small Congruence-free Tarski monoid} & {\small Purely infinite, minimal and effective Tarski groupoid} \\ \hline
{\small $\sigma$-unital} & {\small Identity space is $\sigma$-compact} \\ \hline
\end{tabular}
\end{center}
\vspace{0.5cm}

\section{Unitization}

Every Boolean inverse semigroup can be embedded (in a nice way) into a Boolean inverse monoid \cite[Definition 6.6.1]{W}.
We shall now obtain this result using our non-commutative Stone duality.

We begin with Lemma~\ref{lem:one-point}.
Let $X$ be a locally compact Boolean space.
Put $X^{\infty} = X \cup \{\infty\}$ and endow $X^{\infty}$ with
the topology that consists of all the open subsets of $X$ together with the complements in $X^{\infty}$ of
the compact sets of $X$ together with $X^{\infty}$ itself.
The following lemma is useful in proving that this really is a topology and will also be needed later.

\begin{lemma}\label{lem:one-point-topology}
The set $U \cap (X^{\infty} \setminus V)$, where $U$ is open in $X$ and $V$ is compact in $X$, is open in $X$.
\end{lemma}
\begin{proof} Observe that since $U \subseteq X$, the intersection in question is actually $U \cap (X \setminus V)$.
But if $U$ is open in $X$ there is a closed set $U_{1}$ in $X$ such that $U = X \setminus U_{1}$.
It follows that the intersection is $X \setminus (U_{1} \cup V)$.
But $V$ is a compact subspace of a Hausdorff space and so is closed by part (2) of Lemma~\ref{lem:topology-needed}.
However, $U_{1} \cup V$ is closed in $X$ and so $X \setminus (U_{1} \cup V)$ is an open subset of $X$.
\end{proof}

The key result is the following (which, of course, is well-known).

\begin{proposition}[One-point compactification]\label{prop:one-point-c} 
Let $X$ be a locally compact Hausdorff space.
Then $X^{\infty}$, with the above topology, is a compact Hausdorff space that contains $X$ as an open subset.
If $X$ is $0$-dimensional so too is $X^{\infty}$.
\end{proposition}
\begin{proof} The proof of the first claim can be found in \cite[Section 37]{Simmons}.
The proof of the second claim is well-known, but we give a proof anyway.
We begin by describing the clopen subsets of $X^{\infty}$.
These are of two types: 
\begin{enumerate}
\item Those that do not contain $\infty$ are precisely the compact-open subsets of $X$.
\item Those that do contain $\infty$ are precisely of the form $X^{\infty} \setminus U$ where $U$ is a compact-open 
subset of $X$.
\end{enumerate}
We now give the proofs of these two claims.
(1) Suppose that $U$ is a clopen subset of $X^{\infty}$ where $\infty \notin U$.
Thus, in particular, $U \subseteq X$.
Now, $U$ is a closed subset of $X^{\infty}$, which is compact.
It follows that $U$ is compact in $X^{\infty}$ and so must be compact in $X$.
It is open by definition.
Thus $U$ is compact-open in $X$.
We now go in the other direction.
Let $U$ be compact-open in $X$.
Then $U$ is open in $X^{\infty}$ by definition.
It remains to prove that $U$ is closed in $X^{\infty}$.
Since $U$ is compact in $X$ it follows by the definition of the topology that $X^{\infty} \setminus U$
is open in $X^{\infty}$.
It follows that $U$ is closed in $X^{\infty}$.
Thus $U$ is clopen in $X^{\infty}$.
(2) Let $U$ be a clopen subset of $X^{\infty}$ that contains $\infty$.
Since $U$ is open and contains $\infty$, we may write $U = X^{\infty} \setminus K$ where $K$ is a compact subset of $X$.
Since $U$ is closed in $X^{\infty}$, there is an open subset $V \subseteq X^{\infty}$ such that
$U = X^{\infty} \setminus V$. Observe that $V$ cannot contain $\infty$ and so must be an open subset of $X$.
It follows that $U$ is the complement in $X^{\infty}$ of a compact-open subset of $X$.
We now go in the opposite direction.
Let $U$ be a compact-open subset of $X$.
We prove that $X^{\infty} \setminus U$ is clopen in $X^{\infty}$.
Since $U$ is compact in $X$, the set $X^{\infty} \setminus U$ is open by definition.
Since $U$ is open in $X$, it is open in $X^{\infty}$ by definition.
It follows that $X^{\infty} \setminus U$ is closed.
We have proved that $X^{\infty} \setminus U$ is clopen.

It remains to prove that the clopen sets form a base for the topology on $X^{\infty}$ if 
$X$ is $0$-dimensional.
We prove that every open subset of $X^{\infty}$ is a union of clopen subsets.
Let $U$ be an open set of $X^{\infty}$.
There are two possibilities.
The first is that $U \subseteq X$ then $U$ is the union of the compact-open subsets of $X$ by Lemma~\ref{lem:boolean-space};
thus it is a union of clopen sets of $X^{\ast}$.
The second is that $U = X^{\infty} \setminus K$, where $K$ is a compact subspace of $X$.
Observe first that $K$ is contained in a compact-open subset of $X$.
We prove that $K$ is in fact equal to the intersection of all the compact-open subsets that contain it.
Let $x \in X$ but $x \notin K$.
Then since $X$ is Hausdorff there are open sets $U$ and $V$ such that
$x \in U$, $K \subseteq V$ and $U \cap V = \varnothing$;
we have used \cite[Theorem 26.C]{Simmons}.
Since $V$ is open it is a union of compact-open subsets.
These cover $K$ which is compact.
We can therefore assume that $V$ is compact-open.
Thus for every point in the complement of $K$ we can find a compact-open set that omits that point and contains $K$.
It follows that $K$ is equal to the intersection of all the compact-open sets that contain $K$.
Whence $X^{\infty} \setminus K = \bigcup_{i} (X^{\infty} \setminus V_{i})$
where each $V_{i}$ is compact-open in $X$ and contains $K$.
\end{proof}

The above result can be used to embed generalized Boolean algebras into Boolean algebras using Stone duality;
this was first proved in \cite{Stone1937} and was touched upon in Lemma~\ref{lem:one-point}.
Let $B$ be a generalized Boolean algebra.
Its Stone space $\mathsf{X}(B)$ is a $0$-dimensional locally compact Hausdorff space.
Construct the one-point compactification of this space to get a $0$-dimensional compact Hausdorff space $\mathsf{X}(B)^{\infty}$.
This gives rise to a Boolean algebra $\mathsf{B}(\mathsf{X}(B)^{\infty})$ into which $B$ can be embedded.
This embedding is summarized by the following (well-known) result.

\begin{proposition}\label{prop:gen-into-unital} Let $B$ be a generalized Boolean algebra without a top element.
Then there is a Boolean algebra $C$ which has $B$ as a subalgebra and order-ideal such that
the elements of $C$ are either $e$ or $e'$, where $e \in B$. 
\end{proposition}

The following example is an illustration of Proposition~\ref{prop:gen-into-unital}.

\begin{example}
{\em Let $B$ be the generalized Boolean algebra of all finite subsets of the set $\mathbb{N}$.
With respect to the discrete topology, $\mathbb{N}$ is a locally compact Boolean space. 
This space can be embedded in the Boolean space $\mathbb{N}^{\infty} = \mathbb{N} \cup \{\infty\}$ which has as open sets
all finite subsets of $\mathbb{N}$ together with all cofinite subsets of $\mathbb{N}$ with $\infty$ adjoined.
The set of all finite subsets of $\mathbb{N}^{\infty}$ which omit $\infty$ together with all cofinite subsets of $\mathbb{N}^{\infty}$
that contain $\infty$ form a Boolean algebra.
This is isomorphic with the Boolean algebra of all finite subsets of $\mathbb{N}$ together with the cofinite subsets of $\mathbb{N}$.}
\end{example}

We shall now extend Proposition~\ref{prop:gen-into-unital} to Boolean inverse semigroups.

Let $G$ be a Boolean groupoid where the space $G_{o}$ is a locally compact Boolean space.
Denote by $G^{\infty}$ the groupoid $G \cup \{\infty\}$ where $\infty$ is a new identity such
that $(G^{\infty})_{o} = G_{o} \cup \{\infty\}$.
Endow $G^{\infty}$ with the topology generated by the base $\Omega (G) \cup \Omega (G_{o}^{\infty})$
(it remains to show that this really is a base);
every open subset of $G^{\infty}$ is therefore the union of an open subset of $G$ and an open subset of $G_{o}^{\infty}$.

\begin{lemma}\label{lem:one-point-groupoid} With the above definitions, 
$G^{\infty}$ is a Boolean groupoid, the identity space of which is compact.
\end{lemma}
\begin{proof} We show first that $\Omega (G) \cup \Omega (G_{o}^{\infty})$ is a base for a topology.
This boils down to checking that if $U$ is an open set of $G$ and $V$ is an open set of $G_{o}^{\infty}$
then $U \cap V$ is an open set.
There are two possibities for $V$.
If $V$ is an open subset of $G_{o}$ then it is also an open subset of $G$, since $G$ is an \'etale groupoid,
and so its space of identities is an open subset.
It follows that $U \cap V$ is open in $G$ and so belongs to our topology.
The other possibility is that $V = G_{0}^{\infty} \setminus K$ where $K$ is a compact subset of $G_{0}$.
But $U$, being in $G$, does not contain $\infty$ so $U \cap V = U \cap (G_{0} \setminus K)$.
Since $G$ is an \'etale groupoid $G_{o}$ is an open subset of $G$. 
We have that $U \cap V = (U \cap G_{o}) \cap (G_{0} \setminus K)$.
We now apply Lemma~\ref{lem:one-point-topology}, to deduce that $U \cap V$ is an open subset of  $G_{o}^{\infty}$.

Next we show that with respect to this topology $G^{\infty}$ is a topological groupoid.
It is clear that $g \mapsto g^{-1}$ is a homeomorphism.
We prove that multiplication
$$\mathbf{m}_{\infty} \colon G^{\infty} \ast G^{\infty} \rightarrow G^{\infty}$$ 
is continuous.
We denote the multiplication on $G$ by $\mathbf{m}$.
The basic open sets of $G^{\infty}$ are of two kinds.
Those in $G$ and those in $G_{o}^{\infty}$.
The former cause us no problems since they do not contain $\infty$
and so the result follows from the fact that $\mathbf{m}$ is continuous.
The open sets of $G_{o}^{\infty}$ are of two kinds.
Those which are simply open subsets of $G_{o}$ and so open subsets of $G$
(since $G$ is \'etale $G_{o}$ is an open subset of $G$)
are dealt with above since they do not contain $\infty$.
Thus the only case we have to deal with are those subsets of the form $U = G_{o}^{\infty} \setminus K$ where $K$ is a compact subset of $G_{o}$.
We have to prove that $\mathbf{m}_{\infty}^{-1} (U)$ is an open subset of $G^{\infty} \ast G^{\infty}$.
Observe first that 
$$\mathbf{m}_{\infty}^{-1} (U)
=
[((G \ast G) \setminus \mathbf{m}^{-1}(K)) \cap \mathbf{m}^{-1}(G_{o})]
\cup 
[(U \times U) \cap (G^{\infty} \ast G^{\infty})].$$
It is easy to show that the left-hand side is contained in the right-hand side;
just recall that the elements of $\mathbf{m}_{\infty}^{-1}(U)$ are of two types:
either ordered pairs $(g,h) \in G \times G$ or $(\infty, \infty)$.
It is also easy to show that the right-hand side is contained in the left-hand side.
It remains to be shown that the right-hand side above is an open subset of $G^{\infty} \ast G^{\infty}$;
to do this we use the product topology on $G^{\infty} \times G^{\infty}$.
The set $K$ is compact and so it is a closed subset of $G_{o}$.
Thus $\mathbf{m}^{-1}(K)$ is a closed subset of $G \ast G$.
Since $G$ is \'etale we know that $G_{o}$ is an open subset of $G$ and so $\mathbf{m}^{-1}(G_{o})$
is an open subset of $G \ast G$.
It follows that the first term is an open subset of $G \ast G$ which is an open subset of $G^{\infty} \ast G^{\infty}$.
The second term is just the intersection of $U \times U$, which is an open subset of $G^{\infty} \times G^{\infty}$,
with $G^{\infty} \ast G^{\infty}$ which gives us an open subset of $G^{\infty} \ast G^{\infty}$.
We have therefore proved that $G^{\infty}$ is a topological groupoid.
It follows that $G^{\infty}$ is an \'etale groupoid (because $G$ is)
and it is Boolean by construction with a compact identity space by Proposition~\ref{prop:one-point-c}. 
\end{proof}

By the above result, $\mathsf{KB}(G^{\infty})$ is a Boolean inverse monoid.
It is clear that $\mathsf{KB}(G)$ embeds into $\mathsf{KB}(G^{\infty})$.
Observe that $\mathsf{KB}(G)$ is closed under binary joins taken in $\mathsf{KB}(G^{\infty})$.
The generalized Boolean algebra $\mathsf{B}(G_{o})$ is a subalgebra of the Boolean algebra $\mathsf{B}(G_{o}^{\infty})$.
We have therefore proved the following.

\begin{lemma}\label{lem:subalgebra} With the above definition,
$\mathsf{KB}(G)$ is a subalgebra of  $\mathsf{KB}(G^{\infty})$.
\end{lemma}

We have therefore embedded a Boolean inverse semigroup into a Boolean inverse monoid
where the generalized Boolean algebra of the former is embedded into the Boolean algebra of the latter.

We now describe the elements $A \in \mathsf{KB}(G^{\infty})$, the compact-open local bisections of $G^{\infty}$.
This will provide the connection with \cite[Definition 6.6.1]{W}.
There are two cases.
Either $\infty \notin A$ or $\infty \in A$.
In the former case, $A$ is just a compact-open local bisection of $G$
and so an element of $\mathsf{KB}(G)$.
We therefore deal with the latter case.
Let $A \in \mathsf{KB}(G^{\infty})$ contain $\infty$.
Since it is open, $A$ can be written as a union $A = U \cup V$ where $U$ is an open subset of $G$
and $V = G_{o} \setminus K$, where $K$ is a compact subset of $G_{o}$.
We now use the fact that $G^{\infty}$ is an \'etale groupoid and so has a base consisting of compact-open local bisections.
We may therefore write $U = \bigcup_{i \in I} B_{i}$ and $V = \bigcup_{j \in J} C_{j}$
where the $B_{i}$ and $C_{j}$ are compact-open local bisections of $G^{\infty}$.
We now use the fact that $A$ is compact to deduce that 
$A = (B_{1} \cup \ldots B_{m}) \cup (C_{1} \cup \ldots \cup C_{n})$.
Now, $B_{1}, \ldots, B_{m} \subseteq U$ and so each subset omits $\infty$.
In addition, $B_{i} \subseteq A$ so they are pairwise compatible.
It follows that $B = B_{1} \cup \ldots \cup B_{m}$ is a compact-open local bisection of $G$.
The union $C_{1} \cup \ldots \cup C_{n}$ is a compact-open local bisection which is an idempotent in $\mathsf{KB}(G^{\infty})$
and contains $\infty$.
It follows that each $C_{i} \subseteq G_{o}^{\infty}$ and is clopen.
We now use the description of the clopen subsets of $G_{o}^{\infty}$ given in Proposition~\ref{prop:one-point-c}.
Those $C_{i}$ which are compact-open subsets of $G_{o}$ can be absorbed into $B$.
We may therefore assume that each $C_{i} = G_{o}^{\infty} \setminus K_{i}$, where each $K_{i}$ is compact-open in $G_{o}^{\infty}$.
It follows that $A = B \cup (G_{o}^{\infty} \setminus D)$, where $B$ is an element of $\mathsf{KB}(G)$
and $D$ is a compact-open bisection and idempotent $\mathsf{KB}(G)$.
The comparison with Wehrung's construction, \cite[Definition 6.6.1]{W}, 
follows from the next lemma.

\begin{lemma}\label{lem:one-point-groupoid}
Let $S$ be a Boolean inverse monoid.
Let $x = e' \vee a$ where $e \in \mathsf{E}(S)$.
Then $x = e' \vee eae$, where clearly $e' \perp eae$.
\end{lemma} 
\begin{proof} By assumption $e' \sim a$.
It follows that $e'a$ and $e'a^{-1} = ae'$ are both idempotents.
We have that $1 = e \vee e'$.
Thus $x = (e \vee e')(e' \vee a) = e' \vee ea \vee e'a$.
But $e'a$ is an idempotent less than $e$.
It follows that $x = e' \vee ea$.
Whence, $x = (e' \vee ea)(e \vee e') = e' \vee eae \vee eae'$.
But $ae'$ is an idempotent and so $eae' = (ae')e = 0$.
We have therefore proved that $x = e' \vee eae$.
\end{proof}

Taking into account Lemma~\ref{lem:subalgebra}, Lemma~\ref{lem:one-point-groupoid} 
and using our non-commutative Stone duality,
we have therefore proved the following, a result first established by Wehrung.

\begin{proposition}[Unitization]\label{prop:wehrung} Let $S$ be a Boolean inverse semigroup which is not a monoid.
Then there is a Boolean inverse monoid $T$ containing $S$ as a subalgebra and ideal
such that each element of $T \setminus S$ is of the form $e' \vee s$ where $e \in \mathsf{E}(S)$
and $s \in eSe$.
\end{proposition} 

\begin{remark}{\em 
Our definition of $G^{\infty}$ is by means of a base extension as in \cite{MR}, although our approach is quite different.}
\end{remark}

We can define the group of units of a Boolean inverse {\em semigroup} $S$ as follows, independently of what we did above.
For each idempotent $e \in \mathsf{E}(S)$ define the group $G_{e}$ to be the group
of units of the Boolean inverse monoid $eSe$.
If $e \leq f$ define a map $\phi^{e}_{f} \colon G_{e} \rightarrow G_{f}$ by $a \mapsto a \vee (f \setminus e)$.
It is easy to check that this is a well-defined injective function that maps $e$ to $f$
and is a group homomorphism.
Observe that $\phi^{e}_{e}$ is the identity function on $G_{e}$,
and if $e \leq f \leq g$ then $(g \setminus e) = (g \setminus f) \vee (f \setminus e)$ and so
$\phi^{e}_{g} = \phi^{f}_{g}\phi^{e}_{f}$.
It follows that we have an ($E$-unitary) strong semilattice of groups $\{G_{e},\phi^{e}_{f} \colon e,f \in \mathsf{E}(S)\}$.
Such a system gives rise to an inverse semigroup with central idempotents as follows.
Put $T = \bigcup_{e \in \mathsf{E}(S)}G_{e}$, a disjoint union, with product $\circ$ defined as follows:
$$a \circ b = \phi^{e}_{e \vee f}(a) \phi^{f}_{e \wedge f}(b)$$
where $a \in G_{e}$ and $b \in G_{f}$.
See \cite[Section 5.2, page 144]{Lawson1998} for details.
Put $\mathsf{C}(S) = (T,\circ)$.
Observe that $a \leq b$ in $\mathsf{C}(S)$ precisely when $a \in G_{e}$ and $b \in G_{f}$
and $e \leq f$ in $S$ and $\phi^{e}_{e \vee f}(a) = b$. 
An inverse semigroup with central idempotents is called a {\em Clifford semigroup}.
An inverse semigroup is said to be {\em $E$-unitary} if $e \leq a$, where $e$ is an idempotent,
implies that $a$ is an idempotent.

\begin{proposition}\label{prop:clifford} With each Boolean inverse semigroup $S$, 
we can associate an $E$-unitary Clifford semigroup
$\mathsf{C}(S)$. 
The meet semilattice of this semigroup is $(\mathsf{E}(S),\vee)$.
\end{proposition}

With each inverse semigroup $S$ we can define the {\em minimum group congruence $\sigma$} which has the property that
$S/\sigma$ is a group. See \cite[Section 2.4]{Lawson1998}.
In the case that $S$ is $E$-unitary, it turns out that $\sigma \, = \, \sim$.
See \cite[Theorem 2.4.6]{Lawson1998}.\\

\noindent
{\bf Definition.} Let $S$ be a Boolean inverse semigroup.
We define the {\em group of units} of $S$, denoted by $\mathsf{U}(S)$,
to be the group $\mathsf{C}(S)/\sigma$.\\

It is routine to check that in the case where $S$ is a Boolean inverse {\em monoid}, 
the usual definition of the group of units is returned.
The group of units we have defined in terms of Boolean inverse semigroups is the same
as the group defined in \cite[Remark 3.10]{NO}, 
when you recall the connection between quotients of certain inverse semigroups and directed colimits \cite[Section 2.5]{LS2}.
The following lemmas will needed for the proof of the last proposition of this section.
The proof of the next lemma uses Lemma~\ref{lem:buffs1} and is routine.

\begin{lemma}\label{lem:matt} Let $S$ be a Boolean inverse monoid where $e$ and $e_{1}$ are idempotents.
Suppose that
$a = b \oplus e'$, where $\mathbf{d}(b) = \mathbf{r}(b) = e$,
and 
$a = b_{1} \oplus e_{1}'$, where $\mathbf{d}(b_{1}) = \mathbf{r}(b_{1}) = e_{1}$.
Then $b \sim b_{1}$ in $S$.
Put $x = b \wedge b_{1}$.
Then $\mathbf{d}(x) = \mathbf{r}(x) = ee_{1}$.
We have that
$b = x \vee e(ee_{1})'$ and $b_{1} = x \vee e_{1}(ee_{1})'$.
\end{lemma}

The proof of the following is routine.

\begin{lemma}\label{lem:beard} Let $S$ be a Boolean inverse monoid where $e$ and $e_{1}$ are idempotents.
Suppose that
$\mathbf{d}(b) = \mathbf{r}(b) = e$,
and 
$\mathbf{d}(b_{1}) = \mathbf{r}(b_{1}) = e_{1}$
and there is an element $x$, 
such that $\mathbf{d}(x) = \mathbf{r}(x) = y \geq e,e_{1}$,
where $b = x \vee ey'$ and $b_{1} = x \vee e_{1}y'$.
Then $b \vee e' = b_{1} \vee e_{1}'$.
\end{lemma}

The proof of the next result is by means of a direct verification.

\begin{lemma}\label{lem:walsh} Let $S$ be a Boolean inverse monoid where $e$ and $f$ are idempotents.
Suppose that 
$a = a_{1} \oplus e'$, where $\mathbf{d}(a_{1}) = \mathbf{r}(a_{1}) = e$,
and 
$b = b_{1} \oplus f'$, where $\mathbf{d}(b_{1}) = \mathbf{r}(b_{1}) = f$.
Then
$$(a_{1} \oplus (e \vee f)e')(b_{1} \oplus (e \vee f)f') \oplus e'f' = ab.$$
\end{lemma}

\begin{proposition}\label{prop:group-of-units} 
Let $G$ be a Boolean groupoid, the identity space of which is locally compact and put $S = \mathsf{KB}(G)$.
Then the group of units of $\mathsf{KB}(G^{\infty})$ (as defined above) is isomorphic to the topological full group of $G^{\infty}$.
\end{proposition}
\begin{proof}  
We first of all establish a bijection between the elements of the topological full group and the elements of the group of units.
We begin by describing the elements of the topological full group of $G^{\infty}$.
A compact-open bisection $A$ of $G^{\infty}$ is a compact-open local bisection such that
$A^{-1}A = AA^{-1} = G_{o}^{\infty}$.
We know also that $A$ can be written in the form $A = B \cup C$, a disjoint union (by Lemma~\ref{lem:one-point-groupoid}),
where $B$ is a compact-open local bisection of $G$ and $C$ is of the form
$C = G_{o}^{\infty}\setminus D$ where $D$ is a compact-open subset of $G_{o}$ and $B = DBD$.
Because of the disjointness, we have that $B^{-1}B = BB^{-1} = E$ and $C = G_{o}^{\infty}\setminus E$.
The above representation for $A$ is, of course, not unique.
Supose that $A = B_{1} \cup C_{1}$, where 
$B_{1}^{-1}B_{1} = B_{1}B_{1}^{-1} = E_{1}$ 
and 
$C_{1} = G_{o}^{\infty}\setminus E_{1}$.
The fact that $B \sim B_{1}$ in $\mathsf{C}(S)$ follows by Lemma~\ref{lem:matt}.
We may therefore map $A$ to $\sigma (B)$.
Now, suppose that $B$ and $B_{1}$ are compact-open local bisections of $G$
such that $B^{-1}B = BB^{-1} = E$, $B_{1}^{-1}B_{1} = B_{1}B_{1}^{-1} = E_{1}$ and $\sigma (B) = \sigma (B_{1})$.
Then by Lemma~\ref{lem:beard}, we have that 
$$A = B \cup (G_{o}^{\infty} \setminus E) 
= 
B_{1} \cup (G_{o}^{\infty} \setminus E_{1})$$
where $A$ is a compact-open bisection of $G^{\infty}$
such that $A$ maps to $\sigma (B) = \sigma (B_{1})$.
We have therefore established our bijection.
The proof that we have a homomorphism and so an isomorphism follows by Lemma~\ref{lem:walsh}.
\end{proof}

\section{Final remarks}

A {\em frame} is a complete infinitely distributive lattice \cite[Page 39]{Johnstone}.
A {\em completely prime filter} in a frame $L$ is a proper filter $F$ such that
$\bigvee_{i \in I} a_{i} \in F$ implies that $a_{i} \in F$ for some $i \in I$.
If $X$ is a topological space then its set of open subsets $\Omega (X)$ is a frame.
If $x \in X$ denote by $\mathscr{G}_{x}$ the set of all open subsets of $X$ that contains $x$.
The set  $\mathscr{G}_{x}$ is a completely prime fiter of $\Omega (X)$.
We refer to the completely prime filters in $\Omega (X)$ as {\em points}
and denote the set of all points in $X$ by $\mbox{\rm pt}(\Omega (X))$.
We have therefore defined a map $X \rightarrow \mbox{\rm pt}(\Omega (X))$
given by $x \mapsto \mathscr{G}_{x}$.
We say that the space $X$ is {\em sober} if this map is a bijection.
See \cite[Page 43]{Johnstone}. 
We say that a frame $L$ is {\em spatial} if for any elements $a,b \in L$
such that $a \nleq b$ then there is a completely prime filter that contains $a$ but omits $b$.
See \cite[Page 43]{Johnstone}.

We say that a topological space is {\em spectral} if it is sober and has a base of compact-open subsets with the additional property that the intersection
of any two compact-open subsets is itself compact-open;
This is different from the definition given in \cite{Johnstone} since we do not assume that the space be compact.
A {\em spectral groupoid} is an \'etale topological groupoid whose space of identities forms a spectral space.
If $S$ is a distributive inverse semigroup then $\mathsf{G}(S)$, the set of prime filters of $S$,
is a spectral groupoid; if $G$ is a spectral groupoid then $\mathsf{KB}(G)$, the set of compact-open local bisections of $G$,
is a distributive inverse semigroup.
The following is proved as \cite[Theorem 3.17]{LL}.

\begin{theorem}[Non-commutative Stone duality for distributive inverse semigroups]\label{them:dist-stone}
The category of distributive inverse semigroups and their callitic morphisms is dually 
equivalent to the category coherent continuous covering functors.
\end{theorem}

By Lemma~\ref{lem:boolean-space} and the fact that Hausdorff spaces are sober \cite[part (ii) of Lemma II.1.6]{Johnstone}, 
a special case of the above theorem is Theorem~\ref{them:non-com-stone}
since the Hausdorff spectral spaces are precisely the locally compact Boolean spaces.

By a {\em pseudogroup} we mean an inverse semigroup which has arbitrary compatible joins
and multiplication distributes over such joins.
Observe that pseudogroups are automatically meet-monoids.
In addition, if $S$ is a pseudogroup then $\mathsf{E}(S)$ is a frame.
Pseudogroups were studied historically by Boris Schein \cite{Schein} and more recently by Resende \cite{Resende2}.
The connection between pseudogroups and distributive inverse semigroups is provided by the notion of `coherence'.
Let $S$ be a pseudogroup.
An element $a \in S$ is said to be {\em finite} if $a \leq \bigvee_{i \in I} x_{i}$, 
where $\{x_{i} \colon i \in I\}$ is any compatible subset of $S$, 
then there exists a finite subset $x_{1},\ldots, x_{n}$ such that $a \leq \vee_{i=1}^{n} x_{i}$.
Denote the set of finite elements of $S$ by $\mathsf{K}(S)$.
We say that the pseudogroup $S$ is {\em coherent} if $\mathsf{K}(S)$ is a distributive
inverse semigroup and every element of $S$ is a join of a compatible subset of $\mathsf{K}(S)$.
It can be shown that every distributive inverse semigroup arises from
some pseudogroup as its set of finite elements \cite[Proposition 3.5]{LL}.

Let $G$ be an \'etale groupoid.
Then the set of open local bisections of $G$, denoted by $\mathsf{B}(G)$,
forms a pseudogroup \cite[Proposition 2.1]{LL}.
Let $S$ be a pseudogroup.
A {\em completely prime filter} $A \subseteq S$ is a proper filter in $S$ 
with the property that $\bigvee_{i \in I} x_{i} \in A$ implies that $x_{i} \in A$ for some $i$.
Denote the set of completely prime filters on $S$ by $\mathsf{G}_{CP}(S)$.
Then $\mathsf{G}_{CP}(S)$ is an \'etale groupoid \cite[Proposition 2.8]{LL}.
A homomorphism $\theta \colon S \rightarrow T$ is said to be {\em hypercallitic}\footnote{The definition we have given here looks different from the one given in \cite{LL}
but is equivalent by using \cite{Resende3}.}
if it preserves arbitrary joins, preserves binary meets, and has the
property that for each $t \in T$ we may write $t = \bigvee_{i \in I} t_{i}$
where $t_{i} \leq \theta (s_{i})$ for each $i \in I$.
Denote by $\mbox{\bf Pseudo}$ the category of pseudogroups and hypercallitic maps.
Denote by $\mbox{\bf Etale}$ the category of \'etale groupoids and continuous covering functors.
If $\theta \colon S \rightarrow T$ is hypercallitic,
define $\mathsf{G}_{CP}(\theta)$ to be $\theta^{-1}$;
by \cite[Lemma 2.14 and Lemma 2.16]{LL}, it follows that $\mathsf{G}_{CP}(\theta)$ is a continuous covering functor
from $\mathsf{G}_{CP}(T)$ to $\mathsf{G}_{CP}(S)$.
On the other hand, if $\phi \colon G \rightarrow H$ is a continuous covering functor
then $\mathsf{B}(\phi) = \phi^{-1}$ is a hypercallitic map from $\mathsf{B}(H)$ to $\mathsf{B}(G)$ by
\cite[Lemma 2.19]{LL}.
The following is immediate by \cite[Corollary 2.18, Theorem 2.22]{LL}.

\begin{theorem}[The adjunction theorem]\label{them:adjunction-theorem} The functor
$\mathsf{G}_{CP} \colon \mbox{\bf Pseudo}^{op} \rightarrow \mbox{\bf Etale}$ is right
adjoint to the functor $\mathsf{B} \colon \mbox{\bf Etale} \rightarrow \mbox{\bf Pseudo}^{op}$.
\end{theorem}

The above theorem is a generalization of \cite[Theorem II.1.4]{Johnstone} from frames/locales to pseudogroups.
It can also be used to prove Theorem~\ref{them:dist-stone}, which is the approach adopted in \cite{LL}.

Let $S$ be a pseudogroup.
If $s \in S$, denote by $\mathscr{X}_{s}$ the set of all completely prime filters in $S$ containing the element $s$.
The function $\varepsilon \colon S \rightarrow \mathsf{B}(\mathsf{G}_{CP}(S))$ given by $s \mapsto \mathscr{X}_{s}$
is a surjective callic morphism \cite[Proposition 2.9, part (1) of Proposition 2.12, Corollary 2.18, Lemma 2.21]{LL}.
We say that $S$ is {\em spatial} if $\varepsilon$ is injective.
Denote by  $\mbox{\bf Pseudo}_{sp}$ the category of spatial pseudogroups and hyperallitic morphisms.

\begin{lemma}\label{lem:spatial} 
The pseudogroup $S$ is spatial if and only if the frame $\mathsf{E}(S)$ is spatial as a frame.
\end{lemma}
\begin{proof} We denote by $\mathscr{X}_{s}$ the set of all completely prime filters in $S$ that contain $s$.
We denote by $\widetilde{\mathscr{X}_{e}}$  the set of all completely prime filters in $\mathsf{E}(S)$ that contain $e$.
Assume first that $S$ is spatial.
Suppose that for idempotents $e$ and $f$ we have that $\widetilde{\mathscr{X}_{e}} = \widetilde{\mathscr{X}_{f}}$.
We use the fact that every idempotent filter in $S$ is determined by the idempotents it contains. 
It follows that $\mathscr{X}_{e} = \mathscr{X}_{f}$ and so, by assumption, $e = f$.
We now assume that $\mathsf{E}(S)$ is spatial.
Suppose that $a$ and $b$ are elements of $S$ such that $\mathscr{X}_{a} = \mathscr{X}_{b}$ 
Using \cite[part (2) and (3) of Lemma 2.6]{LL}, it follows that every element of  $\mathscr{X}_{a^{-1}b}$ 
contains an idempotent. Similarly for $\mathscr{X}_{ab^{-1}}$.
It follows that 
$\mathscr{X}_{a^{-1}b} = \mathscr{X}_{a^{-1}b \wedge 1}$
and  
$\mathscr{X}_{ab^{-1}} = \mathscr{X}_{ab^{-1} \wedge 1}$
where $a^{-1}b \wedge 1$ and $ab^{-1} \wedge 1$ are idempotents.
We have that
$\mathscr{X}_{a^{-1}b} = \mathscr{X}_{a^{-1}a}$.
We now use \cite[part (2) of Lemma 2.2]{LL} and the fact that $\mathsf{E}(S)$ is spatial
to deduce that $a^{-1}b \wedge 1 = a^{-1}a$.
Similarly, $ab^{-1} \wedge 1 = bb^{-1}$.
We have that $a^{-1}a \leq a^{-1}b$ from which we deduce that $a \leq b$.
Similarly, $bb^{-1} \leq ab^{-1}$ from which we deduce that $b \leq a$.
It follows that $a = b$, as required.
\end{proof}

Let $G$ be an \'etale groupoid.
For each $g \in G$, denote by $\mathscr{C}_{g}$ the set of all open local bisections of $G$ containing the element $g$.
The function $\eta \colon G \rightarrow \mathsf{G}_{CP}(\mathsf{B}(G))$ given by $g \mapsto \mathscr{C}_{g}$
is a continuous covering functor by \cite[Proposition 2.11]{LL}.
We say that $G$ is {\em sober} if $\eta$ is a homeomorphism.
Denote by $\mbox{\bf Etale}_{so}$ the category of sober \'etale groupoids and continuous covering functors.
 
\begin{lemma}\label{lem:sober} Let $G$ be an \'etale groupoid.
\begin{enumerate}
\item Each open set in a completely prime filter of open subsets $F$ contains as a subset an open local bisection also in $F$.
\item Completely prime filters of open subsets are determined by the open local bisections they contain.
\item If $F$ is a completely prime filter of open sets in $G$ then 
$$\mathbf{d}(F) = \{ \mathbf{d}(U) \colon U \in F\}$$
is a completely prime filter of open sets in $G_{o}$.
\item If $F$ is a completely prime filter of open sets and $U \in F$ is an open local bisection
then $U\mathbf{d}(F)$ consists of open local bisections and $F = (U\mathbf{d}(F))^{\uparrow}$.
\item $G$ is a $T_{0}$-space if and only if $G_{o}$ is a $T_{0}$-space.
\item $G$ is sober if and only if $G_{o}$ is sober.
\end{enumerate}
\end{lemma}
\begin{proof} (1) We use the fact that in an \'etale topological groupoid the open local bisections form a basis for the topology.
Let $F$ be a completely prime filter of open sets and let $U \in F$ be any element.
Then $U = \bigcup_{i \in I} U_{i}$, where the $U_{i}$ are open local bisections.
But $F$ is completely prime and so $U_{i} \in F$ for some $i \in I$.
Thus $U_{i} \subseteq U$ is an open local bisection and also belongs to $F$. 

(2) Denote by $F'$ the set of all open local bisections in the completely prime filter of open sets $F$.
Then $F'$ is closed under finite intersections and  $(F')^{\uparrow} = F$.
It is now easy to check that if$A$ and $B$ are completely prime filters
and $A = (A')^{\uparrow}$ and $B = (B')^{\uparrow}$ then $A = B$ if and only if $A' = B'$.

(3) Observe first that
$\mathbf{d}(F) = \{ \mathbf{d}(U) \colon U \in F \}$
is a set of open subsets of $G_{o}$ because $G$ is \'etale and so $\mathbf{d}$ is a local homeomorphism and so an open map.
Let $\mathbf{d}(U) \subseteq X$ where $X$ is an open subset of $G_{o}$.
Then $GX$ is an open subset of $G$ since $G$ is \'etale and $\mathbf{d}(GX) = X$.
But $U = U\mathbf{d}(U) \subseteq GX$ and so $GX \in F$.
It follows that $X \in \mathbf{d}(F)$.
Let $\mathbf{d}(U), \mathbf{d}(V) \in \mathbf{d}(F)$.
Then $\mathbf{d}(U \cap V) \subseteq \mathbf{d}(U) \cap \mathbf{d}(V)$.
It follows that $\mathbf{d}(U) \cap \mathbf{d}(V) \in \mathbf{d}(F)$.
Let $\bigcup_{i \in X} X_{i} \in \mathbf{d}(F)$.
Then there exists $U \in F$ such that $\mathbf{d}(U) = \bigcup_{i \in X} X_{i} \in \mathbf{d}(F)$.
But $U = \bigcup_{i \in I} UX_{i}$ and the result now follows.

(4) Let $F$ be any completely prime filter of open subsets.
By part (3), we have proved that $\mathbf{d}(F)$ is a completely prime filter of open subsets of $G_{o}$.
By part (1),  let $U \in F$ be any open local bisection.
Observe that $U = UU^{-1}U$ and that,
more generally, if $U,V,W \in F$ are open local bisections then $UV^{-1}W \in F$.
In addition, when $W$ is an open local bisection we have that $\mathbf{d}(W) = W^{-1}W$.
We prove that  if $V \in F$ then $U \mathbf{d}(V) \in F$.
Since $F$ is completely prime and $G$ is \'etale,
we can find an open local bisection $W$ such that $W \subseteq V$.
Thus $V \mathbf{d}(W) = VW^{-1}W \in F$.
But $V\mathbf{d}(W) \subseteq U\mathbf{d}(V)$ and so $U\mathbf{d}(V) \in F$.
We have shown that $U\mathbf{d}(A) \subseteq A'$.
Now let $W \in F$ be any open local bisection.
Then $U \cap W \in F$ is an open local bisection
But $U \cap W = U\mathbf{d}(U \cap W)$.
Thus $W$ contains as a subset an element of $U\mathbf{d}(A)$. 
But every element of $F$ contains as a subset an open local bisection in $F$.
It follows that $F = (U\mathbf{d}(F))^{\uparrow} \subseteq F$.

(5) If $G$ is $T_{0}$ then it is immediate that $G_{o}$ is $T_{0}$ because in an \'etale groupoid
the space of identities forms an open subspace.
Suppose now that $G_{o}$ is $T_{0}$.
We shall prove that $G$ is $T_{0}$.
Let $g,h \in G$ be distinct elements of $G$.
There are two cases.
First, suppose that $\mathbf{d}(g) \neq \mathbf{d}(h)$.
Since $G_{o}$ is $T_{0}$ we can, without loss of generality, assume that there is an open subset $X \subseteq G_{o}$
that contains $\mathbf{d}(g)$ but does not contain $\mathbf{d}(h)$.
Put $Y = \mathbf{d}^{-1}(X)$.
Then $Y$ is an open set that contains $g$ but does not contain $h$.
Second, suppose that $\mathbf{d}(g)  = \mathbf{d}(h)$.
If the set of all open sets that contains $g$ were the same as the set all open sets that contains $h$
{\em then there would be an open local bisection that contained both $g$ and $h$.}
This cannot happen because $\mathbf{d}(g) = \mathbf{d}(h)$ and $g$ and $h$ are distinct.
It follows that there must be an open set that contains one of $g$ or $h$ but not the other. 

(6) If $G$ is sober it is easy to check that $G_{o}$ is sober.
Suppose now that $G_{o}$ is sober. 
We prove that $G$ is sober.
Let $F$ be a completely prime filter of open subsets of $G$.
Then by part (3), $\mathbf{d}(F)$ is a completely prime filter of open subsets of $G_{o}$.
From the assumption that $G_{o}$ is sober, there is a unique identity $e \in G_{o}$
such that $\mathbf{d}(F)$ is precisely the set of all open subsets of $G_{o}$ that contain $e$.
Choose any $U \in F$ an open local bisection by part (1).
Then $e \in \mathbf{d}(U)$.
But $U$ is a local bisection and so there is a unique $g \in U$ such that $\mathbf{d}(g) = e$.
Let $V$ be any other open local bisection in $F$.
Then there is a unique element $h \in V$ such that $\mathbf{d}(h) = e$.
But $U \cap V$ is also an open local bisection in $F$.
Thus there is a unique element $k \in U \cap V$ such that $\mathbf{d})(k) = e$.
But $k \in U$ and $k \in V$ thus $g = k = h$.
It follows that all the open local bisections in $F$ contain $g$ and so all elements of $F$ contain $g$ by part (4).
We have therefore proved that $F \subseteq F_{g}$.
But $\mathbf{d}(F) = \mathbf{d}(F_{g})$.
Thus by part (4), we must have that $F = F_{g}$.
We have proved that every completely prime filter of open sets in $G$ is determined by an element of $G$.
To complete the proof, suppose that $F_{g} = F_{h}$.
Then $g = h$ because by part (5), the space $G$ is $T_{0}$.
\end{proof}

The following was proved as \cite[Proposition 2.12]{LL}.

\begin{lemma}\label{lem:spatial-sober}\mbox{}
\begin{enumerate}
\item For each pseudogroup $S$, the \'etale groupoid $\mathsf{G}_{CP}(S)$ is sober.
\item For each \'etale groupoid $G$, the pseudogroup $\mathsf{B}(G)$ is spatial. 
\end{enumerate}
\end{lemma}

From Theorem~\ref{them:adjunction-theorem} and what we have said above
we obtain the following \cite[Theorem 2.23]{LL} which is the basis of all of our duality theorems.

\begin{theorem}[Duality theorem between spatial pseudogroups and sober \'etale groupoids]
The category  $\mbox{\bf Pseudo}_{sp}^{op}$ is equivalent to the category $\mbox{\bf Etale}_{so}$.
\end{theorem}

Our whole approach is predicated on the idea that suitable classes of inverse semigroups 
can be viewed as generalizations of suitable classes of lattices. 
The following table summarizes our approach:

\vspace{0.5cm}
\begin{center}
\begin{tabular}{|c||c|}\hline
{\bf Lattices} & {\bf Non-commutative lattices}  \\ \hline \hline
{\small Meet semilattices} & {\small Inverse semigroups} \\ \hline
{\small Frames} & {\small Pseudogroups} \\ \hline
{\small Distributive lattices} & {\small Distributive inverse semigroups}  \\ \hline 
{\small Generalized Boolean algebras} & {\small Boolean inverse semigroups} \\ \hline
\end{tabular}
\end{center}
\vspace{0.5cm}

The theory of coverages on inverse semigroups (as a way of constructing pseudogroups)
is touched on in \cite[Section 4]{LL}.
Weaker axioms for a coverage (which more naturally generalize the meet-semilattice case)
are discussed in \cite{Castro}.
Although we discuss Paterson's Universal groupoid in \cite[Section 5.1]{LL}
as well as Booleanizations, much better presentations of these results can be found in \cite{Lawson2020}. 
Tight completions of inverse semigroups, the subject of \cite[Section 5.2]{LL},
are discussed in greater generality in \cite{LV}.



\begin{thebibliography}{99}



\bibitem{Burris} S. Burris, The laws of Boole's thought, \url{http://www.math.uwaterloo.ca/~snburris/htdocs/MYWORKS/PREPRINTS/aboole.pdf}.

\bibitem{BS} S. Burris, H. P. Sankappanavar, {\em A course in universal algebra}, The Millennium Edition, freely available from \url{http://math.hawaii.edu/~ralph/Classes/619/}.

\bibitem{Castro} G. G. de Castro, Coverages on inverse semigroups, {\em Semigroup Forum} {\bf 102} (2021), 375--396.

\bibitem{Doctor} H. P. Doctor, The categories of Boolean lattices, Boolean rings and Boolean spaces, {\em Canad. Math. Bull.} {\bf 7} (1964), 245--252.

\bibitem{E} Ch.~Ehresmann, {\em Oeuvres compl\`etes et comment\'ees}, (ed A.~C.~Ehresmann) 
Supplements to {\em Cah. Top. G\'eom.  Diff\'er. Cat\'eg.}, Amiens, 1980--83.

\bibitem{Exel} R. Exel, Inverse semigroups and combinatorial $C^{\ast}$-algebras, {\em Bull. Braz. Maths. Soc. (N.S.)} {\bf 39} (2008), 191--313.

\bibitem{Foster} A. L. Foster, The idempotent elements of a commutative ring form a Boolean algebra; ring duality and transformation theory,
{\em Duke Math. J.} {\bf 12} (1945), 143--152.

\bibitem{G} M. Gehrke, S. Grigorieff, J.-E. Pin, {\em Duality and equational theory of regular languages}, 
Lecture Notes in Computer Science 5126, Springer Verlag, 2008, pp~246--257. 

\bibitem{GH} S. Givant, P. Halmos, {\em Introduction to Boolean algebras}, Springer, 2009.

\bibitem{Higgins} Ph. J. Higgins, {\em Categories and groupoids}, Van Nostrand Reinhold Company, London, 1971.

\bibitem{TH} T. Hailperin, Boole's algebra isn't Boolean algebra, {\em Math. Mag.} {\bf 54} (1981), 172--184.

\bibitem{Johnstone} P.~T.~Johnstone, {\em Stone spaces}, CUP, 1986.

\bibitem{Kell1} J. Kellendonk, The local structure of tilings and their integer groups of coinvariants, {\em Commun. Math. Phys.}
{\bf 187} (1997), 115--157.

\bibitem{Kell2} J. Kellendonk, Topological equivalence of tilings, {\em J. Math. Phys.} {\bf 38} (1997), 1823--1842.

\bibitem{Koppelberg} S. Koppelberg, {\em Handbook of Boolean algebra Volume 1}, North-Holland, 1989.

\bibitem{KLLR} G. Kudryavtseva, M. V. Lawson, D. H. Lenz, P. Resende, Invariant means on Boolean inverse monoids, {\em Semigroup Forum} {\bf 92} (2016), 77--101.

\bibitem{KL} G. Kudryavtseva, M. V. Lawson, Perspectives on non-commutative frame theory, {\em Adv. Math.} {\bf 311} (2017), 378--468.

\bibitem{K} A.~Kumjian, On localizations and simple $C^{\ast}$-algebras, {\em Pacific J. Math.} {\bf 112} (1984), 141--192.

\bibitem{Lawson1996} M. V. Lawson, Coverings and embeddings of inverse semigroups, {\em Proc. Edinb. Math. Soc.} {\bf 36} (1996), 399--419. 

\bibitem{Lawson1998} M.~V.~Lawson, {\em Inverse semigroups: the theory of partial symmetries}, World Scientific, 1998.

\bibitem{Lawson2003} M. V. Lawson, {\em Finite automata}, Chapman and Hall/CRC, 2003.

\bibitem{Lawson2007} M. V. Lawson, The polycyclic monoids $P_{n}$ and the Thompson groups $V_{n,1}$, {\em Comm. Algebra} {\bf 35} (2007), 4068--4087.

\bibitem{Lawson2007b} M. V. Lawson, A class of subgroups of Thompson's group $V$, {\em Semigroup Forum} {\bf 75} (2007), 241--252.

\bibitem{Lawson3} M.~V.~Lawson, A non-commutative generalization of Stone duality, {\em J. Aust. Math. Soc.} {\bf 88} (2010), 385--404.

\bibitem{Lawson5} M.~V.~Lawson, Non-commutative Stone duality: inverse semigroups, topological groupoids and $C^{\ast}$-algebras,  
{\em Int. J. Algebra Comput.} {\bf 22}, 1250058 (2012) DOI:10.1142/S0218196712500580. 

\bibitem{Lawson2016} M.~V.~Lawson, Subgroups of the group of homeomorphisms of the Cantor space and a duality 
between a class of inverse monoids and a class of Hausdorff \'etale groupoids, 
{\em J. Algebra} {\bf 462} (2016), 77--114.

\bibitem{Lawson2017} M. V. Lawson, Tarski monoids: Matui's spatial realization theorem, {\em Semigroup Forum} {\bf 95} (2017), 379--404.

\bibitem{Lawson2020} M. V. Lawson. The Booleanization of an inverse semigroup, {\em Semigroup Forum} {\bf 100} (2020), 283--314.

\bibitem{Lawson2021} M. V. Lawson, Finite and semisimple Boolean inverse monoids, arXiv:2102.12931. 

\bibitem{LL} M. V. Lawson, D. H. Lenz, Pseudogroups and their \'etale groupoids, {\em Adv. Math.} {\bf 244} (2013), 117--170.

\bibitem{LMS} M. V. Lawson, S. W. Margolis, B. Steinberg, The \'etale groupoid of an inverse semigroup as a groupoid of filters, {\em J. Aust. Math. Soc.} {\bf 94} (2014), 234--256. 

\bibitem{LS} M. V. Lawson, P. Scott, AF inverse monoids and the structure of countable MV-algebras, {\em J. Pure Appl. Algebra} {\bf 221} (2017), 45--74.

\bibitem{LS2} M. V. Lawson, P. Scott, Characterizations of classes of countable Boolean inverse monoids, arXiv:2204.10033.

\bibitem{LV} M. V. Lawson, A. Vdovina, The universal Boolean inverse semigroup presented by the abstact Cuntz-Krieger relations,
{\em J. Noncommutative Geom.} {\bf 15} (2021), 279--304.

\bibitem{Leech} J.~Leech, Inverse monoids with a natural semilattice ordering, {\em Proc. London Math. Soc.} (3) {\bf 70} (1995), 146--182.

\bibitem{Lenz} D.~H.~Lenz, On an order-based construction of a topological groupoid from an inverse semigroup, \emph{Proc. Edinb. Math. Soc.} {\bf 51} (2008), 387--406.

\bibitem{Malandro} M. E. Malandro, Fast Fourier transforms for finite inverse semigroups, {\em J. Algebra} {\bf 324} (2010), 282--312.

\bibitem{MR} D. Matsnev, P. Resende, Etale groupoids as germ groupoids and their base extensions, {\em Proc. Edinb. math. Soc.} {\bf 53} (2010), 765--785.

\bibitem{Matui12} H.~Matui, Homology and topological full groups of \'etale groupoids on totally disconnected spaces,  {\em Proc. London Math. Soc.} (3) {\bf 104} (2012), 27--56.

\bibitem{Matui13} H.~Matui, Topological full groups of one-sided shifts of finite type, {\em J. Reine Angew. Math} {\bf 705} (2015), 35--84.

\bibitem{Maclane} S. Mac~Lane, {\em Categories for the working mathematician}, Second Edition, Springer, 1998.

\bibitem{NO} P. Nyland, E. Ortega, Topological full groups of ample groupoids with applications to graph algebras, {\em Int. J. Math.} {\bf 30} (2019), 1950018.

\bibitem{Paterson} A.~L.~T.~Paterson, {\em  Groupoids, inverse semigroups, and their operator algebras}, Progress in Mathematics,  {\bf 170}, Birkh\"auser,  Boston, 1998.

\bibitem{Pin} J.-E. Pin, Dual space of a lattice as the completion of a Pervin space, in {\em RAMiCS 2017} (eds P. H\"ofner et al.) LNCS 10226 (2017), 24--40.

\bibitem{Pipp} N. Pippenger, Regular languages and Stone duality, {\em Theory of Computing Systems} {\bf 30} (1997), 121--134.

\bibitem{Renault} J.~Renault, {\em A groupoid approach to $ C^*$-algebras},  Lecture Notes in Mathematics 793, Springer, 1980.


\bibitem{Resende} P.~Resende, Lectures on \'etale groupoids, inverse semigroups and quantales, lecture notes for the GAMAP IP Meeting, Antwerp, 4-18 September, 2006, 115pp
\url{https://www.math.tecnico.ulisboa.pt/~pmr/poci55958/gncg51gamap-version2.pdf}.

\bibitem{Resende3} P.~Resende, A note on infinitely distributive inverse semigroups, {\em Semigroup Forum} {\bf 73} (2006), 156--158.

\bibitem{Resende2} P.~Resende, Etale groupoids and their quantales, {\em Adv. Math.} {\bf 208} (2007), 147--209.

\bibitem{Schein} B. Schein, Completions, translational hulls, and ideal extensions of inverse semigroups,
{\em Czechoslovak Math. J.} {\bf 23} (1973), 575--610.

\bibitem{Sikorski} R. Sikorski, {\em Boolean algebras}, Third Edition, Springer-Verlag, 1969.

\bibitem{Solomon} L. Solomon, Representations of rook monoids, {\em J. Algebra} {\bf 256} (2002), 309--342. 

\bibitem{Simmons} G. F. Simmons, {\em Introduction to topology and analysis}, McGraw-Hill Kogakusha, Ltd, 1963.

\bibitem{Sims} A. Sims, Hausdorff \'etale groupoids and their $C^{\ast}$-algebra, arXiv:1710.10897.

\bibitem{Stone1936} M. H. Stone, The theory of representations for Boolean algebras, {\em Trans. Amer. Math. Soc.} {\bf 40} (1936), 37--111.

\bibitem{Stone1937} M. H. Stone, Applications of the theory of Boolean rings to general topology, {\em Trans. Amer. Math. Soc.} {\bf 41} (1937), 375--481.

\bibitem{Stone1937b} M. H. Stone, Topological representations of distributive lattices and Brouwerian logics,
{\em \v{C}asopis Pro Pestovani Matematiky a Fysiky} {\bf 67} (1937), 1--25. 

\bibitem{Vickers} S Vickers, {\em Topology via logic}, CUP, 1989.

\bibitem{W} F.~Wehrung, {\em Refinement monoids, equidivisibility types, and Boolean inverse monoids}, Lecture Notes in Mathematics 2188, Springer, 2017.

\bibitem{Willard} S. Willard, {\em General topology}, Dover, 2004.

\end{thebibliography}
\end{document}